%% file: main.tex
\documentclass{article}

\usepackage{imports/arxiv}
\input{imports/common_packages.tex}
\input{imports/caam_501_commands}

\input{imports/common_commands.tex}

\input{imports/common_colors.tex}

\usepackage{imports/fonts}
\usepackage{algorithm}
\usepackage{algpseudocode}
\usepackage{comment}
\input{imports/matti_commands}

\newcommand{\extent}{\operatornamewithlimits{ext}}

\newcommand{\mps}{\widehat{\cE\cS}}
\newcommand{\ps}{\cE\cS}

\makeatletter
\newtheorem*{rep@theorem}{\rep@title}
\newcommand{\newreptheorem}[2]{%
\newenvironment{rep#1}[1]{%
 \def\rep@title{#2 \ref{##1}}%
 \begin{rep@theorem}}%
 {\end{rep@theorem}}}
\makeatother
\newreptheorem{theorem}{Theorem}

\title{Maximal Equidistant Spacings}
\author{Michael Puthawala}
\date{\today}

\begin{document}

\maketitle

\begin{abstract}
    We call a family $\set{Y_1,\dots,Y_I}$ in Euclidean space an \emph{equidistance spacing} if $\norm{y_i - y_j} = 1$ whenever $y_i \in Y_i, y_j \in Y_j$ and $i \neq j$. In other words, choosing a representative from each set produces a complete distance graph (i.e. equilateral set). We say such a spacing is maximal if each $Y_i$ is maximal under inclusion. In this work we characterize maximal equidistant spacings in $\R^n$. For each equidistant spacing there is an associated \emph{center} (a point in $\R^n$) and \emph{radius} (a non-negative scalar) so that the centers form an orthocentric system. Using arguments from classical geometry we find that the moduli space of maximal equidistant spacings is described in terms of the movement of its center and radii. Using tools from geometric combinatorics, we develop a discrete combinatorial object called the \emph{signature}. Our classification theorem shows that maximal equidistant spacings are isometric if and only if they have the same signature. We also construct all maximal equidistant spacings in $\R^1,\R^2$ and $\R^3$, outline a procedure for constructing all maximal equidistant spacings in $\R^n$, and give an algorithm for checking if a locus of points is equidistantly spaced that is linear in the number of points, an improvement over the naive direct quadratic algorithm.

\end{abstract}

\section{Introduction}

A disjoint union of subsets $\cY = \sqcup_{i = 1}^IY_i$ of Euclidean space is  \emph{spaced equidistantly} if $\norm{y_i - y_j} = 1$ whenever $i \neq j$ for $y_i \in Y_i$ and $y_j \in Y_j$. In other words, choosing one point from each set produces a complete unit distance graph. The configuration of these points is closely related to problems in geometric combinatorics such as the study of equilateral sets/points \cite{brass2005research,danzer1962zwei,erdos1946sets} but, to our knowledge, equidistant spacings have not been studied before. In this work we characterize maximal (under subsetting) equidistant spacings in $\R^n$.

It turns out that for each $i = 1,\dots,I$ there is a unique point $c_i$ in the affine hull of ${Y_i}$ which we call the \emph{center} of $Y_i$ that is equidistant for each point in $Y_i$. The idea that drives the development of this manuscript is that many properties of equidistant spacings may be rephrased as statements about their centers and vice versa. 

By using arguments from classical geometry we find that the centers of maximal equidistant spacings either all coincide or else form an orthonormal system. That is, are the vertices of an orthocentric simplex plus its orthocenter. This has the surprising implication: the classes of a maximal equidistant spacings are \emph{not} geometrically homogeneous. Just as every (non-degenerate) orthocentric system has a unique vertex that lies inside the simplex generated by the others, one center in equidistant spacings has a center that lies inside the simplex generated by the others. By exploiting this inhomogeneity we develop \emph{inner-outer squash and stretches}, generators of non-trivial isometries of equidistant spacings. Using them, we identify a normal form for maximal equidistant spacings. The normal forms are themselves discrete objects, and so may be classified using combinatorial techniques.

To understand the combinatorics of normal forms, we introduce the space of \emph{signatures}, which are tuples of natural numbers with some ordering constraints. The point is that we may easily find all signatures that correspond to maximal equidistant spacings. Moreover, these signatures are simple to work with, which enables many follow-up results about equidistant spacings. For example, we can count the number of maximal equidistant spacings with $I$ classes in dimension $\R^n$ for any $I,n \in \bbN$. As an example, in $\R^n$ there are $\sum_{i = 0}^np(i)$ distinct maximal equidistant spacings in $\R^n$ where $p$ is the partition function from number theory. %

Our main technical result is Theorem \ref{thm:maximal-equidistant-spacings-are-isometric-iff-same-signature}, a classification theorem that relates isometry classes of maximal equidistant spacings (denoted by $\mps(I,n) / \cong$) with $I$ classes in $\R^n$ to a corresponding space of signatures (denoted by $\hat \cS(I,n)$).

\begin{reptheorem}{thm:sig-is-a-bijection-on-maximal-equidistant-spacings}
    Let $\widehat\ps_{=}(I,n)$ (resp. $\widehat\ps_{\neq}(I,n)$) denote maximal equidistant spacings with $I$ classes in $\R^n$ whose centers coincide (resp. do not coincide), $\cong$ denote isometries of equidistant spacings, $\hat \cS_{=}$ and $\hat \cS_{\neq}$ denote the subset of valid signatures given in Def. \ref{def:valid-signatures}, and $\sig$ denote the map that associates equidistant spacings with signatures. Then all of the following maps are bijections.
    \begin{align*}
        \sig\colon \frac{\mps_=(I,n)}{\cong} &\to \hat\cS_=(I,n),\\
        \sig\colon \frac{\mps_{\neq}(I,n)}{\cong} &\to \hat\cS_{\neq}(I,n),\\
        \sig\colon \bigcup_{I,n \in \bbN}\frac{\mps_=(I,n)}{\cong} &\to \bigcup_{I,n \in \bbN}\hat\cS_=(I,n),\\
        \sig\colon\bigcup_{I,n \in \bbN} \frac{\mps_{\neq}(I,n)}{\cong} &\to \bigcup_{I,n \in \bbN}\hat\cS_{\neq}(I,n).
    \end{align*}
\end{reptheorem}

\subsection{Comparison to Prior Work}
Equidistant spacing is related to other classical problems in geometric combinatorics. For example, \cite{danzer1962zwei,petty1971equilateral} characterized equilateral sets which are equidistant spacings where each set is a single element. 

 There are many coloring problems in geometric combinatorics that consider a geometric graph (where edges are determined by a geometric rule) and ask what the minimal number of colors is so that the graph may be colored without coloring two nodes that share an edge the same color. See e.g. \cite{soifer2008mathematical} and \cite[Chapt. 5]{brass2005research}. One example is the famous Hadwiger–Nelson problem \cite{hadwiger1955altes} where the vertices of the graph are the points on the plane, and nodes are connected if they are unit distance. Equidistant spacings are related these coloring problems through a sort of duality relationship. In coloring problems the nodes of the graph are fixed, and the goal is to minimize the number of colors needed. In the equidistant spacing setting, the number of colors is fixed and the number of points is maximized.

The geometry of equidistant spacings is also related to the applied field of contrastive learning in machine learning and statistics \cite[Sec. 6.3.5]{bishop2023deep}. In this setting, one has a collection of data $\cX = \set{X_i}_{i = I}$, each of which is associated with some label $i$. The task is to find a function $f$ that embeds\footnote{The word embed in the context of machine learning differs from its usual meaning in topology (i.e. a homeomorphism on its range). In the context of machine learning, an embedding is a reparameterization into a new coordinate system that is more `natural.' These embeddings are typically not required to be injective.} the $X_i$'s so that $d(f(x_i), f(x_i'))$ is small if $x_i,x_i'\in X_i$, and so that $d(f(x_i), f(x_j))$ is large for $x_i \in X_i$, $x_j \in X_j$ where $i \neq j$. Intuitively, $f$ embeds the data into a different space where semantically similar data is close together and semantically different data is far apart. If we suppose that $f(x_i)$ and $f(x_j)$ are the same distance apart whenever $i \neq j$, then $\cY \coloneqq \sqcup_{i = 1}^If(X_i)$ is equidistantly spaced, after rescaling. In other words, the range of a good contrastive embedding is nearly an equidistant spacing.%

An interesting take-away from the work in this manuscript is that there is a `local-global product structure' that describes when points are equidistantly spaced. A locus of points is equidistantly spaced if some local conditions hold (i.e. conditions on each $Y_i$), as well as some global conditions hold, which relate how $Y_i$ and $Y_j$ are positioned with respect to each other. These conditions may be checked independently of each other. If $\cY$ is obtained so that each $Y_i$ satisfies the correct local conditions, then each $Y_i$ may always be positioned in space to satisfy the correct global conditions and vice versa. This observation ultimately leads to an algorithm, Algorithm \ref{alg:verification-of-equidistant-spacing}, that can verify if a given locus of points is equidistant in linear (in the number of points) time. This is an improvement over the naive algorithm that checks $\norm{y_i - y_j} = 1$ directly.

\subsection{Manuscript Organization}

In Section \ref{sec:equidistant-spacings-center-and-radii} we establish the basic properties of equidistant spacings. We find that each class of a equidistant spacing may be summarized by two quantities, the class' center (a point in $\R^n$) and its radius (a scalar), Definition \ref{def:center-and-radius}. The general strategy for the rest of the manuscript is to understand equidistant spacings by understanding their centers and radii. 
The most important result from this section is Corollary \ref{cor:characterization-of-equidistant-spacings} which says that $\cY$ being a equidistant spacing may be deduced from its centers and radii (plus some trivial orthogonality conditions). In particular, the classes and radii of equidistant spacings satisfy a system of compatibility equations, Equation \ref{eqn:center-compatability-condition}. %

In Section \ref{sec:ortho-proj-lemma-and-conseq} we consider a trick whereby a equidistant spacing may be expanded via a orthogonal projection and reflection, Lemma \ref{lem:ortho-reflection-lemma}. Maximal equidistant spacings are invariant under this operation. By exploiting this invariance, we obtain two remarkable results. The first, Corollary \ref{cor:leave-out-rank-and-max-spacing}, says that the centers of maximal equidistant spacings are either distinct points, or else all coincide. The second, Proposition \ref{prop:centers-are-orthocentric-systems}, says that if the centers of equidistant spacings do not coincide, then any center is the orthocenter of the simplex formed by the complimentary centers. Combined, these two results heavily constrain maximal equidistant spacings.

In Section \ref{sec:gluing-equidistant-spacings-together} we consider when multiple equidistant spacings may be `glued' together in a natural way to form a larger equidistant spacing. We find that we may associate a scaler, called the extent, to each equidistant spacing, Definition \ref{def:extent}, and that two equidistant spacings glue together if and only if their extents satisfy an inequality, Proposition \ref{prop:extent-characterizes-gluablity}. Moreover, we may easily compute the extent of such gluings, Proposition \ref{prop:extent-of-glueable-equidistant-spacings}.

In Section \ref{sec:isometric-embeddings} we consider what the natural isometries are on equidistant spacings. We find that there are three natural isometries. Two (Euclidean isometries and recolorings) are trivial. There is, however, one other isometry of equidistant spacings which we call a squash and stretch, Definition \ref{def:squash-stretch}. These isometries are less trivial to describe and characterize, and typically do not extend to an isometry of the ambient space.%

In Section \ref{sec:canonical-form} we characterize the squash and stretch isometries of maximal equidistant spacings. We do this by considering a simple squash and stretch called an outer-inner squash and stretch, Proposition \ref{prop:outer-inner-squash-stretch}. Then, by deriving a canonical form, Definition \ref{def:equilateral-normal-form}, we show that squash and stretches are generated by outer-inner squash and stretches for maximal equidistant spacings, Theorem \ref{thm:maximal-equidistant-spacings-have-equilateral-normal-forms} and Proposition \ref{prop:invariance-of-equilateral-normal-form}. Importantly, the canonical forms is concrete and may be easily constructed.

In Section \ref{sec:signatures-of-equidistant-spacings} we associate to each maximal equidistant spacing an object called its signature, Definition \ref{def:signagure-of-equidistant-spacing}. We then show that maximal equidistant spacings are isometric if and only if they have the same signature, Theorem \ref{thm:maximal-equidistant-spacings-are-isometric-iff-same-signature}. We then compute all signatures in $\R^0$ $\R^1,\R^2$ and $\R^3$, and thereby compute all maximal equidistant spacings in those spaces. We conclude by counting the number of maximal equidistant spacings in $\R^n$ and find it has a simple expression, Proposition \ref{prop:counting-maximal-equidistant-spacings}.

In Section \ref{sec:algorithm-for-testing-equidistant-spacings} we give an algorithm, Algorithm \ref{alg:verification-of-equidistant-spacing}, to verify if a finite collection of points form a equidistant spacing. By analyzing the run time of this algorithm, Theorem \ref{thm:correctness-alg-1}, we find that its runtime grows linearly with the number of points.

\section{Results}

\subsection{Problem Setup and Basic Examples}

We begin by defining our central object of study.

\begin{definition}[Equidistantly Spaced, $\ps$]
    \label{def:equidistantly-spaced}
    Let $n,I \in \bbN$ and $I \geq 2$. A disjoint union $\cY =  \sqcup_{i = 1}^I Y_i$ where $Y_i \subset \R^n$ is called \emph{equidistantly spaced} if for each $i\neq j$, and $y_i \in Y_i, y_j \in Y_j$
    \begin{align}
        \label{eqn:equidistant-spacing-equality}
        \norm{y_i - y_j} = 1.
    \end{align}

    If $I = 1$, we say that $\cY$ is equidistantly spaced if there is a point $c_0 \in \R^n$ such that $\norm{y - c}\leq 1$ is constant for all $y \in \cY$.
    
    We say that such a $\cY$ is a \emph{equidistant spacing}.
    An $i \in 1,\dots,I$ is called a \emph{class} of $\cY$. We denote $\ps(I,\R^n)$ as the set of all equidistant spacings with $I$ classes in $\R^n$. We denote $\ps(I,\R^n)$ by $\ps$, when $I$ and $n$ are clear from context.
\end{definition}

In this work, we study how the set $\ps$ looks when the norm in Eqn. \ref{eqn:equidistant-spacing-equality} is the Euclidean norm. It is left to future work to consider the case for other geometries.

Before proceeding, we give some examples of equidistant spacings.

\begin{figure}
    \centering
    \begin{subfigure}{.32\linewidth}
        \centering
        \includegraphics[width=1.\linewidth]{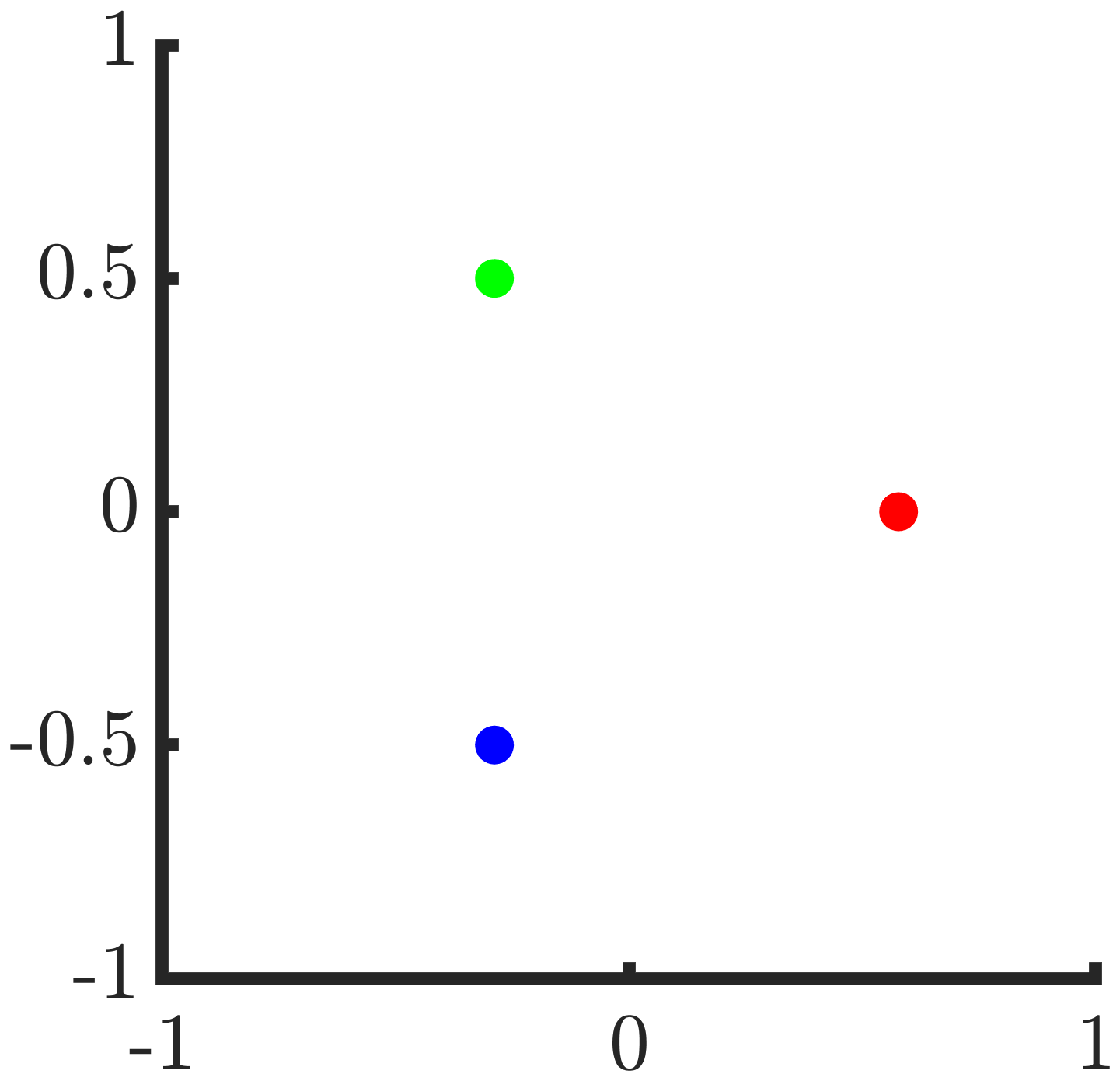}
        \subcaption{$\cY^{(1)}$}
        \label{fig:eps-first-examp:1}
    \end{subfigure}
    \begin{subfigure}{.32\linewidth}
        \centering
        \includegraphics[width=1.\linewidth]{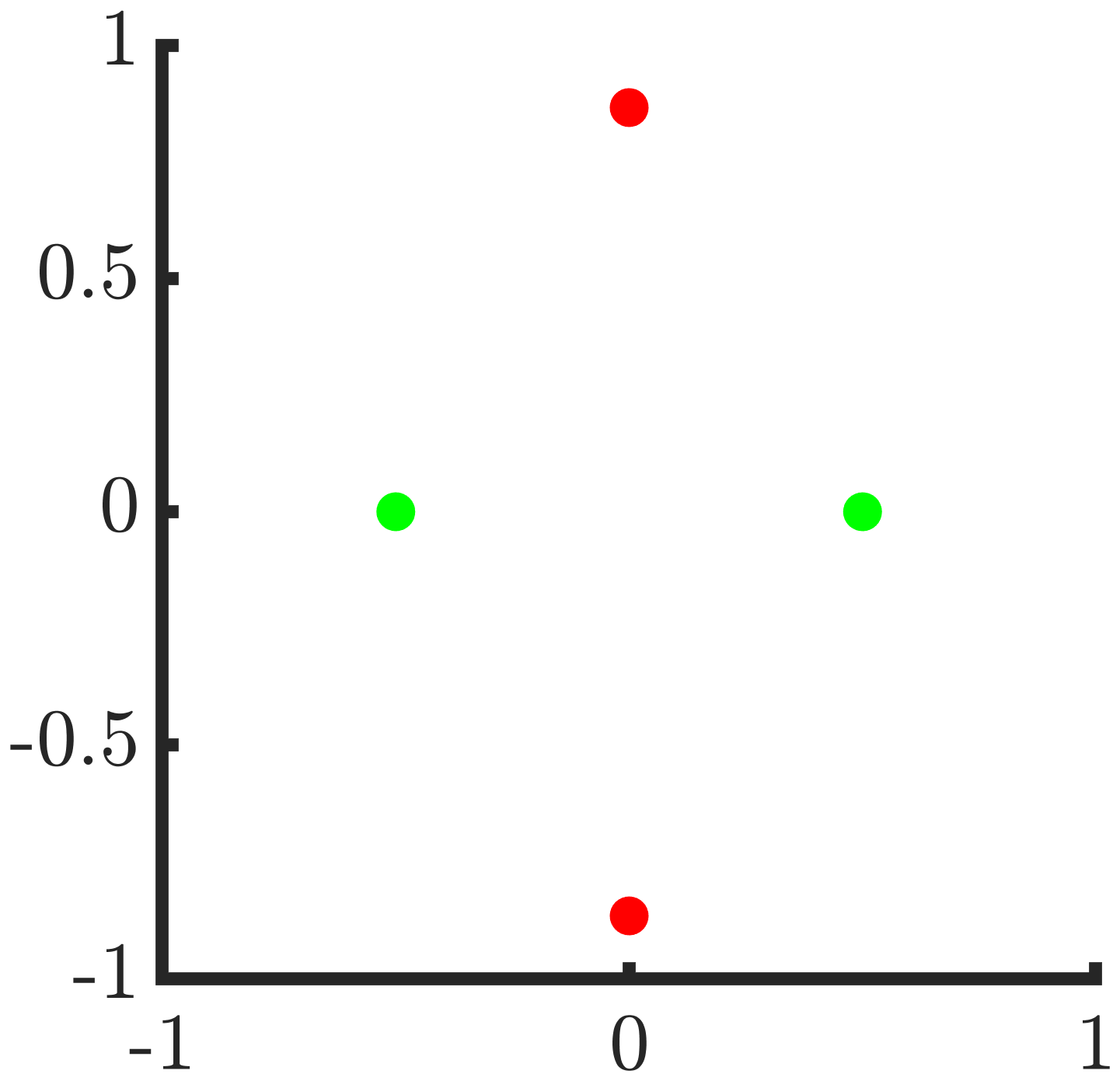}
        \subcaption{$\cY^{(2)}$}
        \label{fig:eps-first-examp:2}
    \end{subfigure}
    \begin{subfigure}{.32\linewidth}
        \centering
        \includegraphics[width=1.\linewidth]{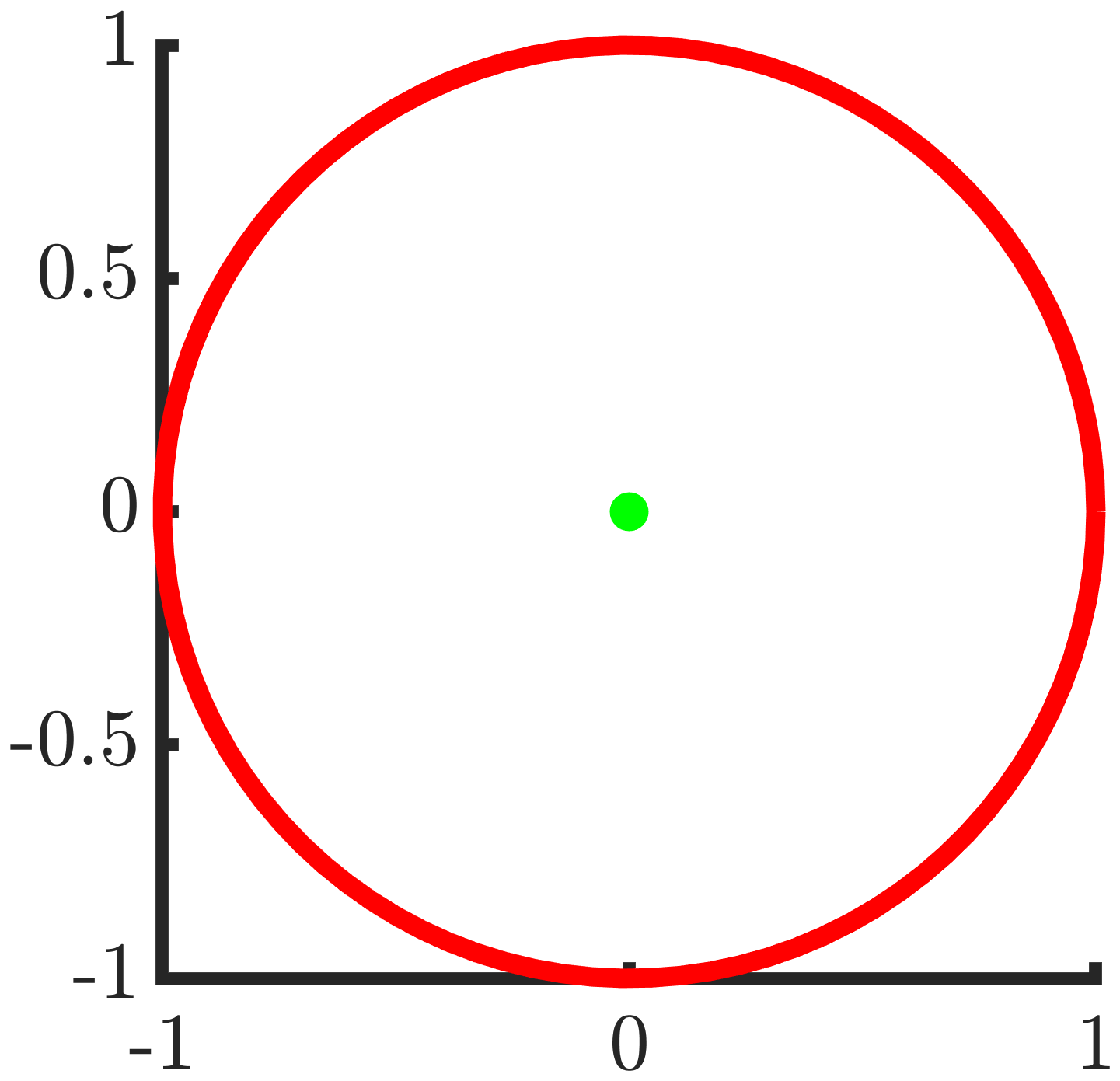}
        \subcaption{$\cY^{(3)}$}
        \label{fig:eps-first-examp:3}
    \end{subfigure}
    \caption{An example of three different equidistant spacings in $\R^2$ with different values for $I$. Details given in Example \ref{exp:perf-spacings}. In each figure, the class label is indicated by color. For example, the two red points in Figure \ref{fig:eps-first-examp:2} are $Y_1$ for that equidistant spacing. Notice how there is no requirement for the spacing of the points within one class.}
    \label{fig:eps-first-example}
\end{figure}
\begin{example}
    \label{exp:perf-spacings}
    Figure \ref{fig:eps-first-example} illustrates three different equidistant spacings, $\cY^{(1)}\in \ps(3,\R^2)$ and $\cY^{(2)},\cY^{(3)}  \in \ps(2,\R^2)$. In each example, class labels are indicated with color. That is, points of the same color have the same class. 
    
    The first example is $\cY^{(1)}$, illustrated in Figure \ref{fig:eps-first-examp:1}. It shows three points that lie on an equilateral triangle with side-length 1. The analogous construction, each vertex of an equilateral simplex is a class, works in any dimension.

    The second example is $\cY^{(2)}$, illustrated in Figure \ref{fig:eps-first-examp:2}. It shows two classes each of which are arranged as four points on a rhombus with side-length 1. This construction generalizes to higher dimensions in the obvious way.%

    The final example we show is $\cY^{(3)}$, illustrated in Figure \ref{fig:eps-first-examp:3}. This example has one class as the boundary of $S^1_1(0)$ and a one point class at the origin. In this example, a continuum of points belongs to the red class, and one point belongs to the green class. 
\end{example}

Consider the following operations that transform one equidistant spacing into another.
\begin{figure}
    \centering
    \begin{subfigure}{.24\linewidth}
        \centering
        \includegraphics[width=1.\linewidth]{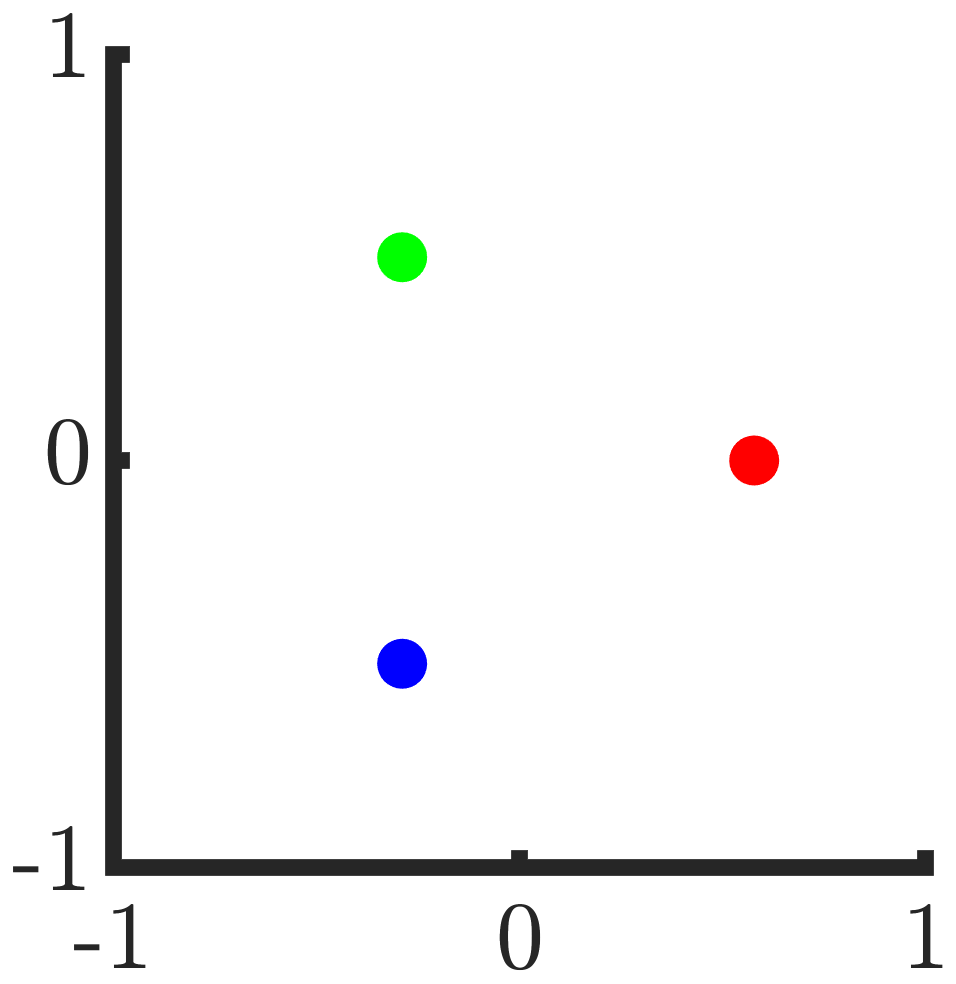}
        \subcaption{$\cY^{(1)}$}
        \label{fig:iso-unitary:1}
    \end{subfigure}
    \begin{subfigure}{.24\linewidth}
        \centering
        \includegraphics[width=1.\linewidth]{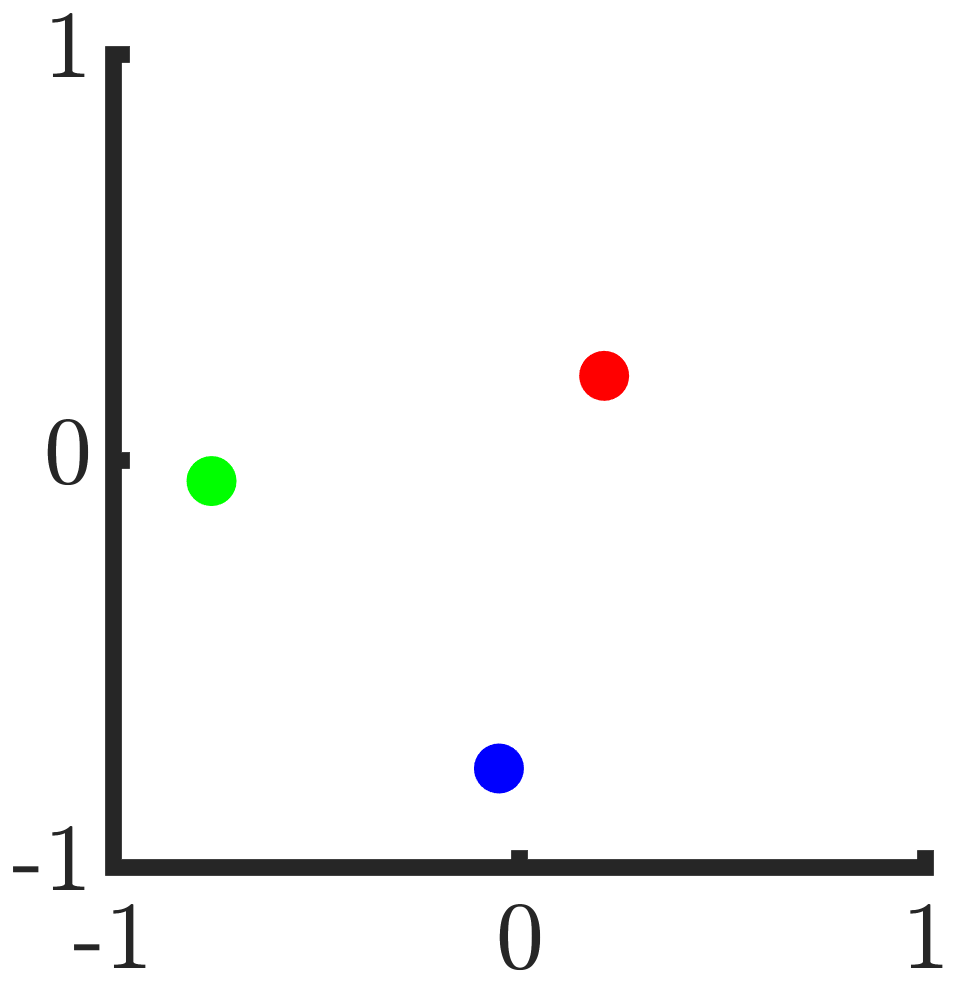}
        \subcaption{$\cY^{(2)}$}
        \label{fig:iso-unitary:2}
    \end{subfigure}
    \begin{subfigure}{.24\linewidth}
        \centering
        \includegraphics[width=1.\linewidth]{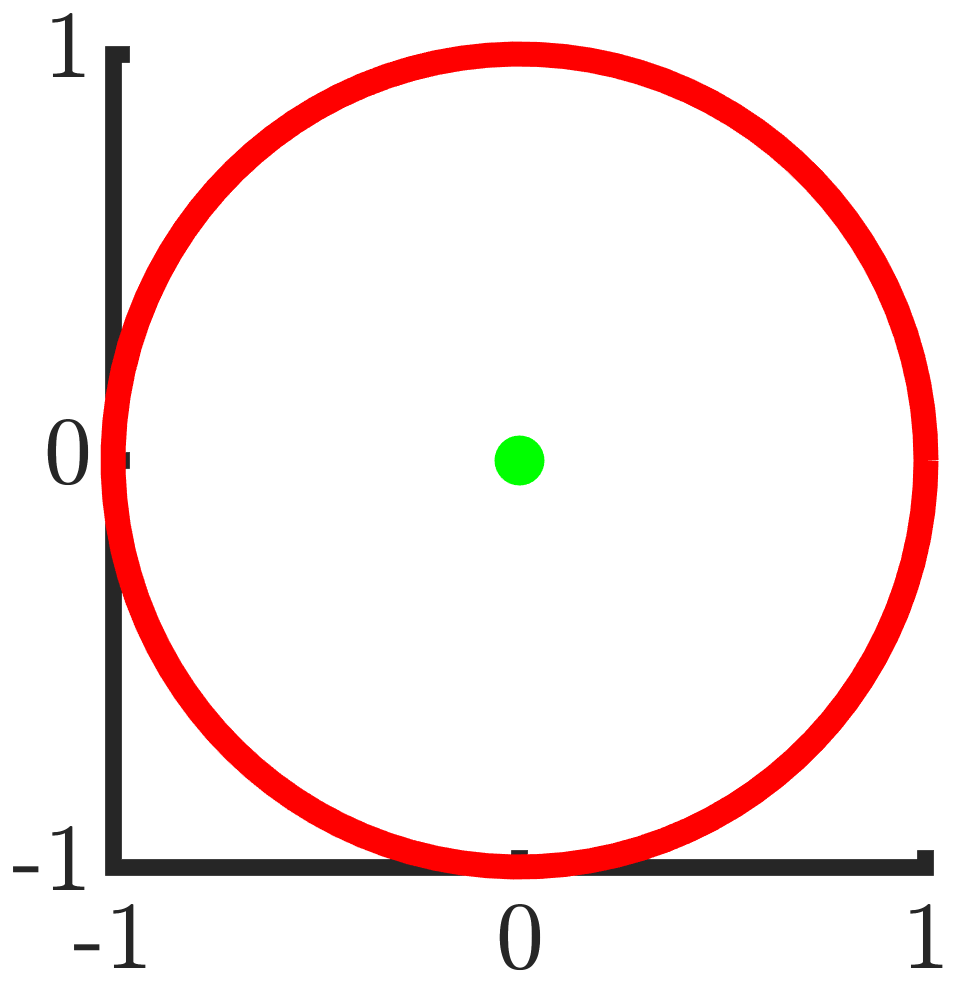}
        \subcaption{$\cY^{(3)}$}
        \label{fig:class-permut:1}
    \end{subfigure}
    \begin{subfigure}{.24\linewidth}
        \centering
        \includegraphics[width=1.\linewidth]{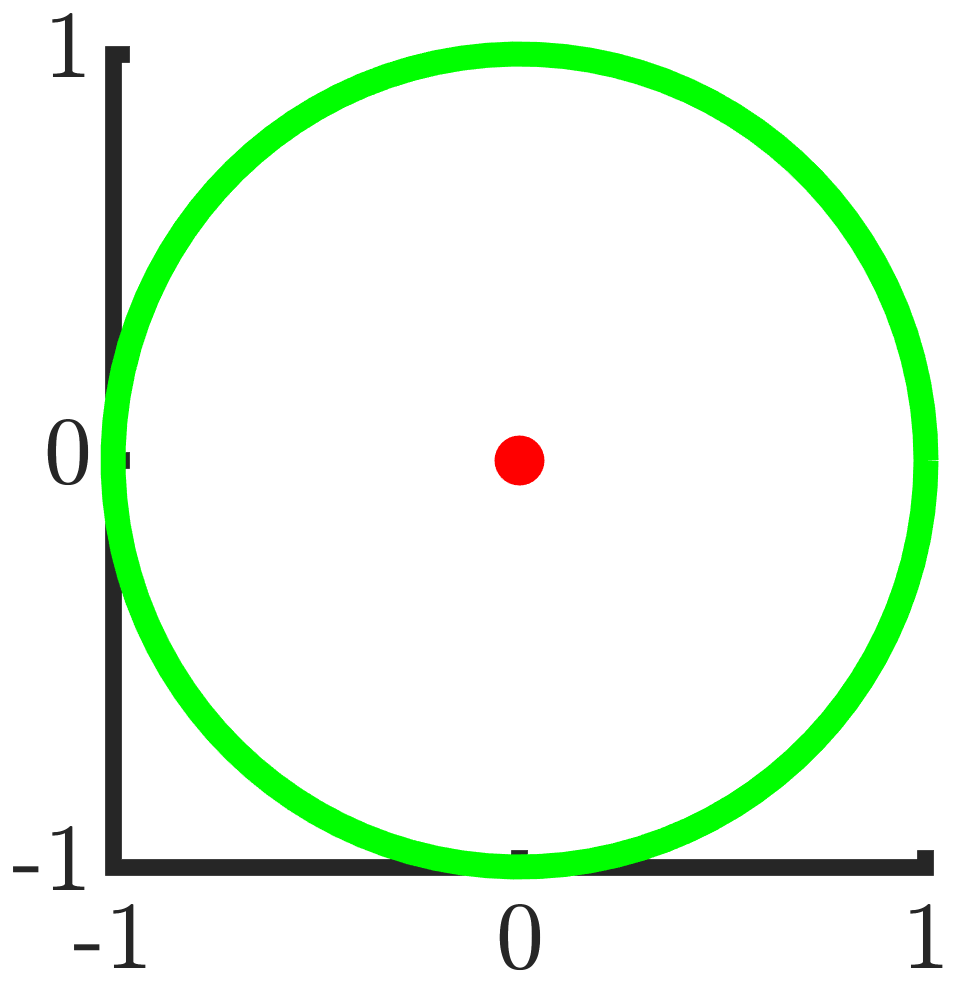}
        \subcaption{$\cY^{(4)}$}
        \label{fig:class-permut:2}
    \end{subfigure}
    \\
    \begin{subfigure}{.24\linewidth}
        \centering
        \includegraphics[width=1.\linewidth]{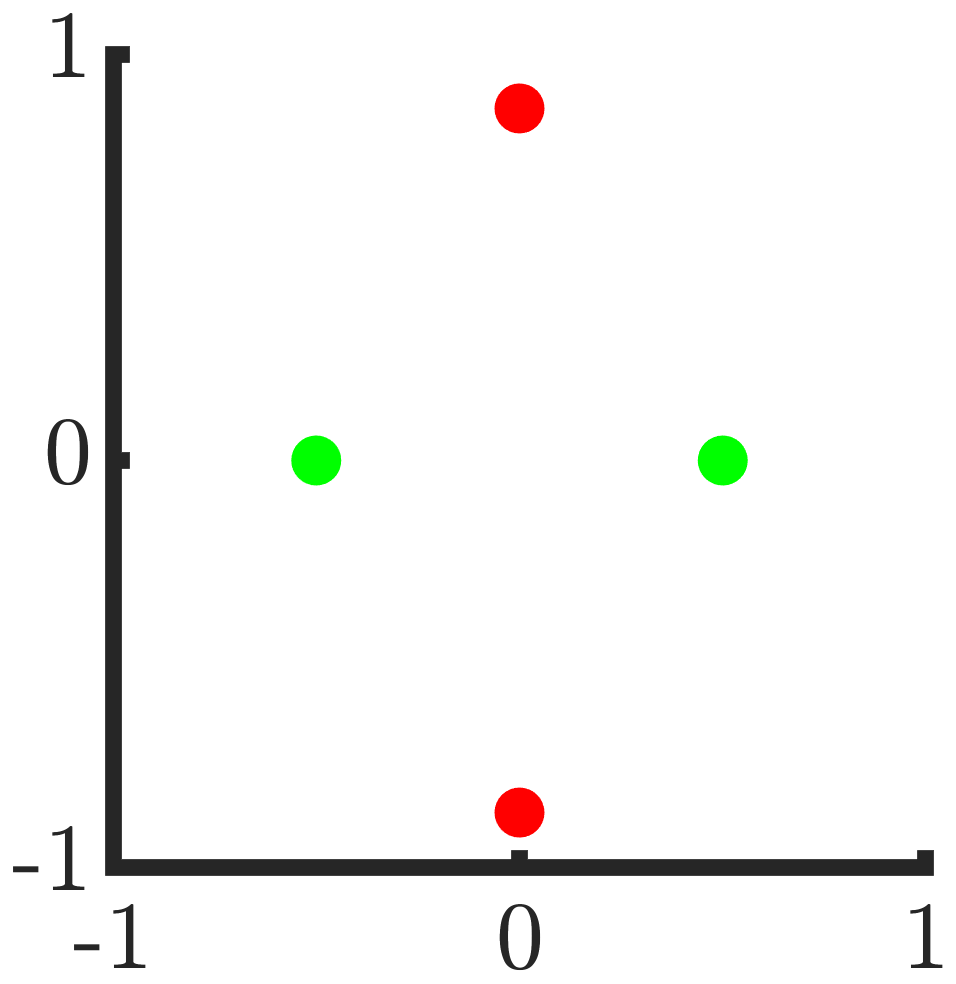}
        \subcaption{$\cY^{(5)}$}
        \label{fig:single-face:1}
    \end{subfigure}
    \begin{subfigure}{.24\linewidth}
        \centering
        \includegraphics[width=1.\linewidth]{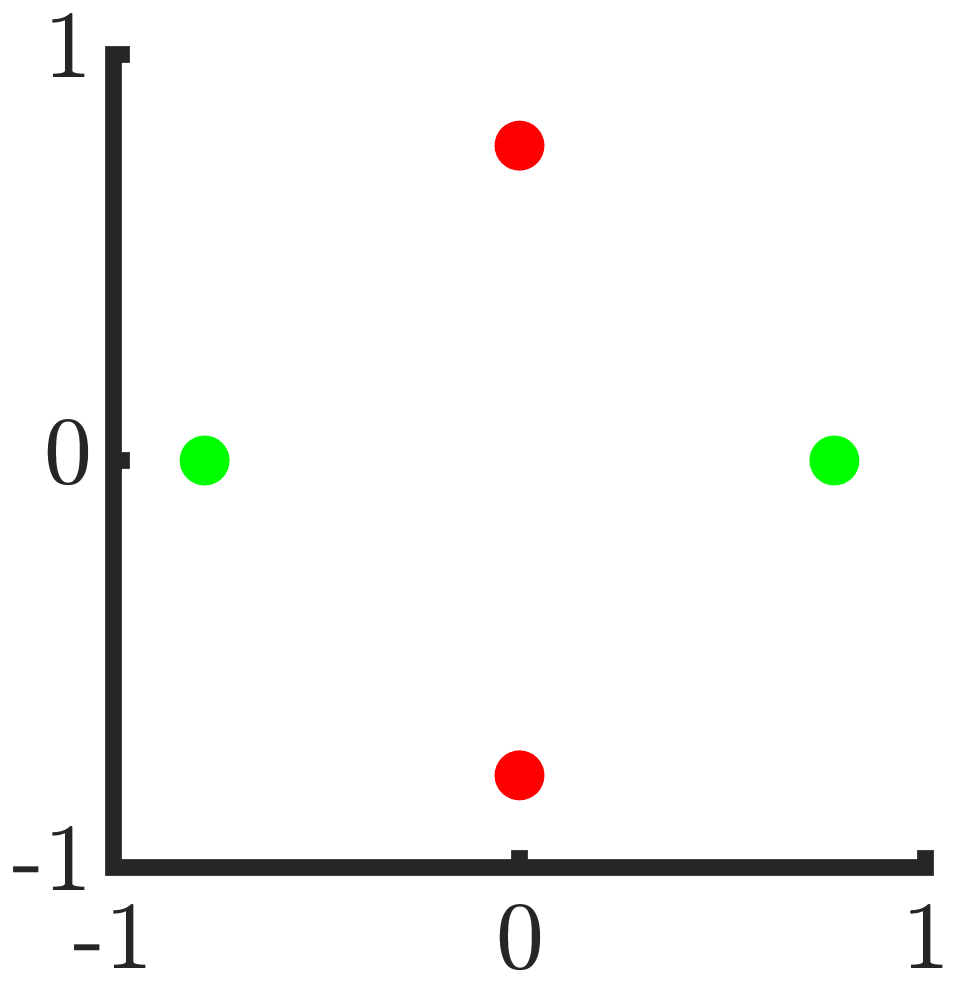}
        \subcaption{$\cY^{(6)}$}
        \label{fig:single-face:2}
    \end{subfigure}
    \caption{An example showing several ways in which a equidistant spacing produces other equidistant spacings in uninteresting ways. This figure shows eight equidistant spacings total. The figures are made up of four pairs. In the first pair, $\cY^{(1)}$ and $\cY^{(2)}$, are related by rototranslation. That is $Y^{(1)}_i = pY^{(2)}_i$ where $p \in E(2)$. The second pair, $\cY^{(3)}$ and $\cY^{(4)}$, are related by permuting class labels. That is, $Y_i^{(3)} = Y^{(4)}_{\sigma(i)}$ for some permutation $\sigma$ on $\set{1,\dots,I}$. The third pair, $\cY^{(5)}$ and $\cY^{(6)}$, are related by what we call a squash and stretch, so named because the green points are `squashed' together, and the red points `streched' out. Intuitively points of one class are `squished' together while points of another class are `stretched' apart to preserve equidistant spacing.}
    \label{fig:perserving-equidistant-spacings}
\end{figure}

\begin{example}[Operations that Preserve Equidistant Spacings]
    \label{exp:operations-that-preserve-equidistant-spacings}
    Figure \ref{fig:perserving-equidistant-spacings} showcases several ways in which equidistant spacings may be transformed into each other. The figure shows four pairs of equidistant spacings. The first three pairs represent isometries on equidistant spacings (further developed in Section \ref{sec:isometric-embeddings}). The final pair illustrates the notion of maximality.

    The first pair, $\cY^{(1)}$ (Figure \ref{fig:iso-unitary:1}) and $\cY^{(2)}$ (Figure \ref{fig:iso-unitary:2}), differ only by a rototranslation, an element of the Euclidean group on $\R^n$, denotes $E(n)$, where $n = 2$. Such an operation is an isometry on the ambient space, and so preserves equidistant spacings.%

    The second pair, $\cY^{(3)}$ (Figure \ref{fig:class-permut:1}) and $\cY^{(4)}$ (Figure \ref{fig:class-permut:2}), differ by permuting the class labels (i.e., the colors of the points). Such an operation preserves equidistant spacing and generalizes to an arbitrary number of points in the obvious way. %

    The third pair, $\cY^{(5)}$ (Figure \ref{fig:single-face:1}) and $\cY^{(6)}$ (Figure \ref{fig:single-face:2}), differ by a squish and stretch, which is defined in Definition \ref{def:squash-stretch}.  For this example, we `stretch' the green classes by moving the green points apart, and `squeeze' the red class by moving the red points together. If done properly, the resulting spacing is still equidistant. %
    In higher dimensions an analogous operation exists, but it is nontrivial to characterize. The key is that some dimensions of $\R^n$ are squished and others are stretched in such a way that maintains equidistant spacing.

    In fact, the three isometries of equidistant spacings outlined here (rototranslations, class permutations and squash and stretches) characterize all isometries in all spaces for maximal equidistant spacings. This is the content of our main theorem, \ref{thm:maximal-equidistant-spacings-are-isometric-iff-same-signature}. %

\end{example}

\subsection{Equidistant Spacings, Centers, and Radii}
\label{sec:equidistant-spacings-center-and-radii}

We now give two important results with trivial proofs. Both results are major workhorses of all developments that follow.

\begin{remark}[Notational Shorthand: I]
    \label{rmk:notational-shorthand:1}
    For the manuscript, we adopt the following notation. If a $\cY \in \ps$ is mentioned in a theorem, lemma, example, etc. it is assumed that,
    \begin{enumerate}
        \item the variables $i,j$ or $k$ denote class labels of $\cY \in \ps$,
        \item the symbols $Y_i, Y_j$ or $Y_k$ denote classes of $\cY$,
    \end{enumerate}
\end{remark}
Pursuant to Remark \ref{rmk:notational-shorthand:1}, in the statement of Lemma \ref{lem:ortho-theorem} we use the convention that $i$ and $j$ are understood to refer to classes of $\cY$. This implicit understanding will be used for subsequent results as well and helps avoid notational bloat.
\begin{lemma}[Orthogonality Lemma]
    \label{lem:ortho-theorem}
    Let $\cY \in \ps$, $i\neq j$, $y_i,y_i' \in Y_i$, and $y_j,y_j' \in Y_j$. Then $\aff Y_i \perp \aff Y_j$.
\end{lemma}

\begin{proof}
    The claim is equivalent to showing that  $\innerprod{y_i - y_i'}{y_j - y_j'} = 0$. This follows from the following chain of equalities.
    \begin{align}
        2\innerprod{y_i - y_i'}{y_j - y_j'} &= 2\paren{\innerprod{y_i}{y_j} + \innerprod{y_i'}{y_j'} - \innerprod{y_i}{y_j'} - \innerprod{y_i'}{y_j}}\\
        &= \innerprod{y_i}{y_i} -  \innerprod{y_i}{y_i} + \innerprod{y_i'}{y_i'} -  \innerprod{y_i'}{y_i'}\nonumber\\
        &+ \innerprod{y_j}{y_j} -  \innerprod{y_j}{y_j} + \innerprod{y_j'}{y_j'} -  \innerprod{y_j'}{y_j'}\nonumber\\
        &+ 2\paren{\innerprod{y_i}{y_j} + \innerprod{y_i'}{y_j'} - \innerprod{y_i}{y_j'} - \innerprod{y_i'}{y_j}}\\
        &= \paren{-\innerprod{y_i}{y_i} + 2\innerprod{y_i}{y_j} - \innerprod{y_j}{y_j}}\nonumber\\
        &+ \paren{-\innerprod{y_i'}{y_i'} + 2\innerprod{y_i'}{y_j'} - \innerprod{y_j'}{y_j'}}\nonumber\\
        &+ \paren{\innerprod{y_i}{y_i} - 2\innerprod{y_i}{y_j'} + \innerprod{y_j'}{y_j'}}\nonumber\\
        &+ \paren{\innerprod{y_i'}{y_i'} - 2\innerprod{y_i'}{y_j} + \innerprod{y_j}{y_j}}\nonumber\\
        &= -\norm{y_i - y_j}^2 - \norm{y_i' - y_j'}^2\nonumber\\ 
        &+ \norm{y_i - y_j'}^2 + \norm{y_i - y_j'}^2 \\
        &= 2 - 2 = 0
    \end{align}
\end{proof}

\begin{definition}[Minimal Affine Set]
    \label{def:smallest-affine-set}
    For an $X \subset \R^n$ we define $\aff X$ as the smallest affine set containing $X$. For $\cY \in \ps$ we define the set $A_i \coloneqq \aff(Y_i)$. That is, $A_i$ is the minimal affine set that contains $Y_i$. %
\end{definition}

We now introduce a trick called the recoloring trick. This trick follows from the observation that, given a equidistant spacing $\cY$, we may make another by partitioning all classes of $\cY$ and coloring all classes of each partition the same color.

\begin{lemma}[The Recoloring Trick]
    \label{lem:recoloring-trick}
    Let $\cY \in \ps(I,\R^n)$ be a equidistant spacing, and $p_1,\dots,p_J$ a partition of $\set{1,\dots, I}$ with $J \geq 2$. Consider $\cY' = \sqcup_{j=1}^J Y_j'$ defined via
    \begin{align}
        \label{eqn:recoloring-eqn}
        Y_j' \coloneqq \bigcup_{i \in p_j} Y_i.
    \end{align}
    Then $\cY' \in \ps(J,\R^n)$.
\end{lemma}

\begin{proof}
    Consider a $i' \neq j'$, $y'_{i'} \in  Y'_{i'}$ and $y'_{j'} \in Y'_{j'}$. Then by Eqn. \ref{eqn:recoloring-eqn} there is an $i \neq j$ so that $y'_{i'} \in Y_i$ and $y_{j'} \in Y_j$. Because $\cY \in \ps$, we have that $\norm{y'_{i'} - y'_{j'}} = 1$
\end{proof}

\begin{figure}
    \centering
    \begin{subfigure}{.24\linewidth}
        \centering
        \includegraphics[width=1.\linewidth]{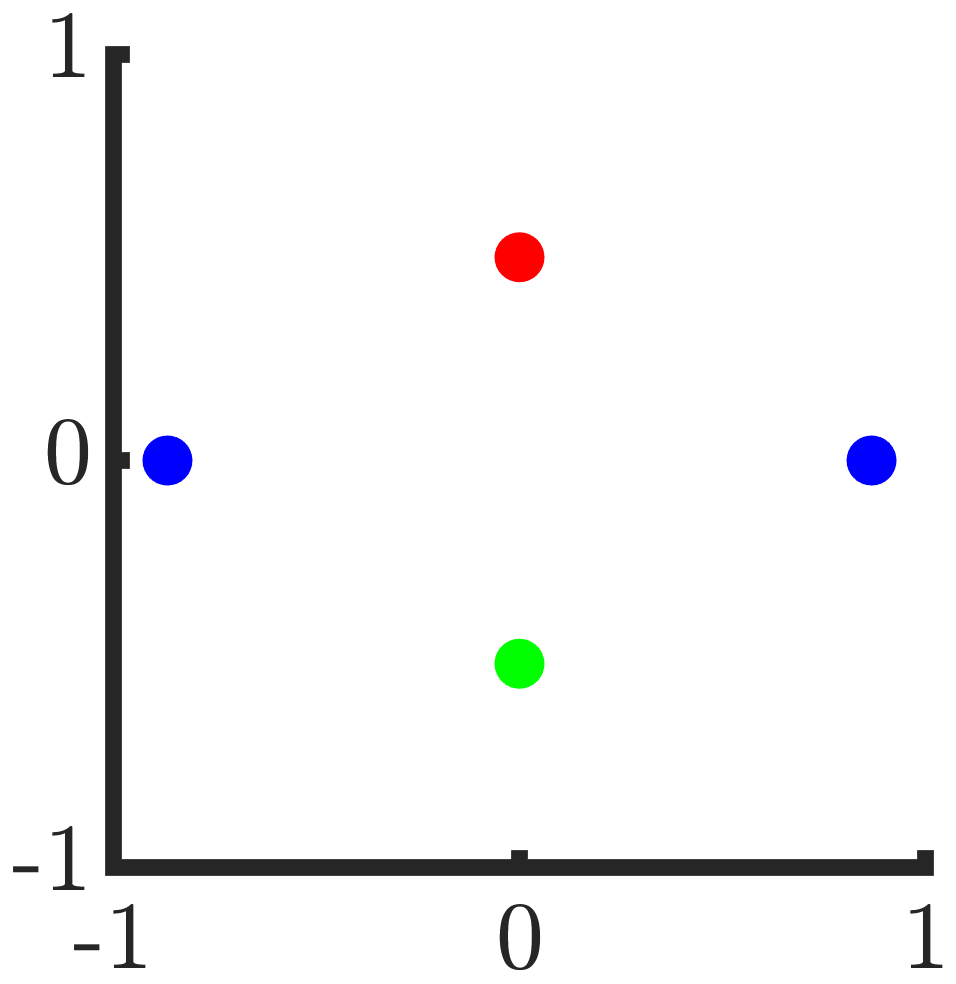}
        \subcaption{$\cY^{(1)}$}
        \label{fig:recoloring:1}
    \end{subfigure}
    \begin{subfigure}{.24\linewidth}
        \centering
        \includegraphics[width=1.\linewidth]{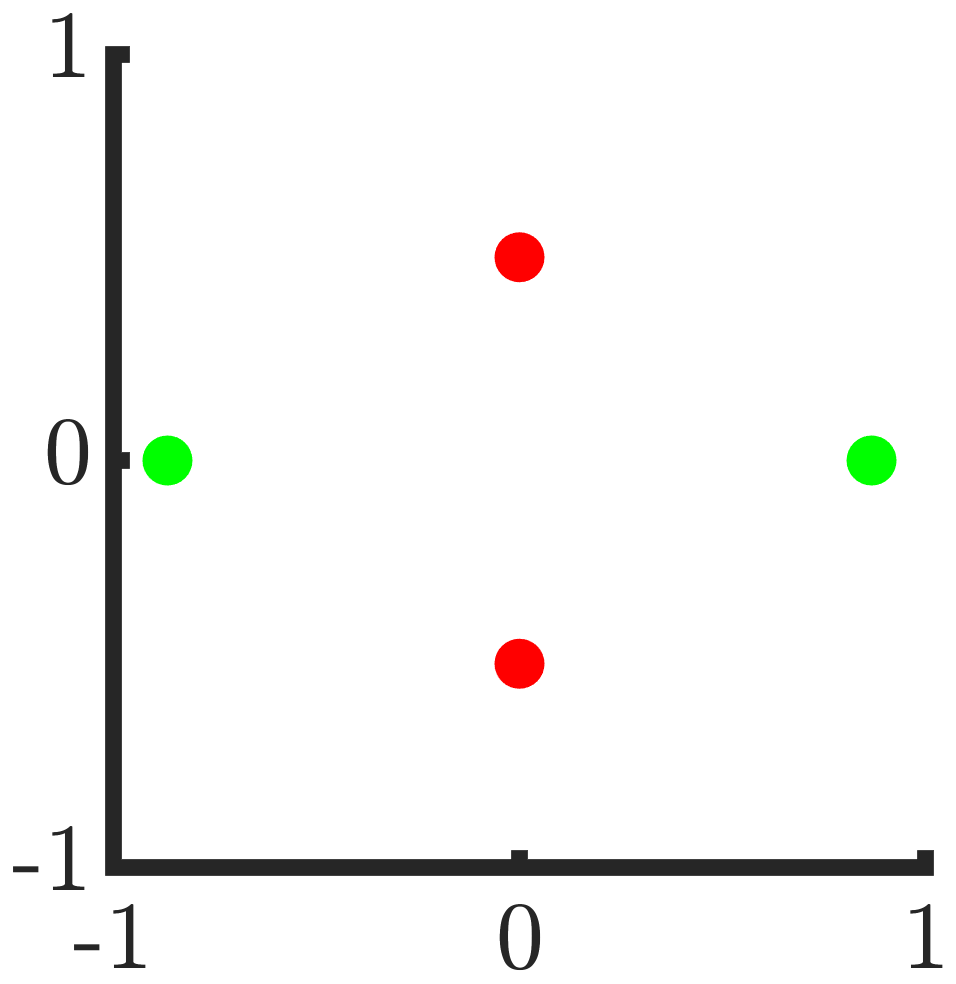}
        \subcaption{$\cY^{(2)}$}
        \label{fig:recoloring:2}
    \end{subfigure}
    \begin{subfigure}{.24\linewidth}
        \centering
        \includegraphics[width=1.\linewidth]{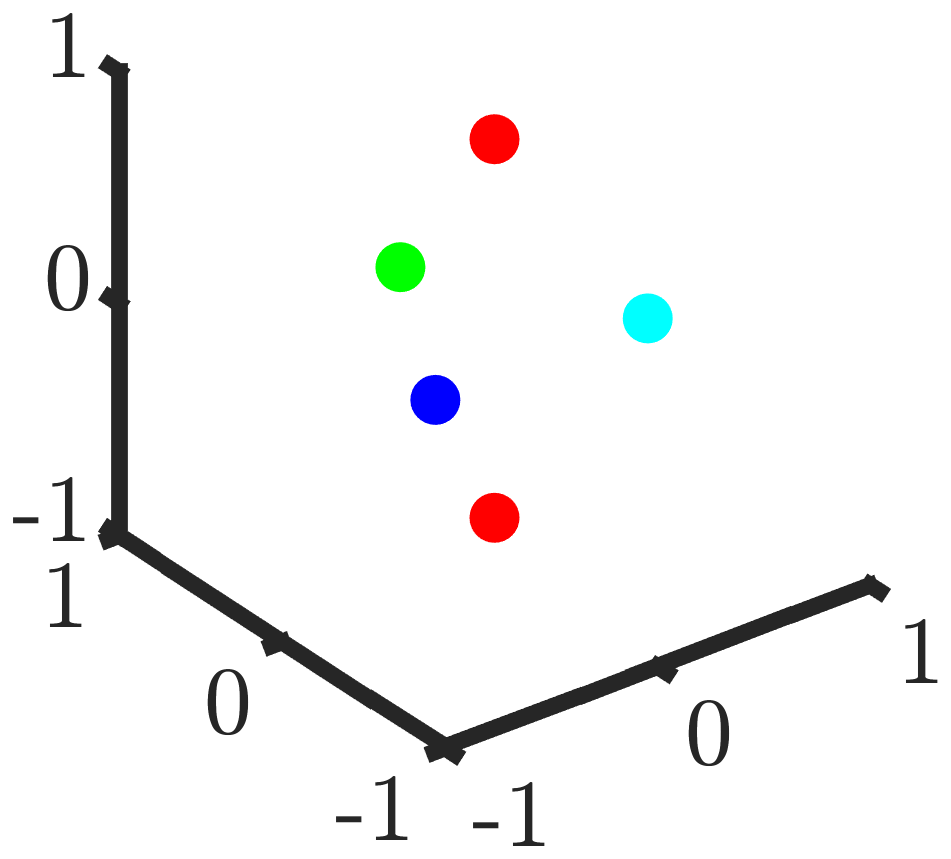}
        \subcaption{$\cY^{(3)}$}
        \label{fig:recoloring:3}
    \end{subfigure}
    \begin{subfigure}{.24\linewidth}
        \centering
        \includegraphics[width=1.\linewidth]{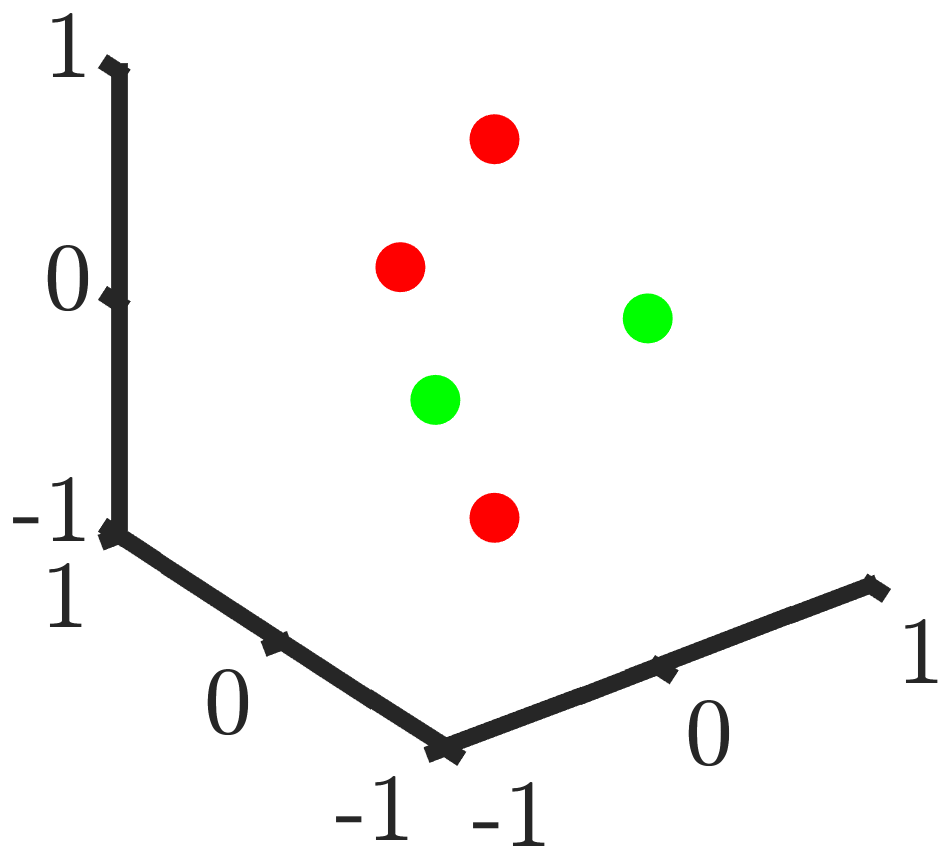}
        \subcaption{$\cY^{(4)}$}
        \label{fig:recoloring:4}
    \end{subfigure}
    \caption{Two examples of the recoloring trick. The first example shows $\cY^{(1)}$, plotted in subfigure \ref{fig:recoloring:1} recolored to $\cY^{(2)}$, plotted in \ref{fig:recoloring:2}. The second shows $\cY^{(3)}$, plotted in subfigure \ref{fig:recoloring:3} recolored to $\cY^{(4)}$, plotted in \ref{fig:recoloring:4}.}
    \label{fig:recoloring-trick}
\end{figure}

\begin{remark}[Necessity for $J \geq 2$ in the Recoloring Trick]
    \label{rmk:nesc-for-j-geq-2-in-the-recoloring-trick}
    The requirement that $J \geq 2$ is necessary, as coloring all points the same color does not always yield a equidistant spacings, see $\cY^{(1)}$ in Figure \ref{fig:recoloring-trick}.
\end{remark}

Two examples of the recoloring trick are given in Figure \ref{fig:recoloring-trick}. The recoloring trick is obvious but has a powerful use case. We may use it to extend results that relate classes $Y_i$ and $Y_j$ to more general results about $\cup_{i \in \cI} Y_i$ and $\cup_{j \in \cJ} Y_j$ for any disjoint $\cI, \cJ \subset \set{1,\dots,I}$. For example, the following corollary of Lemma \ref{lem:ortho-theorem}.

\begin{corollary}[Strong Orthogonality]
    \label{cor:strong-ortho}
    Let $\cY \in \ps(I)$ and $ \cI \subset \set{1,\dots,I}$. Then,
    \begin{align}
    \aff{\cup_{i \in\cI} Y_i} \perp \aff{\cup_{i \in\cI^c} Y_i}.
    \end{align}
\end{corollary}

\begin{proof}
    Apply the recoloring trick to $\cY$ to obtain $\cY'$ where $p_1 \coloneqq \cI$ and $p_2 \coloneqq \cI^c$. Then, the corollary follows from applying Lemma \ref{lem:ortho-theorem} to $\cY'$.
\end{proof}

With the definition of $A_i$, we may present our first theorem, the projection theorem. %

\begin{theorem}[Projection Theorem]
    \label{thm:projection-theorem}
    Let $\cY \in \ps(I,\R^n)$. For a given $i$, define the set $\cA_i \subset \R^{n+1}$ by 
    \begin{align}
        \label{eqn:equispaces-pairs}
        \cA_i \coloneqq \set{(c,r) \colon Y_i \subset S^{n-1}_{r}(c)}.
    \end{align}
    Then the following holds
    \begin{enumerate}
        \item The set $\cA_i$ is non-empty.
        \item Let $r_i \coloneqq \min \pi_{n+1} \cA_i$. Then $r_i \in [0,1]$ and there is a unique $c_i$ so that $(c_i,r_i) \in \cA_i$.
        \item The point $c_i$ satisfies
        \begin{align}
            \label{eqn:center-is-fixed}
            c_i \coloneqq \aff_{A_i} \cup_{j \in \cJ} Y_j
        \end{align}
        where $\cJ \subset \set{1,\dots,I}$ does not contain $i$.
    \end{enumerate}
    
\end{theorem}
\begin{proof}
    Point 1. If $I = 1$, then point 1 follows from definition. If $I > 1$, then $(y_j,1) \in \cA_i$ where $j \neq i$ and $y_j \in Y_j$. 

    Point 2. Let $(y,r) \in \cA_i$ and let $y^* \coloneqq \proj_{A_i}y$. By Pythagorean's theorem 
    \begin{align}
        \paren{y^*,\sqrt{r^2 - \norm{y - y^*}^2}} \in \cA_i.
    \end{align}
    Hence, if $(c_i,r_i)$ is such that $r_i = \min \pi_{n+1} \cA_i$, then $c_i \in A_i$. Because $A_i$ is the minimal affine set that contains $Y_i$, $Y_i$ contains a (non-degenerate) $\dim A_i+1$-simplex of points $\sigma$. The statements $(c_i,r_i) \in \cA_i$ and $c_i \in A_i$ means that $c_i$ is the circumcenter of $\sigma$. Uniqueness of $c_i$ comes from uniqueness of circumcenters.

    Point 3. Corollary \ref{cor:strong-ortho} means that $A_i \perp \aff \cup_{j \in \cI}A_j$. Hence, if $y \in Y_j$ for $j \in \cJ$ then $(\proj_{A_i}y_j , \sqrt{1 - \norm{y - \proj_{A_i}y_j}^2})$. But $\norm{y - \proj_{A_i}y_j}^2 = r_i$, and so $\proj_{A_i}y_j = c_i$ by point 2.

\end{proof}

\begin{definition}[Center and Radius of $Y_i$]
    \label{def:center-and-radius}
    Let $\cY \in \ps$. For each $i$, let $\cA_i$ be given by Eqn. \ref{eqn:equispaces-pairs}. We define the \emph{center} and \emph{radius} of $Y_i$ by $c_i$ and $r_i$ respectively where $r_i \coloneqq \min \pi_{n+1} \cA_i$ and $c_i$ is the unique point so that $(c_i,r_i) \in \cA_i$.
\end{definition}

\begin{remark}[Notational Shorthand: II]
    \label{rmk:notational-shorthand:2}
    We adopt the following notation. If a $\cY \in \ps$ is mentioned in a theorem, lemma, example, etc. and if a $c_i, c_j, c_k, r_i, r_j$ or $r_k$ are mentioned it is understood that $c_i$ (resp. $c_j, c_k$) refers to the center of class $Y_i$ (resp. $Y_j, Y_k$), and likewise $r_i$ (resp. $r_j, r_k$) refers to the radius.
\end{remark}

Next, we present the following technical helper lemma which may be considered a generalization of Lemma \ref{lem:ortho-theorem}.%

\begin{proposition}[Tri-part Orthogonality]
    \label{prop:tri-part-ortho}
    Let $\cY \in \ps$. Then 
    \begin{align}
        \cup_{i = 1}^I A_i \perp \aff\paren{c_1,\dots,c_I}.
    \end{align}
\end{proposition}
\begin{proof}
    For the proof, it suffices to show that $A_i\perp \aff\paren{c_1,\dots,c_I}$ for each $i = 1,\dots,I$. To show this, it suffices to show that $A_i\perp \aff\paren{c_i,c_j}$ for each $j$. Let $y_i \in Y_i$ and $y_j \in Y_j$, then 
    \begin{align}
        \innerprod{y_i - c_i}{c_i - c_j} &= \innerprod{y_i - c_i}{c_i - y_j + y_j - c_j}\\
            &= \langle\underbrace{y_i - c_i}_{\in A_i}\underbrace{c_j - y_j}_{\in A_j}\rangle + \innerprod{y_i - c_i}{y_j - c_i}\\
            &= 0
    \end{align}

    where $\innerprod{y_i - c_i}{y_j - c_i} = 0$, because $c_i$ is the orthogonal projection of $y_j$ onto $A_i$ by Theorem \ref{thm:projection-theorem} point 3.
\end{proof}

The guiding principle of this manuscript is that equidistant spacings may be understood by understanding their radii and centers. That is, many questions about equidistant spacings may be turned into questions about its centers and radii. Theorem \ref{thm:projection-theorem} shows that a equidistant spacing induces radii and centers (one per class). A natural question is if the reverse is true, that is if we may pick centers and radii in some way and then identify a equidistant spacings with those centers and radii. This is the case, as the later Theorem \ref{thm:radius-recipe} shows, but there are some compatibility conditions between the $r_i$'s and $c_i$'s. %

We may summarize the above results in the following Corollary. %

\begin{corollary}[Properties of Equidistant Spacings]
    \label{cor:properties-of-class-embeddings}
    Let $\cY \in \ps$. Then the following holds. %
    \begin{enumerate}
        \item For all $i,j$ so that $i \neq j$,
        \begin{align}
            \label{cor:properties-of-class-embeddings:1}
            r_i^2 + r_j^2 &\leq 1\\
            \label{cor:properties-of-class-embeddings:2}
            \norm{c_i - c_j}^2 &= 1 - r_i^2 - r_j^2.
        \end{align}
        \item The affine spaces $A_i$ are such that
        \begin{enumerate}
            \item $r_i > 0 \iff \dim A_i > 0$.
            \item The affine spaces
            \begin{align}
                \label{cor:properties-of-class-embeddings:3}
                A_1,\dots,A_I,\aff(\set{c_1,\dots,c_I})
            \end{align}
            are pairwise orthogonal. 
        \end{enumerate}
    \end{enumerate}
\end{corollary}
\begin{proof}
    We prove each property in order
    \begin{enumerate}
        \item From Proposition \ref{prop:tri-part-ortho} we have that 
        \begin{align}
            1 &= \norm{y_i - y_j}^2 = \norm{y_i - c_i}^2 + \norm{c_i - c_j}^2 + \norm{c_j - y_j}^2\nonumber\\
            &= r_i^2 + r_j^2 + \norm{c_i - c_j}^2\nonumber\\
            &\geq r_i^2 + r_j^2.
        \end{align}
        This proves Eqns. \ref{cor:properties-of-class-embeddings:1} and \ref{cor:properties-of-class-embeddings:2}.
        
        \item Let $r_i > 0$. Then this implies that there is at least one $y_i \in Y_i$ so that $y_i \neq c_i$. This is only possible if $\dim A_i > 0$ (otherwise, $A_i$ is a single point and $\dim A_i = \dim \set{c_i} = 0$). Now, suppose that $\dim A_i > 0$. By minimality of $A_i$, this means that there is a $y_i \in Y_i$ so that $y_i \neq c_i$ (otherwise, $A_i = \set{c_i}$ would be smaller). Hence, $r_i = \norm{y_i - c_i} > 0$. This proves 2 (a).

        The proof that $A_i$ and $A_j$ are orthogonal if $i \neq j$ follows from Lemma \ref{lem:ortho-theorem}. The proof that $A_i$ and $\aff(c_1,\dots,c_I)$ are orthogonal follows from Proposition \ref{prop:tri-part-ortho}. This proves 2 (b).
    \end{enumerate}
\end{proof}

This corollary gives an idea about how to construct $Y_i$'s that arise from an $\cY \in \ps$. First, choose some centers $c_i$'s and radii $r_i$ that satisfy Eqns. \ref{cor:properties-of-class-embeddings:1} and \ref{cor:properties-of-class-embeddings:2}. Then, choose vector spaces $V_1,\dots,V_I$ that are all pairwise orthogonal and orthogonal to $\aff{\set{c_1,\dots,c_I}}$ subject to $r_i > 0 \implies \dim V_i > 0$. Then put $A_i \coloneqq c_i + V_i$. This idea works and, in fact, only depends on how the $r_i$'s are chosen.

\begin{theorem}[Construction of Equidistant Spacings]
    \label{thm:radius-recipe}
    Let radii $\set{r_i}_{i = 1}^I$ be such that $r_i \in [0,1]$, and let the centers $\set{c_i}_{i = 1}^I$ and $\set{V_i}_{i = 1}^I$ be defined such that the following holds.
    \begin{enumerate}
        \item For all $i,j$ so that $i \neq j$,  $r_i^2 + r_j^2 \leq 1$,
        \item The centers satisfy
        \begin{align}
            \label{eqn:center-condition}
            \norm{c_i - c_j}^2 \coloneqq 1 - r_i^2 - r_j^2.
        \end{align}
        \item The subspaces $V_i$ are chosen so that 
        \begin{enumerate}
            \item The dimension of $V_i$ is chosen so that $r_i > 0 \iff \dim(V_i) > 0$.
            \item The set $\set{V_i}_{i = 1}^I$ is pairwise-orthogonal.
            \item Each $V_i$ is orthogonal to $\Span{c_j - c_1,\dots,c_j - c_{I}}$
            for any $j$.
        \end{enumerate}
    \end{enumerate}
    Then $\sqcup_{i = 1}^I Y_i \in \ps(I,\R^n)$ where
    \begin{align}
        \quad Y_i \coloneqq S^{n-1}_{r_i}(c_i) \cap A_i
    \end{align}
    and $A_i \coloneqq c_i + V_i$. Moreover, $\sqcup_{i = 1}^I Y_i$ has centers $\set{c_i}$ and radii $\set{r_i}$.
\end{theorem}

Note that $\Span{c_i - c_1, \dots,c_i - c_I} = \Span{c_j - c_1, \dots,c_j - c_I}$ by Proposition \ref{prop:characterization-of-a_i}. 
\begin{proof}
    
    Let $y_i \in Y_i$ and $y_j \in Y_j$. Note that $y_i - c_i$ and $y_j - c_j$ are orthogonal by the assumption that the $V_i$'s are pairwise orthogonal, and $y_i - c_i$ and $c_i - c_j$ are orthogonal by orthogonality of $V_i$ and $\Span{c_j - c_1, \dots,c_j - c_I}$, and Proposition \ref{prop:affine-space-containing-points} point 2. 
    Applying the Pythagorean theorem twice, we have that
    \begin{align}
        \norm{y_i - y_j}^2 &= \norm{y_i - c_i + c_i - c_j + c_j - y_j}^2\nonumber\\
            &= \norm{y_i - c_i}^2 + \norm{c_i - c_j}^2 + \norm{c_j - y_j}^2\nonumber\\
            &= r_i^2 + \ell^2_{i,j} + r_j^2\nonumber\\
            &= r_i^2 + 1 - r_i^2 - r_j^2 + r_j^2\nonumber\\
            &= 1.
    \end{align}

    By the orthogonality of $V_i$'s, for each $i$, $\proj_{A_i}y_j = c_i$. Thus, $\sqcup_{i = 1}^I Y_i$ has the promised centers. If $r_i = 0$, then $S^{n-1}_{r_i}(c_i) \cap A_i = \set{c_i} \cap A_i = \set{c_i}$. So, $Y_i$ has radius 0 as desired. If $r_i > 0$, then $S^{n-1}_{r_i}(c_i) \cap A_i \neq \emptyset$ if and only if $\dim A_i > 0$. This proves that $Y_i$ has radius $r_i$ as desired.

\end{proof}

Corollary \ref{cor:properties-of-class-embeddings} and Theorem \ref{thm:radius-recipe} combine together to form an if and only if characterization of equidistant contrastive embedded points, as the following corollary shows.

\begin{corollary}[Characterization of Equidistant Spacings]
    \label{cor:characterization-of-equidistant-spacings}
    Let $\cY \coloneqq \set{Y_i}_{i = 1}^I$ be a set of points in $\R^n$. Then $\cY \in \ps(I,\R^n)$ if and only if there are centers $\set{c_i}_{i = 1}^I$, radii $\set{r_i}_{i = 1}^I$, affine subspaces $\set{A_i}_{i = 1}^I$ so that the following holds.
    \begin{enumerate}
        \item The centers and radii are such that 
        \begin{align}
            \label{eqn:center-compatability-condition}
            \norm{c_i - c_j}^2 = 1 - r_i^2 - r^2_j
        \end{align}
        \item For each $i \in I$, $A_i$ passes through $c_i$, and the $V_i$ are such that the following holds.
        \begin{enumerate}
            \item The dimension of $A_i$ is chosen so that $r_i > 0 \iff \dim(A_i) > 0$.
            \item The set $\set{A_i}_{i = 1}^I$ is pairwise-orthogonal.
            \item Each $A_i$ is orthogonal to $\tilde A$ where
            \begin{align}
                \tilde A \coloneqq \Span{c_j - c_1,\dots,c_j - c_{I}}
            \end{align}
            for any $j$.
        \end{enumerate}
        \item For each $Y_i \subset S^{n-1}_{r_i}(c_i) \cap A_i$.
    \end{enumerate}
\end{corollary}
\begin{proof}
    For the forward direction, apply Corollary \ref{cor:properties-of-class-embeddings}. For the inverse direction apply Theorem \ref{thm:radius-recipe}.
\end{proof}

\begin{remark}[Counting Compatibility Equations]
    \label{rmk:counting-compatibility-equations}
    In Corollary \ref{cor:characterization-of-equidistant-spacings}, we can see that the problem of generating equidistant spacings reduces to making a choice of compatible $r_i$'s, $c_i$'s  and choosing subsets $V_i$ that satisfy all of the necessary orthogonality conditions. Making the choice of compatible $r_i$'s and $c_i$'s, however, is non-trivial. One idea is to choose a $c_1,\dots,c_I$, and then let $r_1,\dots,r_I$ be given by solving the linear (in $r_i^2$) system of equations
    \begin{align*}
        \norm{c_1 - c_2}^2 &= 1 - r_1^2 - r^2_2,\\
        \norm{c_1 - c_3}^2 &= 1 - r_2^2 - r^2_3,\\
        \vdots&\\
        \norm{c_{I-1} - c_I}^2 &= 1 - r^2_{I-1} - r^2_I.
    \end{align*}
    There are, however, $I$ unknowns and $\binom{I}{2}$ equations. When $I > 2$, $\binom{I}{2} > I$ and so it seems natural, therefore, to expect that for most choices of $c_1,\dots,c_I$, there are no choices of $r_1,\dots,r_I$ that satisfy Eqn. \ref{eqn:center-compatability-condition} when $I > 2$. This is indeed the case, and ruins the plan to choose $c_i$'s first and then choose the $r_i$'s.
\end{remark}
 To make Remark \ref{rmk:counting-compatibility-equations} more concrete, we present the following example.
\begin{example}[Non-triviality of choosing $\set{c_i}$]
    \label{examp:non-triv-choosing-c_i}
    Let $I = 3$ and $\set{c_i}_{i = 1}^I$ be given by 
    \begin{align}
        \label{eqn:examp:center-def}
        c_1 \coloneqq (0,-\epsilon), \quad c_2 \coloneqq (0,\epsilon) \quad, c_3 \coloneqq (\alpha,0)
    \end{align}
    where $\epsilon > 0$ is given. For which values of $\epsilon,\alpha$ are there choices of $\set{r_i}_{i = 1}^I$ so that Eqn. \ref{eqn:center-compatability-condition} holds?

    From symmetry, it is clear that $r_1 = r_2$. From Eqn. \ref{eqn:center-compatability-condition}, we then have that
    \begin{align}
        \norm{c_1 - c_2} = 2\epsilon  = \sqrt{1 - 2r_1^2} \implies r_1 = \sqrt{\frac12 - 2\epsilon^2}.
    \end{align}
    Applying Eqn. \ref{eqn:center-compatability-condition} again, we have that
    \begin{align}
        \label{eqn:examp:alpha-epsilon-relation}
        \norm{c_1 - c_3}^2 = \epsilon^2 + \alpha^2 = 1 - r^2_1 - r^2_3 
        \implies \alpha^2 = \frac12 + \epsilon^2 - r_3^2.
    \end{align}
    So, if $\alpha^2 > \frac12 + \epsilon^2$ then $r^2_3 < 0$, a contradiction. Hence, if $\alpha^2 > \frac12 + \epsilon^2$ there is no $r_3$ that yields a equidistant spacing with centers given by Eqn. \ref{eqn:examp:center-def}.
\end{example}
Understanding the constraints of the geometry of the centers is a major focus of this manuscript. Supposing that we have a compatible set of $c_i$ and $r_i$, we may present the following result.

\begin{corollary}[Feasibility of Choosing $V_i$]
    \label{cor:feasibility-of-chosing-v}
    Let $\set{r_i}_{i = 1}^I$ and $\set{c_i}_{i = 1}^I$ satisfy Assumptions (1 - 2) of Theorem \ref{thm:radius-recipe}. Then there are a set of $\set{V_i}_{i = 1}^I$ that satisfy Assumption (3) of Theorem \ref{thm:radius-recipe} in $\R^n$ if any only if 
    \begin{align}
        \label{eqn:dimensionality-condition}
        \dim{\tilde V} + \#\set{r_i > 0}_{i = 1}^I \leq n.
    \end{align}
    where $\tilde V \coloneqq \Span{c_1 - c_1, \dots, c_1 - c_I}$. Further, if $\set{d_i}_{i = 1}^I$ are such that $d_i \in \bbN$, $d_i = 0$ if and only if $r_i = 0$, and 
    \begin{align}
        \label{eqn:d-i-condition}
        \dim{\tilde V} + \qsum iI d_i \leq n
    \end{align}
    then there is a choice of $\set{V_i}_{i = 1}^I$ so that $\dim (V_i) = d_i$.
\end{corollary}

\begin{proof}
    By the proof of Theorem \ref{thm:radius-recipe}, we have that the simplex formed by the $c_i$'s is contained in $\tilde V$. %
    Such a choice of $\set{V_i}_{i = 1}^I$ exists if and only if
    \begin{align}
        \label{eqn:dimension-distribution-sum}
        \paren{\qsum iI \dim(V_i)} + \dim(\tilde V) \leq n.
    \end{align}
    
    From Assumption 3.a, we have that $\dim(V_i) > 0$ if and only if $r_i > 0$. Hence the least value of $\qsum iI \dim(V_i)$ is $\#\set{r_i > 0}_{i = 1}^I$. And so,
    \begin{align}
        \#\set{r_i > 0} + \dim(\tilde V) \leq n.
    \end{align}
    This proves the first claim.

    The second claim follows from letting $d_i = \dim(V_i)$ in Eqn. \ref{eqn:dimension-distribution-sum} and moving the $\dim(\tilde V)$ term to the other side.
\end{proof}

Corollary \ref{cor:feasibility-of-chosing-v} says, in short, that the only difficult part of generating $\cY \in \ps$ is finding compatible centers and radii.

\subsection{The Ortho-Reflection Lemma and its Consequences}
\label{sec:ortho-proj-lemma-and-conseq}

\begin{figure}
    \centering
    \begin{subfigure}{.48\linewidth}
        \centering
        \includegraphics[width=1.\linewidth]{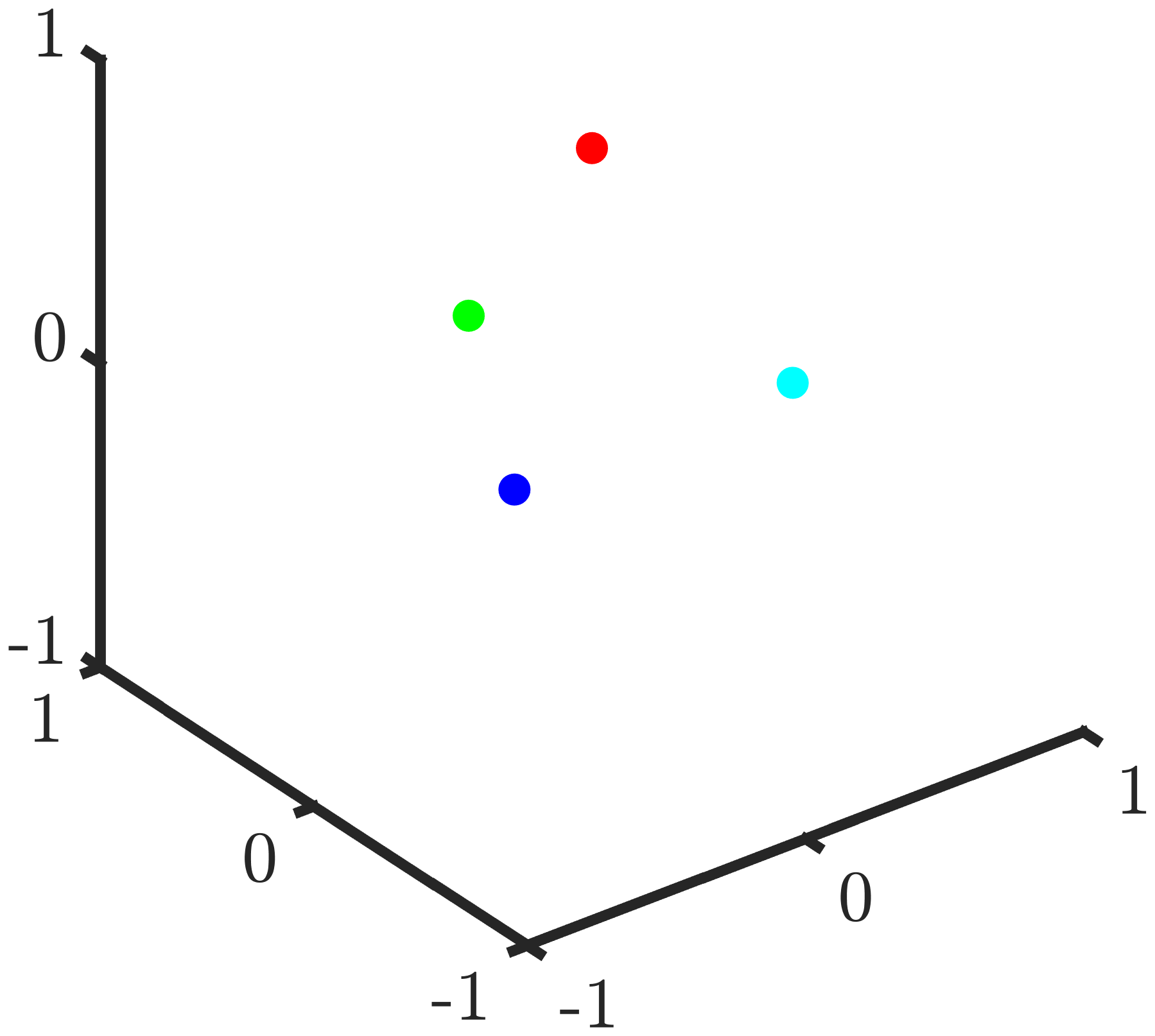}
        \subcaption{$\cY$}
    \end{subfigure}
    \begin{subfigure}{.48\linewidth}
        \centering
        \includegraphics[width=1.\linewidth]{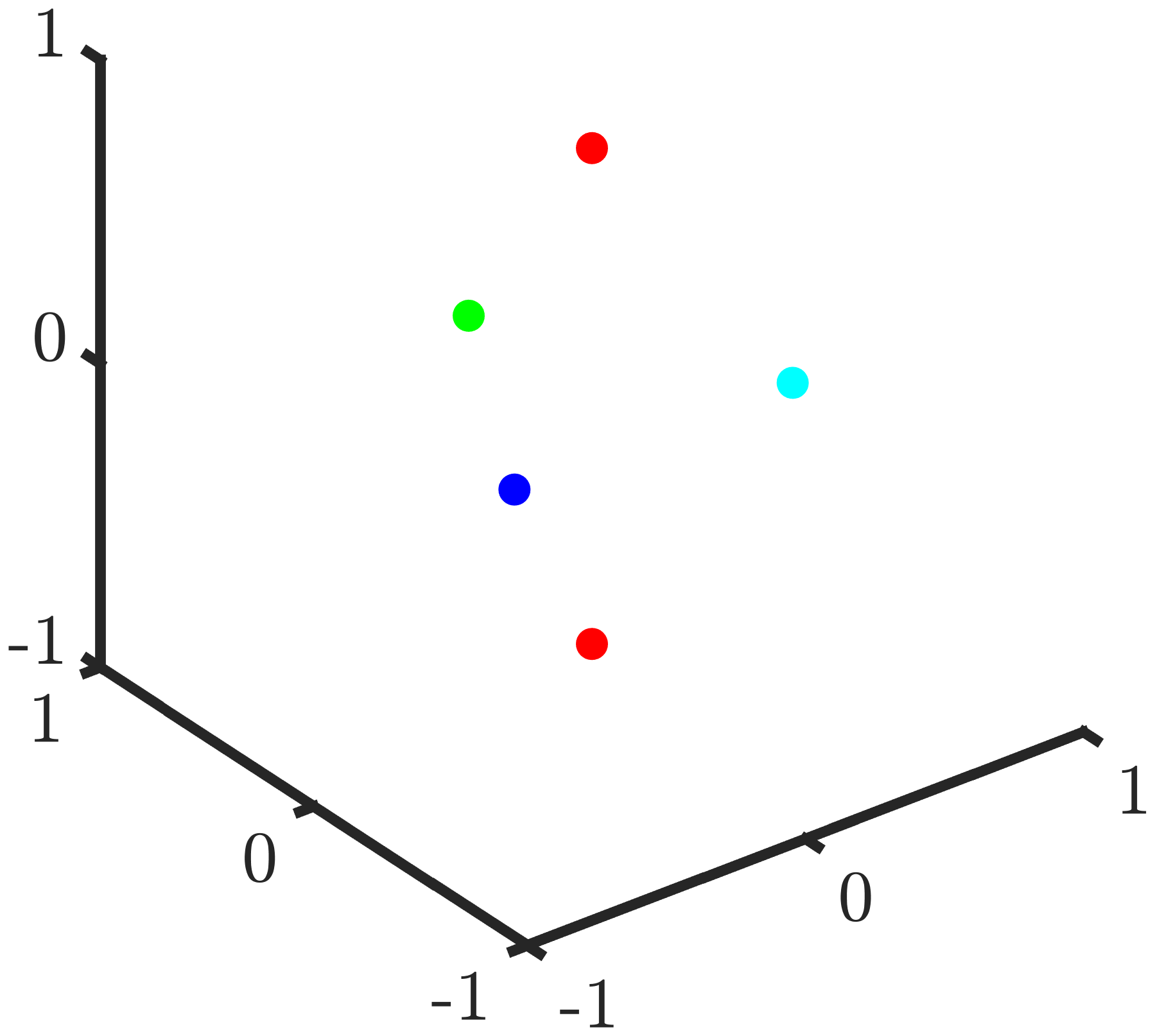}
        \subcaption{$\cY'$.}
    \end{subfigure}
    \caption{An example of the construction in Lemma \ref{lem:ortho-reflection-lemma} in three dimensions for $I = 4$. (a) An illustration of $\cY \in \ps(4,\R^n)$ where the red dot denotes $Y_4$. (b) By applying Lemma \ref{lem:ortho-reflection-lemma} we union the red dot with its reflection across its orthogonal projection onto the affine space containing the other three dots. This produces a spacing $\cY'$ that maximizes $\cY$.}
    \label{fig:ortho-proj-lemma}
\end{figure}

Next, we introduce a construction that enables deeper geometric results. The idea, broadly speaking, is that we may make a equidistant spacing larger in one class by projecting the class' center orthogonally onto the subspace spanned by the other centers, and doing a reflection. %

\begin{lemma}[Ortho-Reflection Lemma]
    \label{lem:ortho-reflection-lemma}
    Let $\cY \in \ps$ with centers $\set{c_i}_{i = 1}^I$. 
    
    Then, there is a equidistant spacing $\cY' = \sqcup_{i = 1}^I Y_i' \in \ps$ so that
    \begin{align}
        \label{eqn:lem:ortho-peojection-lemma:1}
        Y_i' = Y_i, \quad \text{for each }i \leq I - 1
    \end{align}
    and 
    \begin{align}
        \label{eqn:lem:ortho-peojection-lemma:2}
        Y_I' = Y_I \cup \reflect_{c^*_I} Y_I \quad \text{where} \quad c^*_I \coloneqq \proj_{\aff(c_1,\dots,c_{I-1})}c_I,
    \end{align}
    where $\reflect$ denotes the reflection operator, see Def. \ref{def:reflection-operator}.
\end{lemma}
\begin{proof}
    Let $y_I,y_I^* \in Y_I'$ be reflections of each other, and $y_j \in Y_j$ for some $j \leq I-1$. Consider the triangles $\triangle y_I c^*_I y_j$ and $\triangle y_I^*c^*_I y_j$. The triangles share the sides $c^*_I y_j$, and the length of sides $y_Ic^*_I$ and $y_I^*c^*_I$ are the same. Further, the angles $\angle y_I c^*_I y_j$ and $\angle y_I c^*_I y_j$ are the same (namely, right angles), as follows from the following computation.
    
    \begin{align}
        \innerprod{c_j - c^*_I}{c^*_I - y^*_I} &= \innerprod{c_j - c^*_I}{c^*_I - y_I - 2(y - c_I^*)}\\
        &= \innerprod{c_j - c^*_I}{y_I - c^*_I}\\
        &= \innerprod{c_j - c^*_I}{y_I - c_I} + \innerprod{c_j - c^*_I}{c_I - c^*_I} = 0.
    \end{align}
    Hence, the two triangles are equivalent by side angle side and so $\norm{y^*_I - y_j} = \norm{y_I - y_j} = 1$ which proves that $\cY_I' \in \ps$.
\end{proof}

A visual example application of this lemma is given in Figure \ref{fig:ortho-proj-lemma}. Note that $\abs{\cY} \subset \abs{\cY'}$. A useful way to think of Lemma \ref{lem:ortho-reflection-lemma} is as an algorithm that takes a equidistant spacing $\cY$ as inputs, and produces a $\cY'$ that may be larger. The main use of Lemma \ref{lem:ortho-reflection-lemma}, is that it produces an invariant for maximal equidistant spacings. This is described in the following definition and proposition.

\begin{definition}[Maximal Equidistant Spacing]
    \label{def:maximal-perf-spacing}
    We say that $\cY \in \ps(I,\R^n)$ is \emph{maximal} if it is maximal under the subset operation in $\ps(I,\R^n)$. That is, $\cY$ is maximal when $\cY' \in \ps(I,\R^n)$ and $\cY \subset \cY'$ implies that $\cY = \cY'$. We define
    \begin{align}
        \mps(I,\R^n) \coloneqq \set{\cY \in \ps(I,\R^n) \colon \text{$\cY$ is maximal}}.
    \end{align}
\end{definition}

\begin{proposition}
    \label{prop:any-ortho-proj-implies-non-maximality}
    Let $\cY \in \mps$. Then
    \begin{align}
        c_I \in \aff(c_1,\dots,c_{I-1}).
    \end{align}
\end{proposition}

\begin{proof}
    Let us apply Lemma \ref{lem:ortho-reflection-lemma} to $\cY$ and define $c^*_I \coloneqq \proj_{\aff(c_1,\dots,c_{I-1})} c_I$.

    Suppose that $\reflect_{c^*_I} Y_I \not \subset Y_I$, then we may make $\cY$ strictly larger by replacing $Y_I$ with  $Y_I \cup \reflect_{c^*_I} Y_I$ via Lemma \ref{lem:ortho-reflection-lemma}. This contradicts maximality of $\cY$. Hence, $\reflect_{c^*_I} Y_I \subset Y_I$.
    
    Let $j \leq I-1$ be arbitrary. Then from definitions we have that $\proj_{A_i} y_j = c_i$. Further, we have that $c^*_I\in A_i$ (because it's on the line connecting $y_I$ and $\reflect_{c^*_I}y_I$) and that $y_j - c^*_I$ and $c^*_I - y_I$ are orthogonal by Proposition \ref{prop:tri-part-ortho}. Hence, $\proj_{A_i} y_j = c_i^*$. Thus $c_i^* = c_i$, so
    \begin{align}
        c_I = c^*_I = \proj_{\aff(c_1,\dots,c_{I-1})} c_I.
    \end{align}

\end{proof}

We may also use Lemma \ref{lem:ortho-reflection-lemma} to prove the following two corollaries, which relate how equidistant spacings may look if dimensions are `projected off.'

\begin{proposition}[Center Orthogonal Projections Coincide]
    \label{prop:center-ortho-projs-coincide}
    Let $\cY \in \ps$. Then
    \begin{align}
        \label{eqn:center-ortho-projs-coincide}
        \proj_{\aff(c_1,\dots,c_{I-2})}c_{I-1} = \proj_{\aff(c_1,\dots,c_{I-2})}c_{I}
    \end{align}
\end{proposition}

\begin{proof}
    When $I = 2$, the projection is onto the origin, and so Eqn. \ref{eqn:center-ortho-projs-coincide} trivially follows. If $I = 3$ then $\aff(c_1,\dots,c_{I-2})$ only has one point ($c_1$), and so the result trivially follows as well. 

    Let $I \leq 4$. Let $\cY^{(I-1)}$ be $\cY$ with class $I-1$ removed. Likewise, let $\cY^{(I)}$ be $\cY$ with class $I$ removed. Let $\hat \cY^{(I-1)}$ and $\hat \cY^{(I)}$ be the equidistant spacings given by applying Lemma \ref{lem:ortho-reflection-lemma} to $\cY^{(I-1)}$ and $\cY^{(I)}$ respectively. 
    
    Let us denote 
    \begin{align}
        \hat c_{I-1}&\coloneqq \proj_{\aff(c_1,\dots,c_{I-2})}c_{I-1}\\
        \hat c_{I}&\coloneqq\proj_{\aff(c_1,\dots,c_{I-2})}c_{I}.
    \end{align}
    The claim follows if we can show that $\hat c_{I-1} = \hat c_{I}$. We denote by $\hat Y_{I-1}$ the $I-1$'th class of $\hat \cY^{(I-1)}$ and similarly we denote by $\hat Y_{I}$ the $I-1$'th class of $\hat \cY^{(I)}$.
    
    Denoting $r_{I-1}^*$ and $r_I^*$ as the radii of classes $\hat Y_{I-1}$ and $\hat Y_{I}$ respectively, for an arbitrary $y_1 \in Y_1, \hat y_{I-1} \in \hat Y_{I-1}$ and $\hat y_I \in \hat Y_{I}$ we have the following result by Proposition \ref{prop:tri-part-ortho},
    \begin{align}
        0 &= \norm{y_1 - \hat y_I^2} - \norm{y_1 - \hat y_{I-1}}^2\\
        &=\norm{y_1 -c_1}^2 + \norm{c_1 - \hat c_I^2} + \norm{\hat c_I - \hat y_I}^2 \\ 
        &- \norm{y_1-c_1}^2 - \norm{c_1 - \hat c_{I-1}}^2 - \norm{\hat c_{I-1} - \hat y_{I-1}}^2\\
        &= \norm{c_1 - \hat c_I}^2 - \norm{c_1 - \hat c_{I-1}}^2 + r_I^{*2} - r_{I-1}^{*2}.\label{eqn:proof:prop:center-ortho-projs-coincide}
    \end{align}

    Suppose that $\hat c_{I-1} \neq \hat c_{I}$. Then there must be one center in $c_1,\dots,c_{I-2}$ closer to $\hat c_{I-1}$ then $\hat c_{I}$, and vice versa. W.l.o.g., let $c_1$ and $c_2$ be such that 
    \begin{align}
        \norm{c_1 - \hat c_I} < \norm{c_1 - \hat c_{I-1}},\quad \text{and} \quad \norm{c_1 - \hat c_{I-1}} < \norm{c_1 - \hat c_{I}}.
    \end{align}
    Hence, we have that by applying Eqn. \ref{eqn:proof:prop:center-ortho-projs-coincide} $r_{I-1} > r_I$. By running through the same calculations but with $c_2$ in the place of $c_1$, we obtain that $r_I > r_{I-1}$. A contradiction. Thus $\hat c_I = \hat c_{I-1}$.
\end{proof}

Proposition \ref{prop:center-ortho-projs-coincide} implies the following, as well.

\begin{corollary}[Projected Radii Coincide]
    \label{cor:projected-radii-coincide}
    Let $\cY \in \ps$ and $c^*_{I}\coloneqq \proj_{\aff(c_1,\dots,c_{I-2})}c_I$ and $c^*_{I-1}\coloneqq \proj_{\aff(c_1,\dots,c_{I-2})}c_{I-1}$. Then, for any $y_I$ and $y_{I-1}$,
    \begin{align}
        \label{eqn:cor:proj-radii:1}
         \norm{y_I - c^*_I}^2_2 &= \norm{y_{I-1} - c^*_{I-1}}^2_2\\
        \label{eqn:cor:proj-radii:2}
         \norm{y_I - c_I}^2_2 + \norm{c_I - c_I^*}^2_2 &=\norm{y_{I-1} - c_{I-1}}^2_2 + \norm{c_{I-1} - c_{I-1}^*}^2_2
    \end{align}
\end{corollary}

\begin{proof}
    Equation \ref{eqn:cor:proj-radii:1} follows from combining Eqn. \ref{eqn:proof:prop:center-ortho-projs-coincide} with Proposition \ref{prop:center-ortho-projs-coincide}. Eqn. \ref{eqn:cor:proj-radii:1} implies \ref{eqn:cor:proj-radii:2} by Pythagorean's theorem.
\end{proof}

By using the trick of reflecting across affine spaces, we may prove the following theorem that relates the condition of being a $\cY \in \ps$ maximized by another spacing all of whose centers coincide with the rank of the affine space that contains the centers of $\cY$, with two of $\cY$'s centers removed. 

\begin{theorem}[Leave-two-out Affine Rank implies Maximization by Coinciding Centers]
    \label{thm:leave-two-out-affine-rank-implies-max-by-concide-centers}
    Let $\cY$ be a equidistant spacing with centers $\set{c_i}_{i = 1}^I$ and let there be an $i,j \leq I$ where $i \neq j$ so that
    \begin{align}
        \label{eqn:affine-rank-condition}
        \aff(c_1,\dots,c_{i-1},c_{i+1},\dots,c_{j-1},c_{j+1},\dots,c_I) = \aff(c_1,\dots,c_I).
    \end{align} %
    
    Then, there is a equidistant spacing $\hat \cY$ that maximizes $\cY$, so that $\hat c_1 = \dots = \hat c_I$ where $\hat c_1,\dots,\hat c_I$ denote the centers of $\hat \cY$.
\end{theorem}

\begin{proof}
    To ease notation, w.l.o.g. let $i = I-1$ and $j = I$. By Proposition \ref{prop:center-ortho-projs-coincide}, we have that 
    \begin{align}
        c_I = \proj_{\aff(c_1,\dots,c_{I-2})}c_I = \proj_{\aff(c_1,\dots,c_{I-2})}c_{I-1} = c_{I-1}.
    \end{align}
    From Corollary \ref{cor:projected-radii-coincide}, we have that $r_I = r_{I-1} = \frac{\sqrt{2}}2$. From this fact, we have that $\norm{y_i - c_I} = \frac{\sqrt{2}}2$ for a arbitrary $k \leq I$. Hence, by the law of cosines, the lines $\overline {y_kc_I}$ and $\overline {y_\ell c_I}$ are orthogonal for an arbitrary $k,\ell \leq I$ such that $k \neq \ell$.

    Consider $\hat \cY \in \ps$ so that $\hat \cY = \set{\hat Y_k}^I_{k = 1}$ defined by
    \begin{align}
        \hat Y_{k} = Y_k \cup \reflect_{c_I}Y_k \quad \text{ for each }k = 1,\dots,I-2,
    \end{align}
    $\hat Y_{I} = Y_{I}$, and $\hat Y_{I-1} = Y_{I-1}$. Because of orthogonality of $\overline {y_kc_I}$ and $\overline {y_\ell c_I}$ we have, by the same arguments as in the proof of Lemma \ref{lem:ortho-reflection-lemma}, that $\norm{y_\ell - y_k} = \norm{\hat y_\ell - y_k}$, $\norm{y_\ell - y_{I-1}} = \norm{\hat y_\ell - y_{I-1}}$ and $\norm{y_\ell - y_{I}} = \norm{\hat y_\ell - y_{I}}$ for any $y_\ell \in Y_\ell, \hat y \in \reflect_{c_I}Y_\ell$, $y_k \in Y_k$, $y_{I-1} \in Y_{I-1}$ and $y_I\in Y_I$. Because $\ell$ and $k$ are arbitrary, we thus have that $\hat \cY$ is a equidistant spacing.

    We now prove that all centers of $\hat \cY$ coincide. This follows from observing that $c_I\in \aff(Y_k \cup \reflect_{c_I}Y_k)$, and so applying Def. \ref{def:center-and-radius}, we have that $c_k = \proj_{\aff(Y_k \cup \reflect_{c_I}Y_k)}y_I = c_I$ for any $y_I \in \hat Y_i$.
\end{proof}

Theorem \ref{thm:leave-two-out-affine-rank-implies-max-by-concide-centers} and Proposition \ref{prop:any-ortho-proj-implies-non-maximality} have the following immediate corollary. It provides a relationship between maximality of a equidistant spacing, and the rank of the affine spaces that span its centers.

\begin{corollary}[Leave-out Rank of Maximal Spacings]
    \label{cor:leave-out-rank-and-max-spacing}
    Let $\cY\in \mps$ with centers $c_1,\dots,c_I$, that don't all coincide and $i \neq j$. Then,
    \begin{align}
        \label{eqn:leave-out-rank:1}
         \aff(c_1,\dots,c_{i-1},c_{i+1},\dots,c_I) = \aff(c_1,\dots,c_I)
    \end{align}
    and
    \begin{align}
        \label{eqn:leave-out-rank:2}
        \aff(c_1,\dots,c_{i-1},c_{i+1},\dots,c_{j-1},c_{j+1},\dots,c_I) \subsetneq \aff(c_1,\dots,c_I).
    \end{align}
\end{corollary}

\begin{proof}
    Eqn. \ref{eqn:leave-out-rank:1} follows immediately from Proposition \ref{prop:any-ortho-proj-implies-non-maximality}.

    Eqn. \ref{eqn:leave-out-rank:2} follows from uniqueness of maximal embeddings. The hypothesis implies that $\cY$ is not maximized by a equidistant spacing with centers that coincide. Therefore, Eqn. \ref{eqn:affine-rank-condition} must not hold.
\end{proof}

Corollary \ref{cor:leave-out-rank-and-max-spacing} gives a simple proof of the following fact.
\begin{corollary}[If Two Centers Coincide in a Maximal Spacing, they all do]
    \label{cor:if-two-centers-coincide-in-a-maximal-spacing-they-all-do}
    Let $\cY \in \mps$. If $i \neq j$ are such that $c_i = c_j$, then $c_1 = \dots = c_I$.
\end{corollary}
\begin{proof}
    If $I = 2$, the claim is automatic. Let $I \geq 3$, and let $k$ be distinct from $i$ and $j$. From $c_i = c_j$ we have that %
    \begin{align*}
        &\aff(c_1,\dots,c_{i-1},c_{i+1},\dots,c_{j-1},c_{j+1},\dots,c_I) \underbrace{=}_{c_i = c_j} \aff(c_1,\dots,c_{i-1},c_{i+1},\dots,c_I) \underbrace{=}_{\text{Eqn. \ref{eqn:leave-out-rank:1}}} \aff(c_1,\dots,c_I).
    \end{align*}
    This contradicts Eqn. \ref{eqn:leave-out-rank:2} of Corollary \ref{cor:leave-out-rank-and-max-spacing} unless $c_1 = \dots = c_I$. Hence $c_1 = \dots = c_I$.
\end{proof}

In light of Corollary \ref{cor:if-two-centers-coincide-in-a-maximal-spacing-they-all-do}, we introduce the following definition.

\begin{definition}[Maximal Coinciding/Noncoindicing Centers]
    \label{def:maximal-coinciding-noncoinciding-centers}
    We define the two subset of $\mps$ as follows.
    \begin{align}
        \mps_=(I,\R^n) &\coloneqq \set{\cY \in \mps \colon \text{the centers of $\cY$ coincide}}\\
        \mps_{\neq}(I,\R^n) &\coloneqq \set{\cY \in \mps \colon \text{the centers of $\cY$ do not coincide}}
    \end{align}
\end{definition}

\begin{remark}[The Nature of Equidistant Spacings whose Centers Coincide]
    \label{rmk:the-nature-of-equidistant-spacings-whose-centers-coincide}

    Using this new notation, we may rephrase Corollary \ref{cor:if-two-centers-coincide-in-a-maximal-spacing-they-all-do} with the following equations.
    \begin{align}
        \mps(I,\R^n) = \mps_=(I,\R^n)\cup \mps_{\neq}(I,\R^n).
    \end{align}
    Clearly, $\mps_=(I,\R^n)\cap \mps_{\neq}(I,\R^n) = \emptyset$, so we see that the set $\mps$ separates into two classes. In the first, $\mps_=$, \emph{all} of the centers coincide. In the second, $\mps_{\neq}$, \emph{none} of the centers coincide. It turns out, separating $\mps$ in this way is natural. 
    
    We shall see that many statements about $\cY \in \mps$ have one part which applies when $\cY \in \mps_=$, and another when $\cY \in \mps_{\neq}$, as happened in Corollary \ref{cor:leave-out-rank-and-max-spacing}. Other examples occur later in the manuscript, such as Definitions \ref{def:extent} \ref{def:inner-class}, \ref{def:equilateral-normal-form}, \ref{def:signagure-of-equidistant-spacing}.
\end{remark}

Note, that in the following proposition and the rest of the manuscript, $\mathrm{conv}$ refers to the convex hull operation, see Definition \ref{def:convex-hull}. Before proceeding, we review the concepts of orthocentric simplices and orthocentric systems from elementary geometry.

\begin{definition}[Orthocentric Simplex]
    \label{def:orthocentric-simplex}
    Let $\sigma$ be a $n$ simplex. We say that $\sigma$ is \emph{orthocentric} if the $n+1$ altitudes have a common point called the \emph{orthocenter}.
    
    Let $x_1,\dots,x_{n+1}$ be vertices of an orthocentric $n$ simplex with orthocenter $x_0$. Then $x_0,x_1,\dots,x_{n+1}$ are called an \emph{orthocentric system} and omitting any point from this system yields an orthocentric simplex whose orthocenter is the omitted point.
\end{definition}

All triangles have an orthocenter. In higher dimensions, this is not the case. A simplex is generically not an orthocentric simplex.

\begin{proposition}[Equivalent Conditions for Orthocentricity/Orthocentric System]
    \label{prop:equiv-conditions-for-orthocentricity}
    Let $\triangle$ be an $n$ simplex. Then the following are equivalent.
    \begin{enumerate}
        \item $\triangle$ is orthocentric.
        \item The feet of the altitudes of $\triangle$ are the orthocenters of their respective faces.
        \item Supposing that $n \geq 3$, each edge of $\triangle$ is orthogonal to the opposite $n-2$ face.
    \end{enumerate}
    Let points $x_0,\dots,x_{n+1} \in \bbR^n$ be given. Then the following are equivalent.
    \begin{enumerate}
        \item The points $x_0,\dots,x_{n+1}$ form an orthocentric system.
        \item There are $\lambda_0,\dots,\lambda_{n+1} \in \bbR$ so that 
        \begin{align}
            \label{eqn:carycentric-coordinates-orthocentric-system}
            \norm{x_i - x_j}^2 = \lambda_i + \lambda_j\quad \text{with} \quad \sum_{i = 0}^{n+1} \frac1{\lambda_i} = 0, \text{ and } \lambda_i + \lambda_j > 0, \ i\neq j.
        \end{align}
        \item Let $p \colon \R^n \to \R^{n+1}$ be defined by $p(x) = (x,1)$. Then, there is an $i \in \set{0,\dots,n+1}$ so that the matrix
        \begin{align}
            A \coloneqq \pmat{p(x_0 - x_i) & \dots &p(x_{i-1} - x_i)&p(x_{i+1} - x_i)&\dots&p(x_{n+1} - x_i)}
        \end{align}
        is orthogonal.
        
    \end{enumerate}
\end{proposition}

\begin{proof}
    For the equivalent conditions for being an orthocentric simplex. For Points 1 $\iff$ 2, see \cite{klamkin1998orthocenters}. For Points 2 $\iff$ 3 see \cite[Page 4]{edmonds2005orthocentric} and the references contained therein.

    For the equivalence of points 2 and 3, see \cite[Section 2]{egervary1940orthocentric}.
\end{proof}

Now we present a powerful geometric result that gives a geometric criterion for how centers in a maximal embedding lie.

\begin{proposition}[Centers are Orthocentric Systems]
    \label{prop:centers-are-orthocentric-systems}
    Let $\cY \in \mps_{\neq}$. Then $c_1,\dots,\dots,c_I$ are an orthocentric system.
\end{proposition}

\begin{proof}
    Let us define $\bigtriangleup_i \coloneqq\mathrm{conv}(c_1,\dots,c_{i-1},c_{i+1},\dots,c_I)$. The claim follows if we can show that $\bigtriangleup_i$ is a simplex, and $c_i$ is its orthocenter. $\bigtriangleup_i$ being a simplex is equivalent to showing that $c_1,\dots,c_{i - 1},c_{i + 1},\dots,c_I$ are affinely independent. This follows from Corollary \ref{cor:leave-out-rank-and-max-spacing}.

    Now we show that $c_i$ is the orthocenter of $\bigtriangleup_i$. Note, we know that $c_i \neq c_j$ when $i\neq j$ from Corollary \ref{cor:if-two-centers-coincide-in-a-maximal-spacing-they-all-do}. Hence, the line $\overline{c_jc_i}$ is well-defined. For ease of notation, let us define
    \begin{align}
        A_{(i,j)}\coloneqq \aff\paren{c_1,\dots,c_{j-1},c_{j+1},\dots,c_{i-1},c_{i+1},\dots,c_I}.
    \end{align}
    and $c^*_j \coloneqq \proj_{A_{(i,j)}}c_j$. From Proposition \ref{prop:center-ortho-projs-coincide}, we have that 
    \begin{align*}
        \proj_{A_{(i,j)}} c_i = c^*_j
    \end{align*}
    Thus, $c_i, c_j$, and $c^*_j$ all lie on a line. Because $c^*_j$ is defined by the orthogonal projection, this line is orthogonal to $A_{(i,j)}$. Hence, $c_i$ is contained on the line containing $c_j$ orthogonal to $A_{(i,j)}$, that is, $c_i$ lies on the line connecting $c_j$ to its orthogonal projection on its opposite face in $\bigtriangleup_i$. This holds for each $j$, and so $c_i$ is the orthocenter of $\bigtriangleup_i$.
\end{proof}
The following corollary gives a useful identity
\begin{corollary}[Dimension of $\tilde A$]
    \label{cor:dim-of-tilde-A}
    Let $\cY \in \mps(I,\R^n)$, then $n = \dim \tilde A + \qsum iI\dim A_i$ and 
    \begin{align}
        \label{eqn:a-tilde-fomula}
        \dim \tilde A = \twopartpiecewise{0}{\cY \in \mps_=}{I-2}{\cY \in \mps_{\neq}}.
    \end{align}
\end{corollary}
\begin{proof}
    Suppose that for all $i$, $r_i = 0$. Then $\cY$ is an equilateral simplex, and may be strictly maximized via the reflection trick as in Figure \ref{fig:ortho-proj-lemma}, a contradiction.
    
    Let $i$ be such that $r_i > 0$. Define
    \begin{align}
        \hat V_i \coloneqq \paren{\aff{\abs\cY}}^\perp = \paren{\aff{\tilde A + \qsum iI A_i}}^\perp
    \end{align}
    and $\hat A_i = c_i V_i$. Define $\hat \cY$ so that 
    \begin{align}
        \hat Y_j = \twopartpiecewise{Y_j}{i \neq j}{\paren{ A_i  + \hat A_i} \cap S^{n-1}_{r_i}(c_i)}{i = j}.
    \end{align}
    Then, $\hat\cY \in \ps$. If $n > \dim \tilde A + \qsum iI\dim A_i$, then $\hat V_i$ is not trivial and so $Y_i\subsetneq\hat Y_i$, hence $\cY$ is not maximal. This is a contradiction. Thus, $n = \dim \tilde A + \qsum iI\dim A_i$.
    
    Now we prove Eqn. \ref{eqn:a-tilde-fomula}. When the centers coincide the proof is automatic. When the centers don't coincide, then 
    \begin{align}
        \dim \tilde A = \dim \aff{c_1,\dots,c_I} = \dim \mathrm{conv}(c_1,\dots,c_I) = I-2
    \end{align}
    where the last equality follows from Proposition \ref{prop:centers-are-orthocentric-systems}.
\end{proof}

It turns out that the compatibility condition, Eqn. \ref{eqn:center-compatability-condition}, is equivalent to the statement that the centers form an orthocentric simplex.

\begin{proposition}[The Ortho-Reflection Lemma Produces Orthocentric Systems]
    \label{prop:the-ortho-reflection-lemma-produces-orthocentric-systems}
    Let $c_1,\dots,c_I$ be centers of $\cY \in \ps$ and affinely independent. Then $c_1,\dots,c_{I-1},c^*_I$ is an orthocentric system where $c^*_I \coloneqq \proj_{\aff(c_1,\dots,c_{I-1})}c_I$.
\end{proposition}

\begin{proof}
    The goal is to show that $\overline{c_{i}c^*_I} \in \aff(c_1,\dots,c_{i-1},c_{i+1},\dots,c_{I-1})^\perp$ for each $i = 1,\dots, I-1$. By the ortho-reflection trick (applied to $I$), we may define $\cY' \in \ps$ where $c'_i = c_i$ when $i = 1,\dots,I-1$, and $c_I' = c^*_I$. Then, Corollary \ref{cor:projected-radii-coincide} applied to $\cY'$ implies that 
    \begin{align}
        \proj_{\aff(c_1',\dots,c_{i-1}',c_{i+1}',\dots,c_{I-1}')}c_i' = \proj_{\aff(c_1',\dots,c_{i-1}',c_{i+1}',\dots,c_{I-1}')}c_I'.
    \end{align}
    By replacing $c'_i = c_i$, $c'_I = c^*_I$ and rearranging terms we have 
    \begin{align}
        \proj_{\aff(c_1,\dots,c_{i-1},c_{i+1},\dots,c_{I-1})}\paren{c_i -c_I^*} = \bszero \implies \overline{c_{i}c^*_I} \in \aff(c_1,\dots,c_{i-1},c_{i+1},\dots,c_{I-1})^\perp
    \end{align}
    which proves the result.
\end{proof}

Proposition \ref{prop:the-ortho-reflection-lemma-produces-orthocentric-systems} implies that a version of Proposition \ref{prop:centers-are-orthocentric-systems} that applies to non-maximal equidistant spacings.

\begin{corollary}[Center form Orthocentric Simplices]
    \label{cor:equidistantly-s}
    If $\cY \in \ps(I)$, then $\triangle_{c_1,\dots,c_I}$ forms a (possibly degenerate) orthocentric simplex.
\end{corollary}
\begin{proof}
    Consider applying the ortho-reflection lemma to $\cY$ in class $I$. 
\end{proof}

Propositions \ref{prop:centers-are-orthocentric-systems} and \ref{prop:the-ortho-reflection-lemma-produces-orthocentric-systems} have inverses, up to scaling, as the following theorem shows.

\begin{theorem}[Small Orthocentric Systems Correspond to Centers in $\mps_{\neq}$]
    \label{thm:small-orthocentric-systems-correspond-to-centers-in-mps}
    Let $c_0,\dots,c_{n+1} \in \R^m$ be an orthocentric system. If $m$ is large enough, there is some $\alpha \in (0,\infty)$ and $\cY \in \mps_{\neq}(n+1,\R^m)$ so that $\alpha c_0,\dots,c\alpha _{n+1}$.
\end{theorem}

\begin{proof}
    Let $c_0,\dots,c_{n+1}$ be given. Then, by Eqn. \ref{eqn:carycentric-coordinates-orthocentric-system} there are $\lambda_0,\dots,\lambda_{n+1}$ so that $\norm{c_i - c_j}^2 = \lambda_i +\lambda_j$ for each $i \neq j$. If we choose
    \begin{align}
        \alpha^2 \leq \min( \inf _{i}\frac1{2\lambda_i}, \inf_{i,j}\norm{c_i-c_j})
    \end{align}
    then $\norm{\alpha c_i - \alpha c_j} \leq 1$ and $\frac12 - \alpha^2\lambda_i \geq 0$. Then $\norm{\alpha c_i - \alpha c_j}^2 = \alpha^2\lambda_i +\alpha^2\lambda_j$. If we define $r_i \coloneqq \sqrt{\frac12 - \alpha^2\lambda_i}$, then we see that $\norm{\alpha c_i - \alpha c_j}^2 = 1 - r_i^2 - r_j^2$. Thus, this choice of $\alpha c_0,\dots,\alpha c_{n+1}$ and $r_{0},\dots,r_{n+1}$ satisfy Eqn. \ref{eqn:center-compatability-condition}, $r_i^2 + r_j^2 \leq 1$.

    By Corollary \ref{cor:dim-of-tilde-A}, $\dim \aff{\alpha c_0 , \dots , \alpha c_{n+1}}$. Hence, by Corollary \ref{cor:feasibility-of-chosing-v} if $m \geq n + \#\set{r_i > 0}^{n+1}_{i=0}$ then there is a choice of affine space $A_i$ so that $\cY \in \ps(n+2,\R^m)$ where
    \begin{align}
        Y_i \coloneqq S^{m-1}_{c_i}(r_i) \cap A_i.
    \end{align}
    By choosing the $A_i$ so that $\sum _{i = 0}^{n+1} \dim A_i = m - n$, then $\cY$ is maximal by Corollary \ref{cor:dim-of-tilde-A} and Proposition \ref{prop:centers-are-orthocentric-systems}.
\end{proof} 

Propositions \ref{prop:equiv-conditions-for-orthocentricity} and \ref{prop:centers-are-orthocentric-systems} implies that the centers of a $\cY \in \mps_{\neq}$ are not geometrically homogeneous. For one of the centers, $c_i$, the corresponding $\lambda_i < 0$. Geometrically, this means that $c_i$ lies in the convex hull of the other centers. We call this distinguished center the inner center. The inner center becomes very important later.

\begin{definition}[Inner Class]
    \label{def:inner-class}
    Let $\cY \in \ps$. We say $i$ is an \emph{inner class} of $\cY$ if $c_i \in \mathrm{conv}(c_1,\dots,c_{i-1},c_{i+1},\dots,c_I)$.

    A \emph{non-inner class} is any class that is not an inner class.
\end{definition}
\begin{lemma}[Existence of Inner Center]
    \label{lem:existence-of-inner-center}
    When $I > 1$, each $\cY \in \mps(I)$ has an inner class. If $\cY \in \mps_{\neq}(I)$, then the inner class is unique. %
\end{lemma}

\begin{proof}

    The proof is obvious when $\cY \in \mps_{=}$ as the centers coincide.
    
    By Proposition \ref{prop:centers-are-orthocentric-systems}, the centers of $\cY \in \mps_{\neq}$ form an orthocentric system. By Proposition \ref{prop:equiv-conditions-for-orthocentricity} point 3, there is a center which, when regarded as the origin, the remaining columns form an orthogonal matrix. Call this center $c_i$. Geometrically, it is clear that $c_i$ is contained in the convex hull of the other classes. Further, this center is unique.

\end{proof}

\subsection{Gluing equidistant spacings together}
\label{sec:gluing-equidistant-spacings-together}

Given some number of equidistant spacings, it may be possible to embed them into some common space so that they form a larger equidistant spacing. In this section, we determine when this is possible.

\begin{definition}[Glue Site, Gluing]
    \label{def:gluing}
    Let $\cY \in \ps$. If $\abs{\cY} \in \ps(1)$, then we define the \emph{glue site} as the center of  $\abs{\cY}$.
    
    Let $N \in \bbN$, and $\cY^{(1)} \in \ps(I^{(1)},\R^{n^{(1)}}), \dots, \cY^{(N)} \in \ps(I^{(N)},\R^{n^{(N)}})$. We say that $\cY^{(1)},\dots,\cY^{(N)}$ \emph{glue together in $\R^m$} or \emph{are glueable in $\R^m$} if there exist isometric embeddings $f^{(1)} \colon\R^{n^{(1)}} \to \R^m,\dots, f^{(N)}\colon \R^{n^{(N)}} \to \R^m$ such that 
    \begin{align}
        \bigcup_{n = 1}^N f^{(n)}(\cY^{(n)}) \in \ps(\qsum nN I^{(n)},\R^m).
    \end{align}
    We call the tuple $(f^{(1)},\dots,f^{(N)})$ a \emph{gluing of $\cY^{(1)},\dots,\cY^{(N)}$}, or simply a gluing when $\cY^{(1)},\dots,\cY^{(N)}$ are clear from context.

    We say that $\cY^{(1)},\dots,\cY^{(N)}$ \emph{may be glued together} or \emph{are glueable} if there is some $m$ and $(f^{(1)},\dots,f^{(N)})$ so that $\cY^{(1)},\dots,\cY^{(N)}$ glue together in $\R^m$.

    When it exists, we denote the gluing of $\cY^{(1)},\dots,\cY^{(N)}$ by $\cY^{(1)}\vee \dots\vee\cY^{(N)}$.
\end{definition}
As noted in Remark \ref{rmk:nesc-for-j-geq-2-in-the-recoloring-trick}, $\cY \in \ps$ does not imply that $\abs{\cY} \in \ps$. Hence, not all equidistant spacings have glue sites. Glue sites are necessary to have any gluings, as the following proposition shows.
\begin{remark}[Interpretation of Centers and Glue Sites]
    \label{rmk:interpretation-of-centers-and-glue-sites}
    We may combine the idea of gluings and the recoloring trick in the following way. If equidistant spacings $\cY^{(1)},\dots,\cY^{(N)}$ glue together in $\R^m$, then they define a $\cY' \in \ps(N,\R^m)$ with classes
    \begin{align}
        Y_i \coloneqq \abs{\cY^{(i)}}.
    \end{align}
    That is, gluing equidistant spacings together `looks like' forming equidistant spacings, just `one level up.' From this perspective, we can see that the glue sites of $\cY^{(i)}$ are just the centers of $\abs{\cY^{(i)}}$, viewed as a single class in $\cY'$. 
\end{remark}

\begin{figure}
    \centering
    \begin{subfigure}{.32\linewidth}
        \centering
        \includegraphics[width=1.\linewidth]{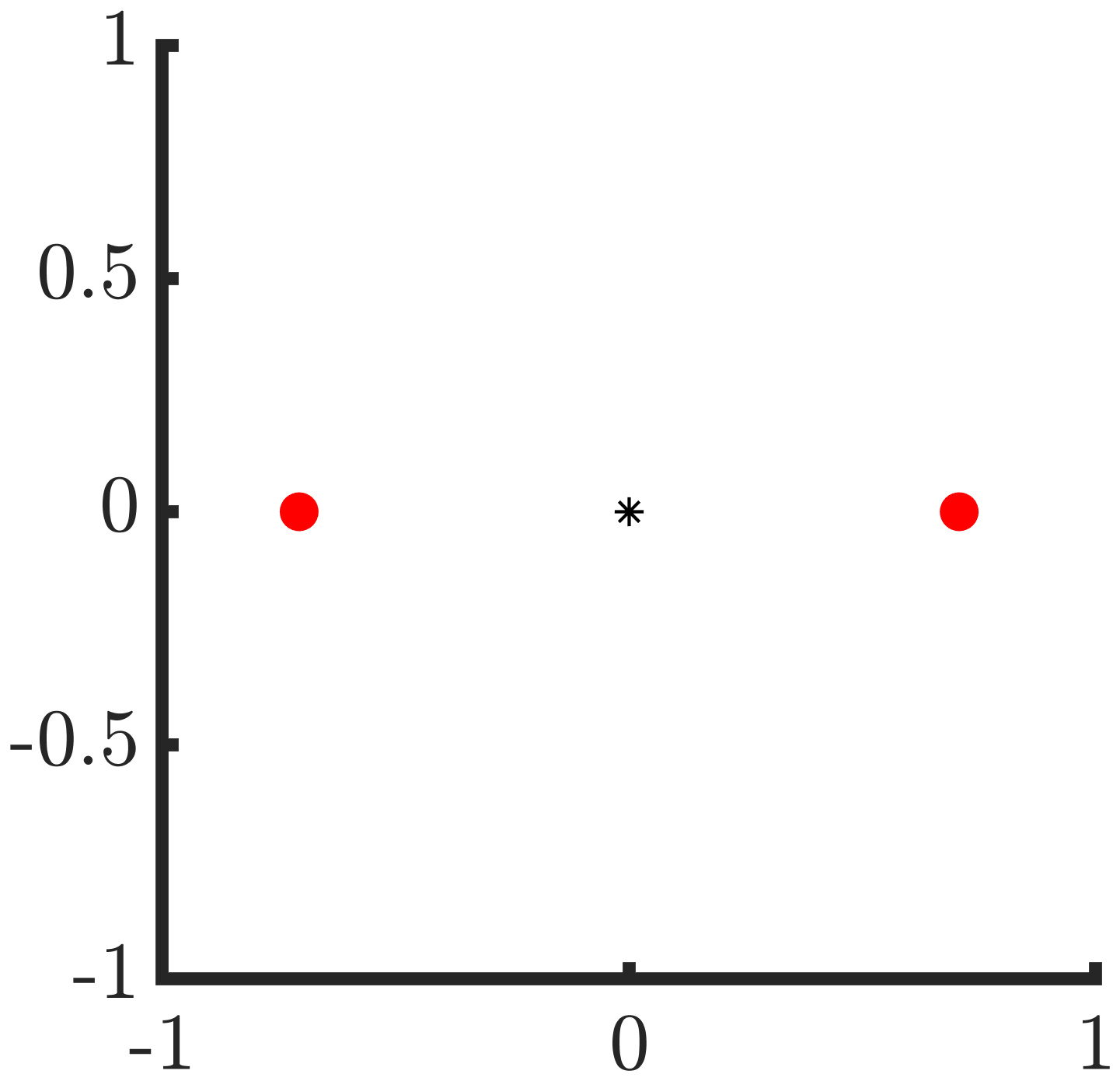}
        \subcaption{$\cY^{(1)}$}
        \label{fig:examp:gluings:Y1:2d}
    \end{subfigure}
    \begin{subfigure}{.32\linewidth}
        \centering
        \includegraphics[width=1.\linewidth]{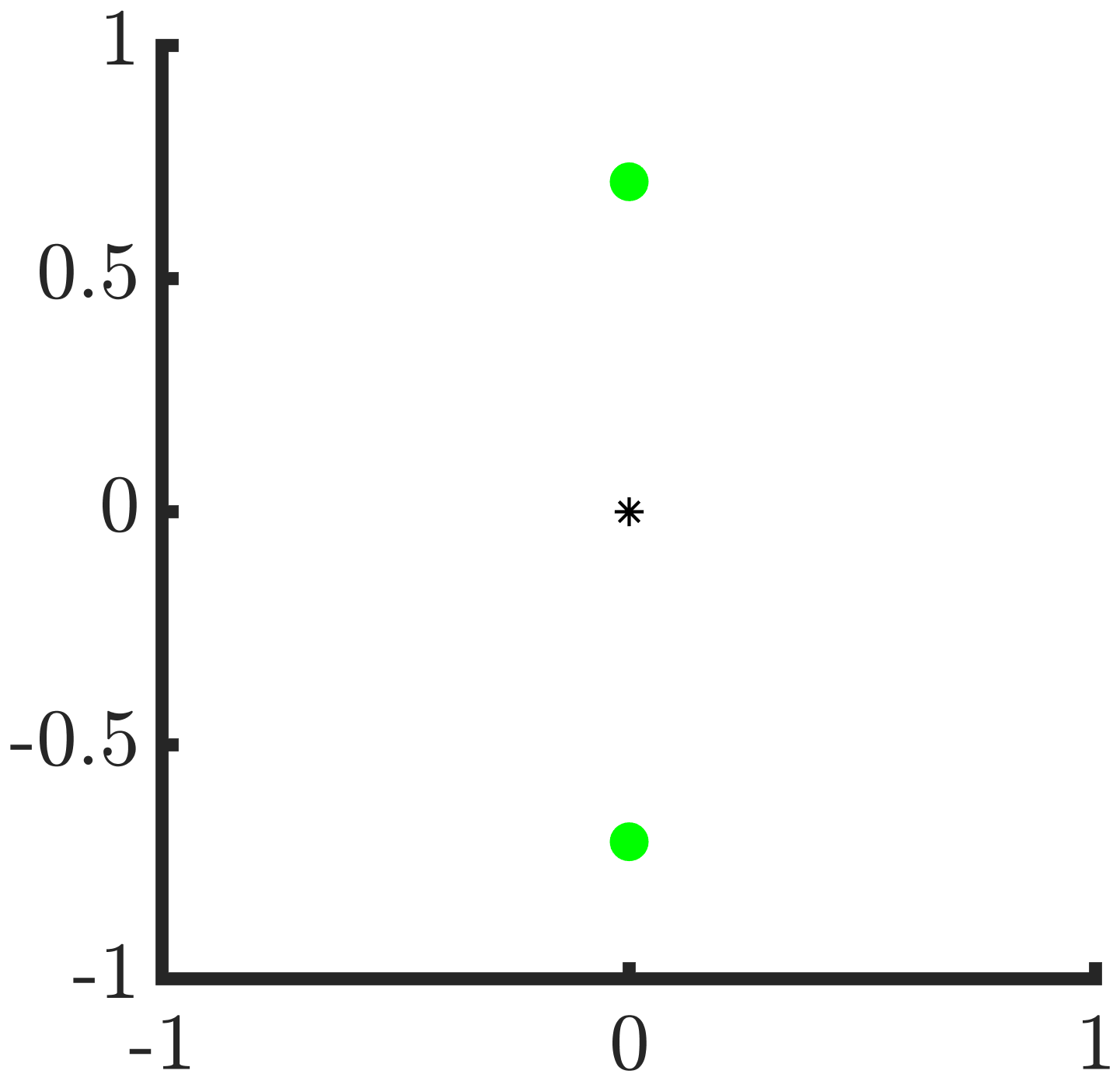}
        \subcaption{$\cY^{(2)}$}
        \label{fig:examp:gluings:Y2:2d}
    \end{subfigure}
    \begin{subfigure}{.32\linewidth}
        \centering
        \includegraphics[width=1.\linewidth]{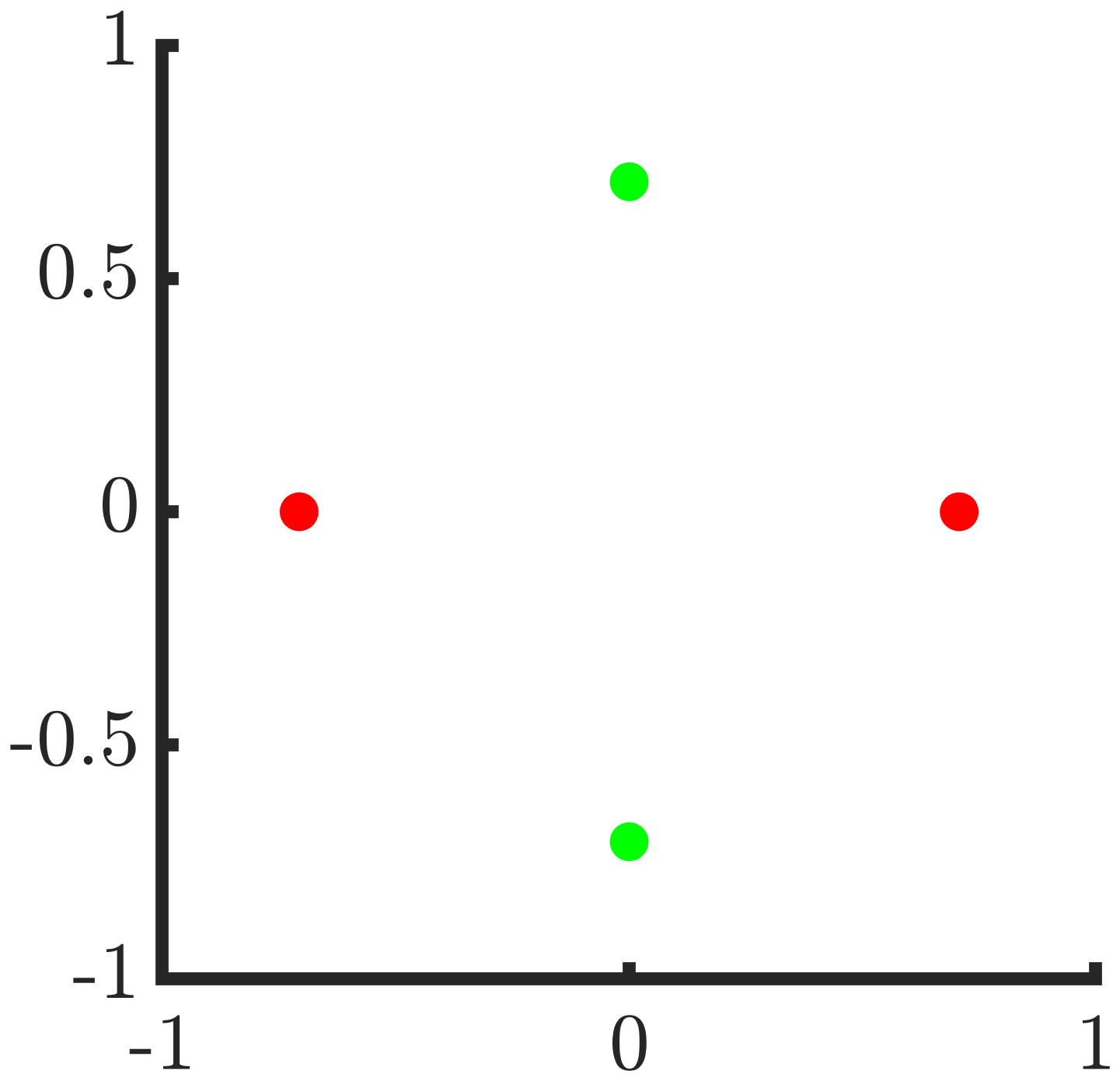}
        \subcaption{$\cY^{(1)}\vee \cY^{(2)}$}
        \label{fig:examp:gluings:Y1-vee-Y2:2d}
    \end{subfigure}\\
    \begin{subfigure}{.32\linewidth}
        \centering
        \includegraphics[width=1.\linewidth]{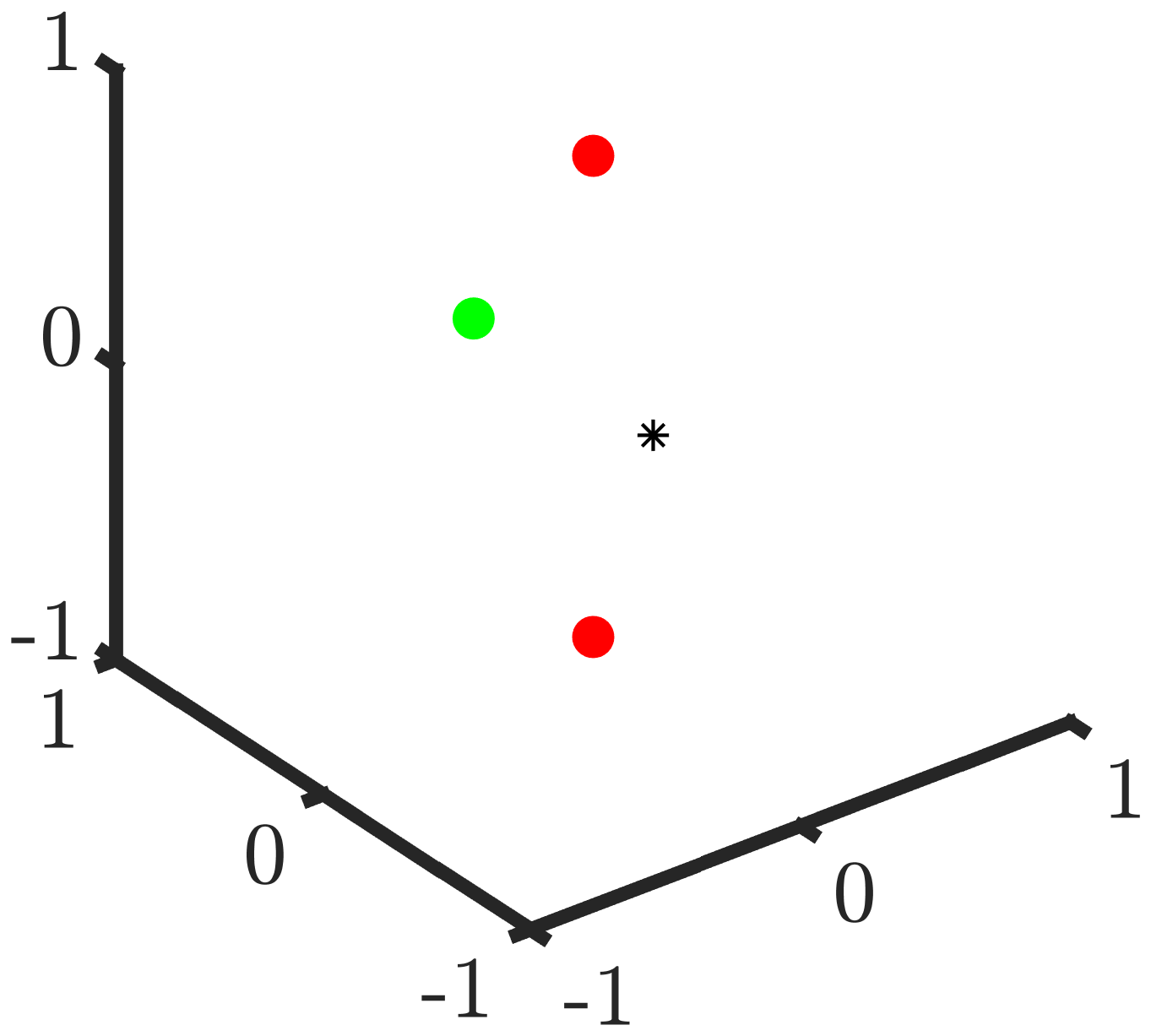}
        \subcaption{$\cY^{(3)}$}
        \label{fig:examp:gluings:Y3:3d}
    \end{subfigure}
    \begin{subfigure}{.32\linewidth}
        \centering
        \includegraphics[width=1.\linewidth]{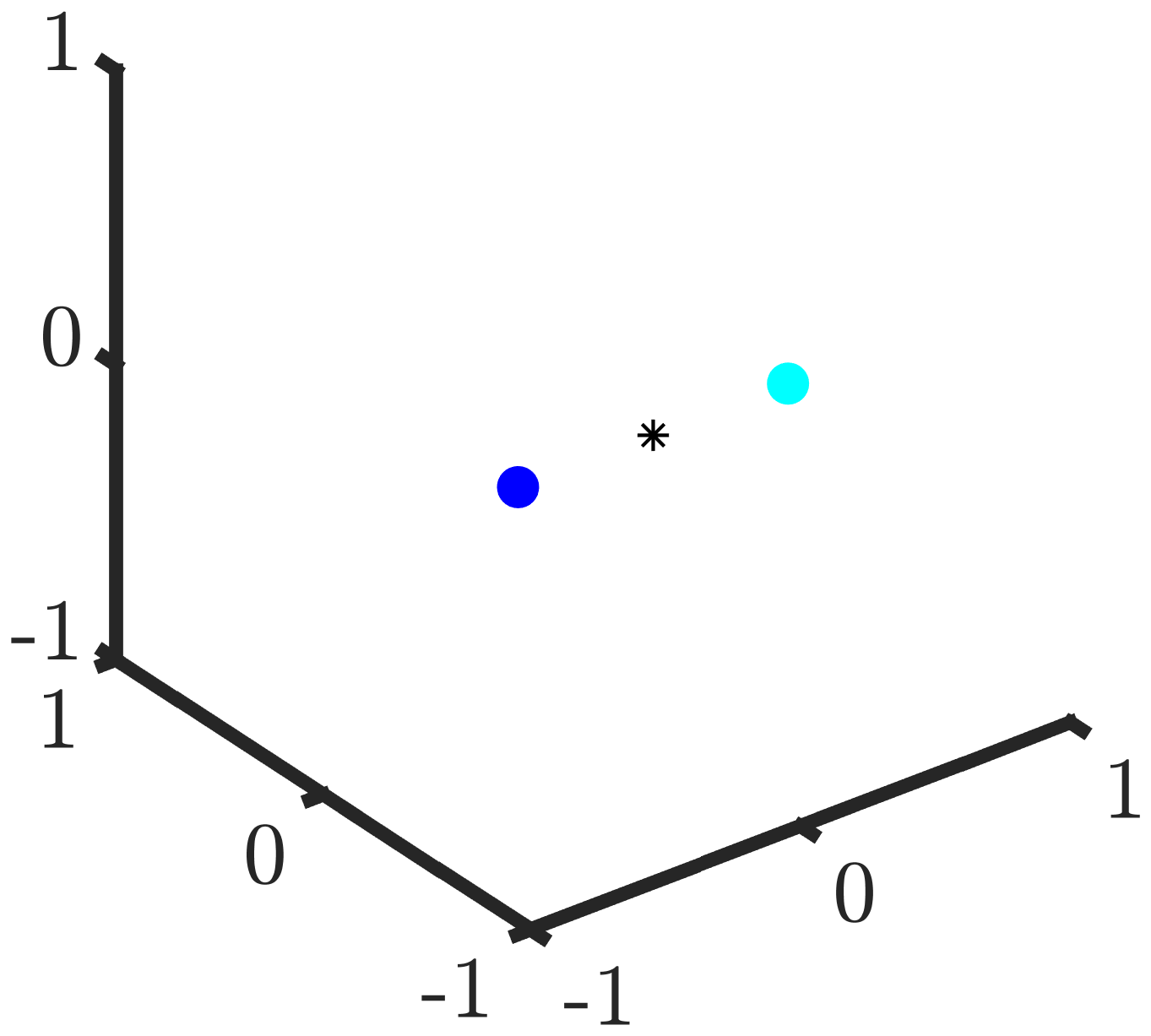}
        \subcaption{$\cY^{(4)}$}
        \label{fig:examp:gluings:Y4:3d}
    \end{subfigure}
    \begin{subfigure}{.32\linewidth}
        \centering
        \includegraphics[width=1.\linewidth]{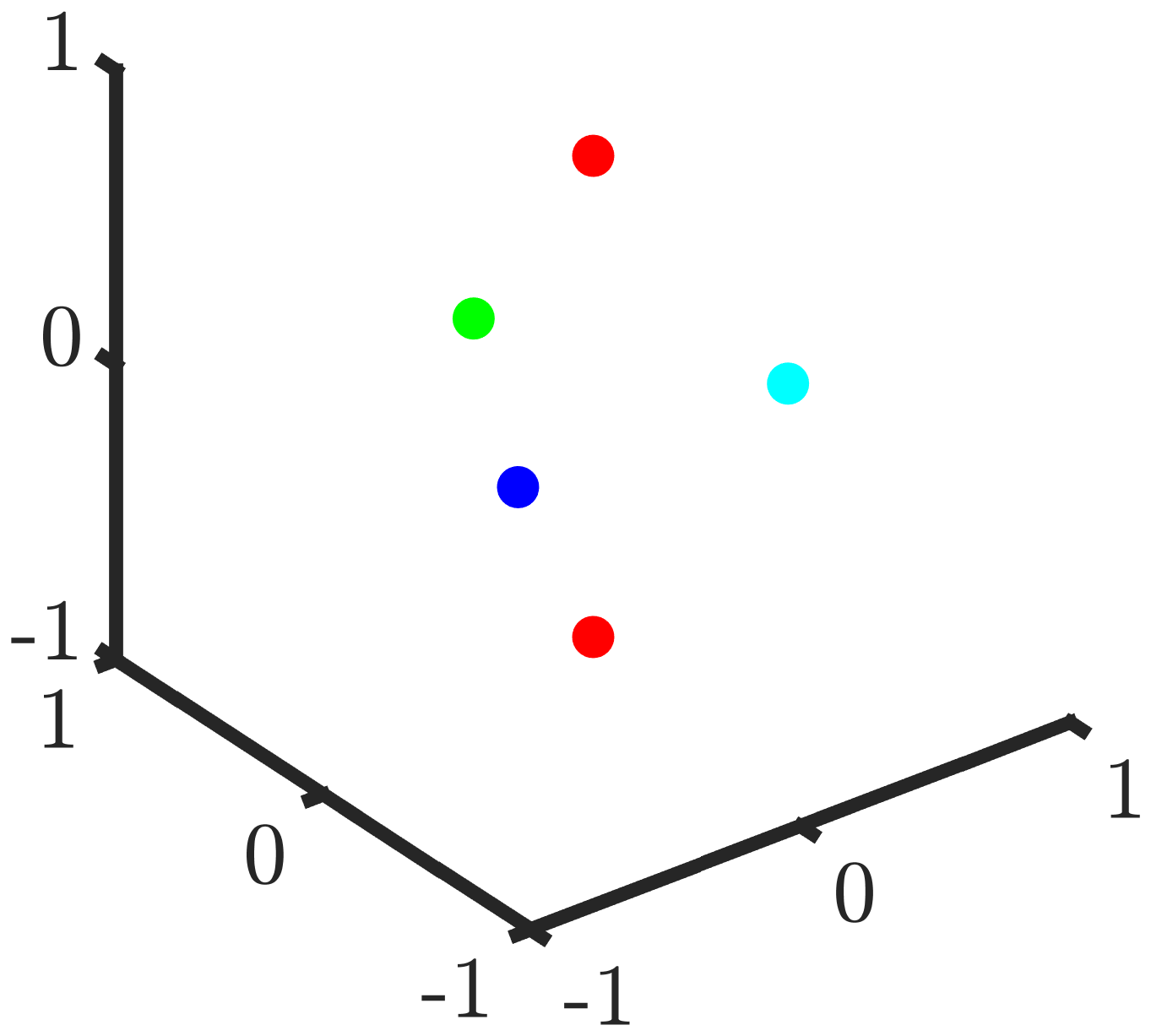}
        \subcaption{$\cY^{(3)}\vee \cY^{(4)}$}
        \label{fig:examp:gluings:Y3-vee-Y4:3d}
    \end{subfigure}\\
    \caption{Examples of gluings with glue sites shown with a black $*$. Subfigures \ref{fig:examp:gluings:Y1:2d} - \ref{fig:examp:gluings:Y1-vee-Y2:2d} show gluings in $\R^2$, and subfigures \ref{fig:examp:gluings:Y3:3d} - \ref{fig:examp:gluings:Y3-vee-Y4:3d} show gluings in three dimensions.}
    \label{fig:examp:gluings}
\end{figure}

\begin{example}[Example of Gluings]
    \label{examp:gluings}

    An example of a gluing is given by $\cY^{(1)},\cY^{(2)} \in \ps(1,\R^2)$ where
    \begin{align}
        Y^{(1)}_1 \coloneqq \set{\paren{\frac{\sqrt 2}2,0},{\paren{\frac{-\sqrt 2}2,0}}} \quad Y^{(2)}_1 \coloneqq \set{\paren{0,\frac{\sqrt 2}2},{\paren{0,\frac{-\sqrt 2}2}}}.
    \end{align}
    These two spacings glue together to form $\cY^{(1)} \vee \cY^{(2)}$ where $f^{(1)} = f^{(2)} = I$ with glue sites at the origin.
    $\cY^{(1)}$ and $\cY^{(2)}$ is illustrated by the red and green points respectively in Subfigures \ref{fig:examp:gluings:Y1:2d} and\ref{fig:examp:gluings:Y1:2d}. The glue sites are shown with black $*$. The gluing $\cY^{(1)} \vee \cY^{(2)}$ is illustrated in \ref{fig:examp:gluings:Y1-vee-Y2:2d}.

    Another example of gluings is given by $\cY^{(3)},\cY^{(4)} \in \ps(2,\R^3)$ where
    \begin{align}
        &Y^{(3)}_1 \coloneqq \set{\paren{0,0,\frac{\sqrt 6}3},\paren{0,0,\frac{-\sqrt 6}3}},\quad &Y^{(3)}_2 \coloneqq\set{{\paren{0,\frac{\sqrt 3}3,0}}}\nonumber\\ &Y^{(4)}_1 \coloneqq \set{\paren{\frac{-1}2,\frac{-1}{\sqrt{12}},0}},\quad &Y^{(4)}_2 \coloneqq\set{{\paren{\frac{1}2,\frac{-1}{\sqrt{12}},0}}}.
    \end{align}
    These two spacings glue together to form $\cY^{(3)} \vee \cY^{(4)}$ where $f^{(3)} = f^{(4)} = I$ with glue sites at the origin.
    $\cY^{(3)}$ and $\cY^{(4)}$ is illustrated by the red and green points respectively in Subfigures \ref{fig:examp:gluings:Y3:3d} and\ref{fig:examp:gluings:Y4:3d}. The glue sites are shown with black $*$. The gluing $\cY^{(3)} \vee \cY^{(4)}$ is illustrated in \ref{fig:examp:gluings:Y3-vee-Y4:3d}

    An example of equidistant spacings that have glue sites but fail to glue is given by $\cY^{(5)},\cY^{(6)}\in\ps(1,\R^2)$ where
    \begin{align}
        Y^{(5)}_1 \coloneqq \set{\paren{1,0},{\paren{-1,0}}} \quad Y^{(6)}_1 \coloneqq \set{\paren{0,1},{\paren{0,-1}}}.
    \end{align}
    Both $\cY^{(5)}$ and $\cY^{(6)}$ have glue sites $c^{*(5)} = c^{*(6)} = \bszero$. Yet, any candidate gluing $f^{(5)},f^{(6)}$ would satisfy $\norm{f^{(5)}(y_5) - f^{(6)}(y_6)} = 1$ but
    \begin{align*}
        \norm{f^{(5)}(y_5) - f^{(6)}(y_6)}^2 &= \norm{f^{(5)}(y_5) - f^{(5)}(c^{*(5)})}^2 + \norm{f^{(5)}(c^{*(5)}) - f^{(6)}(c^{*(6)})}^2 + \norm{f^{(6)}(c^{*(6)}) - f^{(6)}(y_6)}^2\\
        &> 2.
    \end{align*}
    So $\cY^{(5)}$ and $\cY^{(6)}$ do not glue.
\end{example}

Remark \ref{examp:gluings} shows that for $\cY^{(1)},\dots,\cY^{(N)}$ to be glueable in $\R^m$, each of the $\cY^{k}$ must have a glue site, and the $m$ must be large enough, and each of the spacings must not be `too large.' To make precise this notion of `largeness' we introduce the concept of extent.

\begin{definition}[Extent]
    \label{def:extent}
    For each $I,n \in \bbN$, we define the \emph{extent}%
    , denoted by $\extent \colon \ps(I,\R^n) \to \R \cup \set{\infty}$ %
    as
    \begin{align}
        \label{eqn:extent-def}
        \extent \cY \coloneqq \twopartpiecewise{\infty}{ \text{ $\cY$ has no glue site}}{\norm{y - c^*}}{\text{where $c^*$ is the glue site of $\cY$ and }y \in \abs\cY}%
    \end{align}
    If the set in the r.h.s. of Eqn. \ref{eqn:extent-def} is empty, then we define $\extent \cY = \infty$.
    
\end{definition}
The concept of extent is so called because it measures both the size of a equidistant spacings, and also the degree to which it is glueable as the following proposition shows.
\begin{proposition}[Elementary Properties of Extent]
    \label{prop:elementary-properties-of-extent}
    Let $\cY \in \ps(I,\R^n)$. Then there is a $\cY' \in \ps$ so that $\cY$ and $\cY'$ glue if and only if $\extent \cY \leq 1$.
\end{proposition}

\begin{proof}
    Let $\cY \in \ps(I,\R^n)$ and $\extent \cY \leq 1$. $\cY$ has a glue site $c$. Then, define 
    \begin{align}
        \cY' \coloneqq \set{(\bszero^{n}, \pm\sqrt{1 - r^2})}
    \end{align}
    where $r \coloneqq \norm{y - c^*}$. Clearly, $\cY' \in \ps(1)$, and $\cY$ and $\cY'$ glue via the mapping $(x \mapsto (x - c,0), Id_{n+1})$. Hence, $\extent \cY \leq 1$ implies that there is a $\cY'$ so that $\cY$ and $\cY'$ glue.

    Suppose that $\cY$ and $\cY'$ glue. Let $y \in \abs{\cY}$, $y' \in \abs{\cY'}$, $c$ denote the glue site of $\cY$ and $c'$ denote the glue site of $\cY'$. Then, by Proposition \ref{prop:tri-part-ortho}, Remark \ref{rmk:interpretation-of-centers-and-glue-sites}, and pythagoreans theorem
    \begin{align}
        1 = \norm{y - y'}^2 = \norm{y - c}^2 + \norm{c - c'}^2 + \norm{c' - y'}^2 \geq \norm{y - c}^2 = \extent \cY^2.
    \end{align}
    Hence, $\extent \cY \leq 1$.

\end{proof}

The extent also has a concrete use. It gives simple conditions under which two equidistant spacings may be glued together, as shown in the following proposition.
\begin{proposition}[Extent Characterizes Gluablity]
    \label{prop:extent-characterizes-gluablity}
    Let $\cY \in \ps(I,\R^n)$ and $\cY' \in \ps(I',\R^{n'})$.
    
    If
    \begin{align}
        \extent \cY ^2 + \extent \cY'{}^2 \leq 1
    \end{align}
    then $\cY$ and $\cY'$ glue together in $\R^m$ where $m = n + n' + 1$. Further, if $\extent \cY ^2 + \extent \cY'{}^2 = 1$. then we may take $m = n + n'$.

    If $\cY$ and $\cY'$ glue together in $\R^m$, then 
    \begin{align}
        \extent \cY ^2 + \extent \cY'{}^2 \leq 1.
    \end{align}
    Further, if $n$ (resp. $n'$) are the minimal dimensions so that $\cY$ (resp. $\cY'$) embeds into $\R^n$ (resp. $\R^{n'}$) and $\cY$ and $\cY'$ glue together in $\R^{n + n'}$ then $\extent \cY ^2 + \extent \cY'{}^2 = 1$.
\end{proposition}
\begin{proof}
    First we prove that $\extent \cY ^2 + \extent \cY'{}^2\leq 1$ implies that $\cY$ and $\cY'$ are glueable. Both $\cY$ and $\cY'$ have glue sites, denoted $c^*$ and $c^*{}'$ respectively by Proposition \ref{prop:elementary-properties-of-extent} point 4.
    
    Define $f \colon \R^n \to \R^n \times \R^{n'} \times \R$ as
    \begin{align}
        f(x) \coloneqq (x - c^*, \bszero^{n'},0)
    \end{align}
    and $f' \colon \R^{n'} \to \R^n \times \R^{n'} \times \R$ be defined as
    \begin{align}
        f'(x) \coloneqq (\bszero^{n},x' - c'{}^*,\alpha)
    \end{align}
    where $\alpha \coloneqq \sqrt{1 - \extent \cY^2 - \extent \cY'^2}$. We claim that $f$ and $f'$ constitute a gluing of $\cY$ and $\cY'$ in $\R^{n+n'+1}$. Clearly, $\norm{y_i - y_j} = 1$ if $y_i$ and $y_j$ are both points in $\cY$ or $\cY'$. What's left to check is the case when $y_i \in \cY$ and $y_j' \in \cY'$. Then
    \begin{align}
        \norm{y_i - y_j'}^2 &= \norm{y_i - c^*}^2 + \alpha^2 + \norm{y_j' - c^*{}'}^2\\
        &= \extent \cY ^2 + (1 - \extent \cY^2 - \extent \cY'{}^2) + \extent \cY'{}^2\\
        &= 1.
    \end{align}
    Hence, $f(\cY) \cup f'(\cY') \in \ps(I + I',\R^{n + n' + 1})$. 

    If $\extent \cY^2 + \extent \cY'{}^2 = 1$, then $\alpha = 0$ and  $f(\cY)\cup f'(\cY') \subset \R^n \times \R^{n'} \times \set{0}$, so we may `trim' the extra dimension.

    Now suppose that $\cY$ and $\cY'$ glue together in $\R^m$. Via the recoloring trick of Lemma \ref{lem:recoloring-trick}, we may realize that the glue sites $c^*$ and $c'{}^*$ of $\cY$ and $\cY'$ are centers of an element of $\ps(2,\R^m)$. Hence, via Proposition \ref{prop:tri-part-ortho}, for all $y \in \abs\cY$ and $y' \in \abs{\cY'}$, we have that 
    \begin{align}
        \label{eqn:tri-part-orth-applied-to-glueibility}
        1 &= \norm{y - c^*}^2 + \norm{c^* - c'{}^*}^2 + \norm{c'{}^* - y'}^2
    \end{align}
    which implies that $\norm{y - c^*}^2 + \norm{c'{}^* - y'}^2 = \extent \cY^2 + \extent \cY'{}^2 \leq 1$.

    Suppose that $n$ and $n'$ are minimal. By the recoloring trick again, and Proposition \ref{prop:tri-part-ortho} we have that $\aff(\cY) \perp \aff(\cY')$, and that $c^* - c'{}^* \in \paren{\aff(\cY) \oplus \aff(\cY')}^\perp = \R^{n + n'}$ and so $c^* - c'{}^* = \bszero$. Therefore, Eqn. \ref{eqn:tri-part-orth-applied-to-glueibility} reads $1 = \norm{y - c^*}^2 + 0 + \norm{c'{}^* - y'}^2 = \extent \cY^2 + \extent \cY'{}^2$.
\end{proof}

\begin{corollary}[Gluing Spacings in the Same Space]
    \label{cor:gluing-spacings-in-the-same-space}
    Let $\cY \in \ps(I,\R^n)$, $\cY' \in \ps(I',\R^n)$. Then, $\cY$ and $\cY'$ glue together in $\R^n$ if and only if $\extent \cY^2 + \extent \cY'{}^2 \leq 1$ and 
    \begin{align}
        \dim \aff \abs{\cY} +\dim \aff \abs{\cY'} \leq \twopartpiecewise{n}{\extent \cY^2 + \extent \cY'{}^2 = 1}{n-1}{\extent \cY^2 + \extent \cY'{}^2 < 1}
    \end{align}
\end{corollary}
\begin{proof}
    For the reverse direction we may embed $\cY$  into $\R^{n_1}$ if and only if $n_1  \geq d \coloneqq \dim \aff \abs{\cY}$, and $\cY'$  into $\R^{n_2}$ if and only if $n_2  \geq d' \coloneqq \dim \aff \abs{\cY'}$. So, by Proposition \ref{prop:extent-characterizes-gluablity} applied to $\cY$'s embedding in $\R^{d}$ and $\cY'$'s embedding in $\R^{d'}$, $\cY$ and $\cY'$ are glueable in $\R^{d + d'}$ if $\extent \cY ^2 + \extent \cY'{}^2 = 1$ and are glueable in $\R^{d+d'+1}$ if $\extent \cY ^2 + \extent \cY'{}^2 < 1$.

    To prove the forward direction, it follows from observing that Eqn. \ref{eqn:tri-part-orth-applied-to-glueibility} holds for all glueable spacings (recall $y - c^*$, $c^* - c'{}^*$ and $c'{}^* - y'$ are all orthogonal). If $\norm{c^* - c'{}^*}^2 = 1 - \extent \cY^2 - \extent \cY'{}^2 = 0$, then $c^* - c'{}^* = \bszero$, and we may take $n \geq \dim \aff \abs{\cY} +\dim \aff \abs{\cY'}$. Otherwise, we need one extra orthogonal dimension for $c^* - c'{}^*$, so $n \geq \dim \aff \abs{\cY} +\dim \aff \abs{\cY'} + 1$.
\end{proof}

The following proposition gives a convenient expression for $\extent \cY \vee \cY'$ in terms of $\extent \cY$ and $\extent \cY'$.

\begin{proposition}[Extent of Glueable Equidistant Spacings]
    \label{prop:extent-of-glueable-equidistant-spacings}
    Let $\cY$ and $\cY'$ glue together. If $\extent \cY \vee \cY' < \infty$, then
    \begin{align}
        \label{eqn:extent-of-a-gluing}
        \extent \paren{\cY \vee \cY'} = \frac{\sqrt{1 - 4 \extent \cY ^2 \extent \cY'{}^2}}{2\sqrt{1 - \extent \cY^2 - \extent \cY'{}^2}}.
    \end{align}

\end{proposition}

\begin{proof}
    Let us denote by $c^{\vee*}$ the glue site of $\cY \vee \cY'$. By orthogonality, we have that
    \begin{align}
        \label{eqn:glued-extent-calc:1}
        \paren{\extent \cY \vee \cY'}^2 = \norm{y - c^*}^2 + \norm{c^* - c^{\vee *}}^2 = \norm{y' - c'{}^*}^2 + \norm{c'{}^* - c^{\vee *}}^2
    \end{align}
    By the recoloring trick, $c^{\vee *}$ lies in the line connecting $c^*$ and $c'{}^*$. Hence there is a $\lambda \in [0,1]$ so that 
    \begin{align}
        \label{eqn:glued-extent-calc:1.5}
        c^{\vee*} = (1 - \lambda) c^* + \lambda c'{}^*.
    \end{align}
    By calculation, $\norm{c^* - c^{\vee *}} = \lambda \norm{c^* - c'{}^*}$ and $\norm{c'{}^* - c^{\vee *}} = (1 - \lambda) \norm{c^* - c'{}^*}$. Substituting this into Eqn. \ref{eqn:glued-extent-calc:1} shows that $\lambda$ satisfies
    \begin{align}
        \norm{y - c^*}^2 + \lambda^2\norm{c^* - c'{}^*}^2 &= \norm{y' - c'{}^*}^2 + (1 - \lambda)^2\norm{c^* - c'{}^*}^2\\
        \label{eqn:glued-extent-calc:2}
       \alpha + \lambda^2 \gamma &= \beta + (1 - \lambda)^2\gamma
    \end{align}
    where we have denoted $\alpha \coloneqq \norm{y - c^*}^2, \beta \coloneqq \norm{y' - c'{}^*}^2$, and $\gamma \coloneqq \norm{c^* - c'{}^*}^2$.

    Solving for $\lambda$ yields 
    \begin{align}
        \label{eqn:glued-extent-calc:2.5}
        \lambda = \frac{\beta - \alpha + \gamma}{2\gamma}.
    \end{align}

    Now we show Eqn. \ref{eqn:extent-of-a-gluing}. Combining Eqns. \ref{eqn:glued-extent-calc:1}, \ref{eqn:glued-extent-calc:2} and \ref{eqn:glued-extent-calc:2.5} yields
    \begin{align}
        \label{eqn:glued-extent-calc:3}
        \paren{\extent \cY \vee \cY'}^2 = \alpha + \lambda ^2 \gamma = \alpha + \frac{\paren{\beta - \alpha + \gamma}^2}{4\gamma}.
    \end{align}
    Observing that $1 = \alpha + \gamma + \beta$ and simplifying Eqn. \ref{eqn:glued-extent-calc:3} with this identity, we have that %
    \begin{align}
        \paren{\extent \cY \vee \cY'}^2 &= \frac{4\alpha(1 - \alpha - \beta) + \paren{1 - 2\alpha}^2}{4\paren{1 - \alpha - \beta}}\\
        &=\frac{1 - 4\alpha \beta}{4\paren{1 - \alpha - \beta}}.
    \end{align}
    Finally, by noting that $\alpha = \extent \cY^2, \beta = \extent \cY'{}^2$ we have that 
    \begin{align}
        \paren{\extent \cY \vee \cY'}^2 = \frac{1 - 4 \extent \cY^2 \extent \cY'{}^2}{4 \paren{1 - \extent \cY^2 - \extent \cY'{}^2}}
    \end{align}
\end{proof}

\begin{corollary}[Glue Site of Gluing]
    \label{cor:glue-site-of-gluing}
    If $\cY$ and $\cY'$ are glueable, the glue site of $\cY \vee \cY'$ is 
    \begin{align}
        \frac12\paren{c^* + c'{}^*} + \delta\paren{c^* - c'{}^*}
    \end{align}
    where $c^*$ and $c'{}^*$ are the glue sites of $\cY$ and $\cY'$ respectively and 
    \begin{align}
        \delta \coloneqq \frac{\extent^2 \cY - \extent^2 \cY'}{2\paren{1 - \extent^2 \cY - \extent^2 \cY'}}.
    \end{align}
\end{corollary}

\begin{proof}
    Using the same notation as in the proof of Proposition \ref{prop:extent-of-glueable-equidistant-spacings}, Equation \ref{eqn:glued-extent-calc:2.5} reads
    \begin{align}
        \lambda = \frac{\extent^2 \cY' - \extent^2 \cY + 1 - \extent^2 \cY - \extent^2 \cY'}{2(1 - \extent^2 \cY + \extent^2 \cY')}.
    \end{align}
    Using the formula \ref{eqn:glued-extent-calc:1.5}, the glue site of $\cY \vee \cY'$ is
    \begin{align}
        (1 - \lambda) c^* + \lambda c'{}^* = \frac12\paren{c^* + c'{}^*} + \frac{\extent^2 \cY - \extent^2 \cY'}{2\paren{1 - \extent^2 \cY - \extent^2 \cY'}}\paren{c^* - c'{}^*}.
    \end{align}
\end{proof}

\begin{corollary}[Extent of Equilateral Equidistant Spacing]
    \label{cor:extent-of-equidistant-spacing}
    Let $\cY \in \ps(n+1, \R^n)$ be the equidistant spacing given by coloring each vertex of $E_{1,n}$ a distinct color. Then $\extent(\cY) = \sqrt{\frac{n}{2(n+1)}}$.
\end{corollary}
\begin{proof}
    The proof follows from Proposition \ref{prop:circumcenter-of-equi-simplex}.
\end{proof}

Now we present two related calculations that give concrete formulae for where the glue site of a gluing lies, given the constituent glue sites.

\begin{proposition}[Calc I]
    \label{prop:calc-1}
    Let $\cY,\cY' \in \ps(I,\R^n)$ glue together and $c^* = \bszero^n$, $c'{}^* = (\bszero^{n-1},\gamma)$ where $\gamma = \sqrt{1 - \extent^2 \cY - \extent^2 \cY'}$. Then 
    \begin{align}
        c^\vee{}^* = \paren{\bszero^{n-1},\frac{1 - 2\extent^2 \cY}{2\sqrt{1 - \extent^2\cY - \extent^2\cY'}}}.
    \end{align}
\end{proposition}

\begin{proposition}[Calc. II]
    \label{prop:calc-2}
    Let $\cY,\cY' \in \ps(I,\R^n)$, $c^*,c^*{}' \in \bszero^{n-1}\times \R$ and $c^{*\vee} = \bszero^n$. Then
    \begin{align}
        c^* = (\bszero^{n-1},\mp\frac{1 - 2\extent^2\cY}{2\sqrt{1 - \extent^2\cY - \extent^2\cY'}}), \quad c^*{}' =(\bszero^{n-1},\pm\frac{1 - 2\extent^2\cY'}{2\sqrt{1 - \extent^2\cY - \extent^2\cY'}}).
    \end{align}
\end{proposition}

\begin{proof}
    Let $c^{(1)*} \coloneqq \bszero^n, c^{(1)}{}'{}^{*} \coloneqq \paren{\bszero^{n-1},\gamma}$ and $c^{(1)\vee*} \coloneqq \paren{\bszero^{n-1},\frac{1 - 2\extent^2 \cY}{2\sqrt{1 - \extent^2\cY - \extent^2\cY'}}}$. These are the centers of $\cY,\cY$ and $\cY\vee\cY'$ respectively by Proposition \ref{prop:calc-1}. Then the expressions for $c^*, c'{}^*$, and $c^{\vee*}$ come from evaluating
    \begin{align}
        c^* = c^{(1)} \pm c^{(1)\vee*}, \quad c'{}^* = c'{}^{(1)} \pm c^{(1)\vee*}, \quad c^{\vee*}  = c^{(1)\vee*} \pm c^{(1)\vee*}.
    \end{align}
\end{proof}

\subsection{Isometric Embeddings}
\label{sec:isometric-embeddings}

Equidistant spacings have two types of isometries. First, are Euclidean isometries (i.e. isometries that extend to the ambient space). Second are what we call a squash and stretch isometry that does not typically extend to an isometry on $\R^n$. This section is spent understanding the squash and stretch isometry.

\begin{definition}[Unitary Rototranslation]
    \label{def:unitary-isometry}
    Let $\cY\in\ps(\R^n,I)$, $\cY' \in \ps(\R^n,I')$. We say that $\cY$ and $\cY'$ are \emph{trivially isomorphic} or \emph{isomorphic by $U$} if $I = I'$ and there is a $U \in E(n)$ so that for each $i = 1,\dots, I$,
    \begin{align}
        UY'_i = Y_i,
    \end{align}
    where $E(n)$ denotes the Euclidean group over $\R^n$.
\end{definition}
Figure \ref{fig:iso-unitary:1} and \ref{fig:iso-unitary:2} shows an example of two equidistant spacings that are trivially isomorphic. 
Figure \ref{fig:iso-unitary:1} and \ref{fig:iso-unitary:2} show an example of two equidistant spacings that are trivially isomorphic.

The final isometry focuses on the role of $r$ in Corollary \ref{cor:characterization-of-equidistant-spacings}. %
\begin{definition}[Squash and Stretch]
    \label{def:squash-stretch}
    Let $\cY \in \ps(\R^n,I)$. %
    A map $f \colon \R^n \to \R^n$ is called a  \emph{squash and stretch of $\cY$} if it satisfies the following properties.
    \begin{enumerate}
        \item The map 
        \begin{align}
            f|_{\bigcup_{i = 1}^I \set{c_i}}\colon \bigcup_{i = 1}^I c_i \to \R^n
        \end{align}
        is injective.
        \item For each $i = 1,\dots,I$, the map $d \colon V_i \to \R^n$ is defined by
        \begin{align}
            \label{eqn:squash-stretch:dilation-map}
            d(x) \coloneqq f(x + c_i) - f(c_i) %
        \end{align}
        is a dilation of $V_i$ about $c_i$ by factor $\alpha \neq 0$.
    \end{enumerate}

    We say that $\cY$ and $\cY'$ are \emph{squash and stretch isomorphic} or \emph{isomorphic by squash and stretch $f$} if $I = I'$ and there is a squash and stretch $f$ of $\cY$ so that for each $i = 1,\dots,I$,
    \begin{align}
        fY_i = Y_i'.
    \end{align}

\end{definition}
The name squash and stretch comes from the animation principle of the same name. In animation, the principle of squash and stretch refers to the act of stretching an object (e.g. a cartoon character) along one axis, while squashing it along a perpendicular axis to emphasize movement or momentum \cite{thomas1995illusion}. The use here is similar. The map $f$ has the effect of `squashing'  one class while `stretching' another. The definition of squash and stretch isometry leaves open the possibility that $\cY$ could be squash and stretch isomorphic to $\cY'$, but not the other way around. This is, however, not the case as the following proposition shows.
\begin{proposition}[Properties of Squash and Stretch]
    \label{prop:prop-of-squash-stretch}
    Let $\cY \in \ps$ and $f$ be a squash and stretch of $\cY$. Then the following holds.
    \begin{enumerate}
        \item For each $i = 1,\dots,I$, the set $f(A_i)$ is a translation of $A_i$, and $f|_{A_i}$ is a bijection.
        \item If $f\cY \in \ps$, then $f|_{\bigcup_{i = 1}^IY_i}$ is injective.
        \item If $f$ is a squash and stretch of $\cY$, and $\cY' \in \ps$ so that $fY_i = Y_i'$, then there is a squash and stretch $g$ of $\cY'$, so that 
        \begin{align}
            gY_i' = Y_i.
        \end{align}
    \end{enumerate}
\end{proposition}

\begin{proof}
    First, we show point 1. Recall that $A_i = c_i + V_i$ from Corollary \ref{cor:characterization-of-equidistant-spacings}. Making the substitution $\tilde x \coloneqq x - c_i$ in Eqn. \ref{eqn:squash-stretch:dilation-map}, the equation becomes
    \begin{align}
        d(x) = f(x + c_i) - f(c_i) = \tilde f(\tilde x) - \tilde f(0)
    \end{align}
    where $\tilde f \colon V_i \to \R^n$. Because $d(x)$ is a dilation, $\tilde f(\tilde x) - \tilde f(0) \subset V_i$ and $\tilde f(V_i) - \tilde f(0) = V_i$ it is clear that $\tilde f$ acts on $V_i$ by dilation plus a shift by $\tilde f(0) = f(c_i)$. That is, $\tilde f \colon V_i \to f(c_i) + V_i$ is a bijection. Hence, $f|_{A_i}\colon A_i \to f(c_i) + V_i$ is a bijection. 

    Now we show point 2. If $y_i,y'_i \in Y_i$, and $f(y_i) = f(y'_i)$ then $y_i = y'_i$ is point 1. What's left is to show $f(y_i) = f(y_j)$ implies $y_i = y_j$ when $i \neq j$. Suppose that 
    $y_i \in Y_i$ and $y_j \in Y_j$ where $i \neq j$, and suppose that $f(y_i) = f(y_j)$, then $y_j = \proj_{A_i}y_j = c_i$ and $y_i = \proj_{A_j}y_i = c_j$ by Theorem \ref{thm:projection-theorem}. Thus, by Property 1 of squash and stretch isometries, we have that $y_i = y_j$.

    Now we show point 3. By point 2, $f|_{\bigcup_{i = 1}^I Y_i}$ has a unique inverse map on its range, which we denote $g$. What's left to show is that $g$ is a squash and stretch. 

    It is obvious that $g$ satisfies Def. \ref{def:squash-stretch} point 1. In light of Proposition \ref{prop:prop-of-squash-stretch} point 1, we see that the $d$ in Eqn. \ref{eqn:squash-stretch:dilation-map} is really a map $d \colon V_i \to V_i$ and a dilation by factor $\alpha \neq 0$. Such a dilation is invertible, and the inverse is also a dilation. Because $g$ is the inverse of $f$, it follows that $g$ must dilate $f(c_i) + V_i$ by factor $\alpha^{-1} \neq 0$. Hence, $g$ is a squash and stretch.

\end{proof}

\begin{remark}[Notational Shorthand: III]
    Given an $\cY \in \ps(\R^n,I)$, we use notation $f \cY$ to denote
    \begin{align}
        f \cY \coloneqq \bigsqcup_{i = 1}^I f(Y_i).
    \end{align}

    We will call a squash and stretch of $\cY$ simply a squash and stretch when the $\cY$ is clear from context.
    
    A squash and stretch of $\cY$ does not need to be invertible on all of $\R^n$, but is invertible on $f\cY$, %
    as shown by Proposition \ref{prop:prop-of-squash-stretch}. Hence, when $\cY$ is clear, we denote by $f^{-1}$ the squash and stretch that sends $f\cY$ to $\cY$.
\end{remark}

In general, if $\cY\in\ps$ then there is no reason to expect that $f\cY \in \ps$ when $f$ is a squash and stretch. Indee, if $f$ is a squash and stretch that has the same dilation factor $\alpha \neq 0,1$ for each class $i$ then $f\cY \not \in \ps$. Nevertheless, it is possible that $f\cY \in \ps$ for some $f$. Intuitively, if $f$ has dilation factor $< 1$ for some class and dilation factor $>1$ for other ones, it may be that these effects cancel out so that $f\cY\in\ps$. %

\begin{figure}
    \centering
    \begin{subfigure}{.24\linewidth}
        \centering
        \includegraphics[width=1.\linewidth]{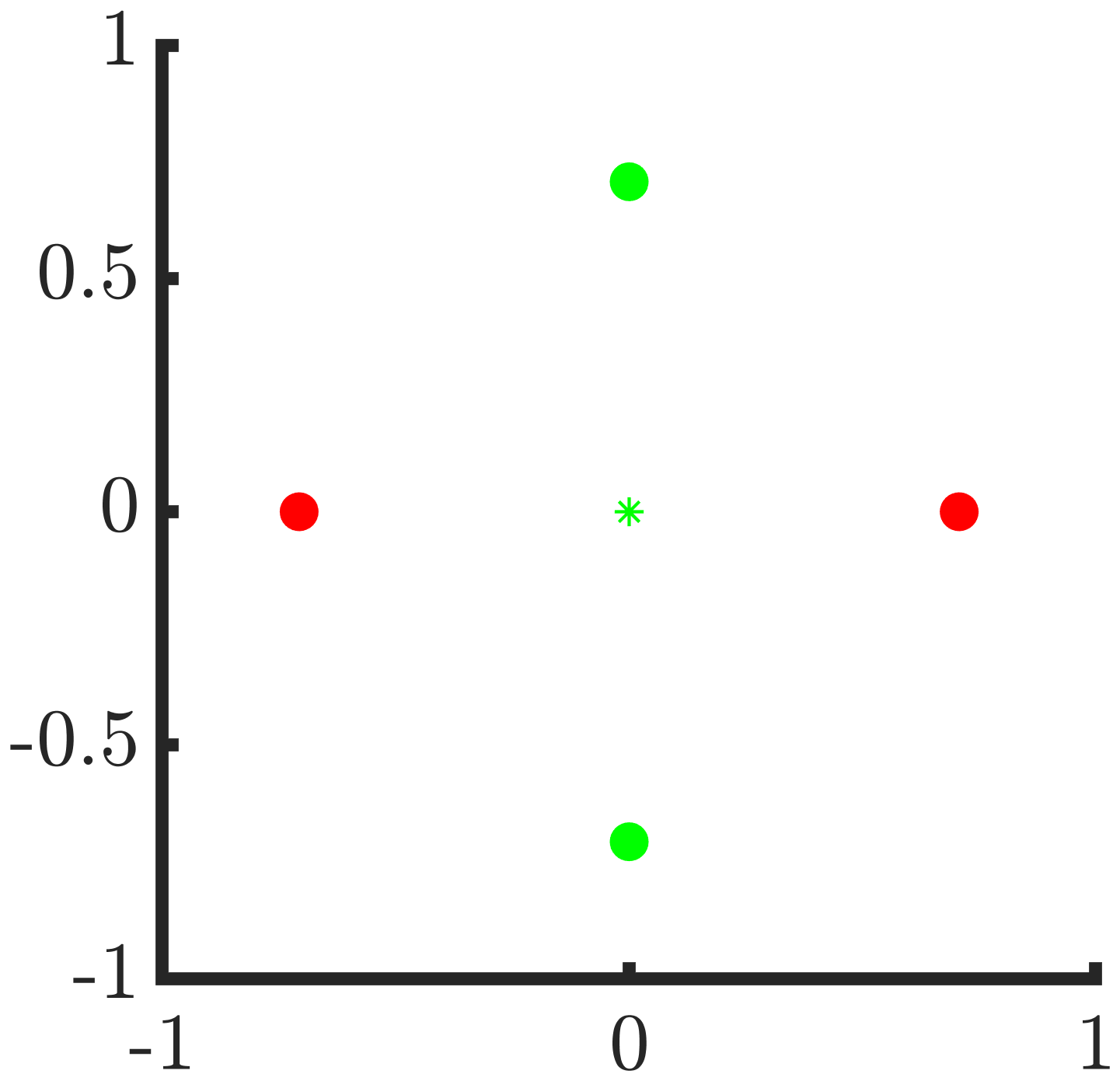}
        \subcaption{$\cY_1$}
        \label{fig:squash-stretch-iso-fig:Y1}
    \end{subfigure}
    \begin{subfigure}{.24\linewidth}
        \centering
        \includegraphics[width=1.\linewidth]{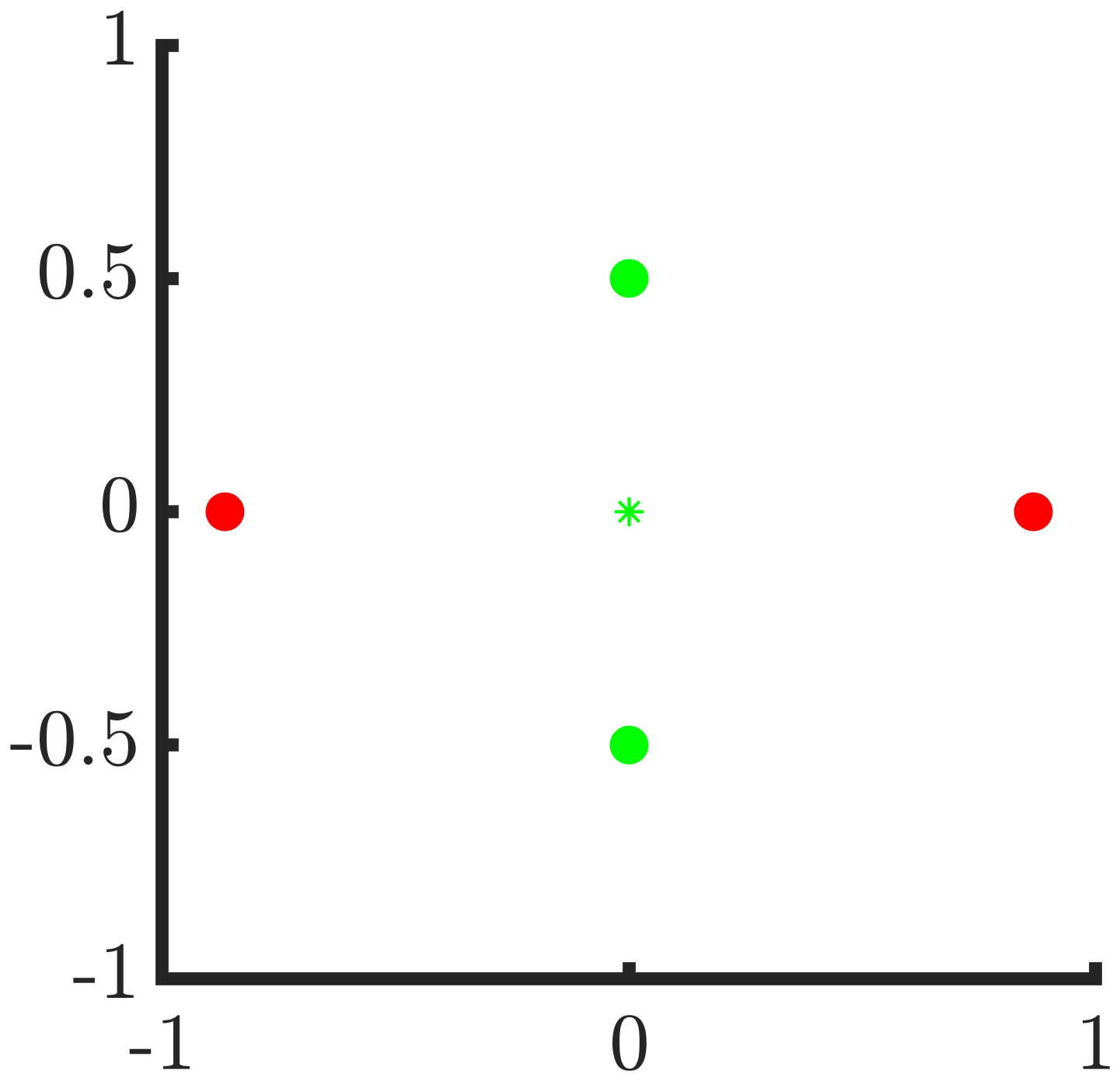}
        \subcaption{$f_1\cY_1$}
        \label{fig:squash-stretch-iso-fig:fY1}
    \end{subfigure}
    \begin{subfigure}{.24\linewidth}
        \centering
        \includegraphics[width=1.\linewidth]{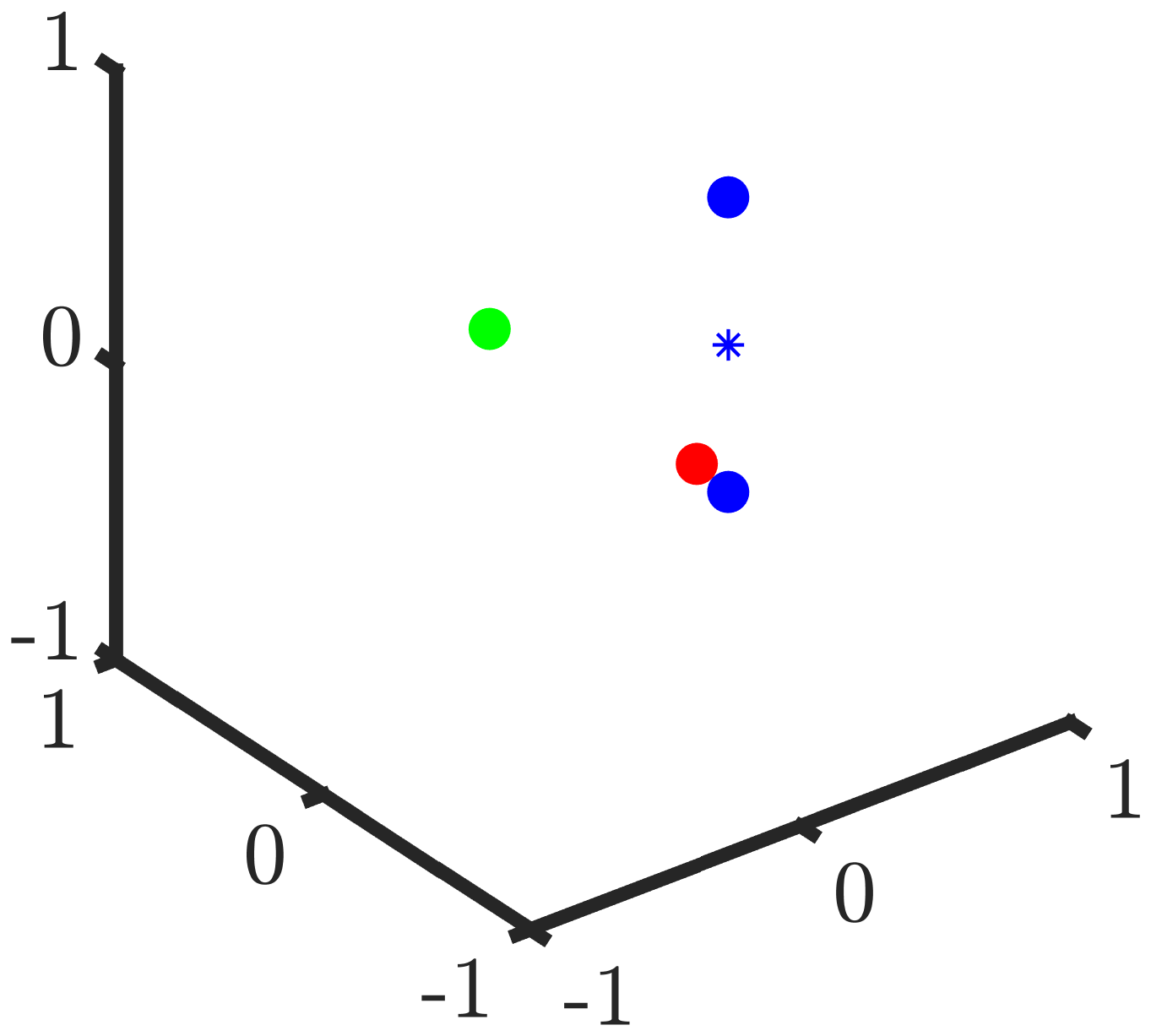}
        \subcaption{$\cY_2$}
        \label{fig:squash-stretch-iso-fig:Y2}
    \end{subfigure}
    \begin{subfigure}{.24\linewidth}
        \centering
        \includegraphics[width=1.\linewidth]{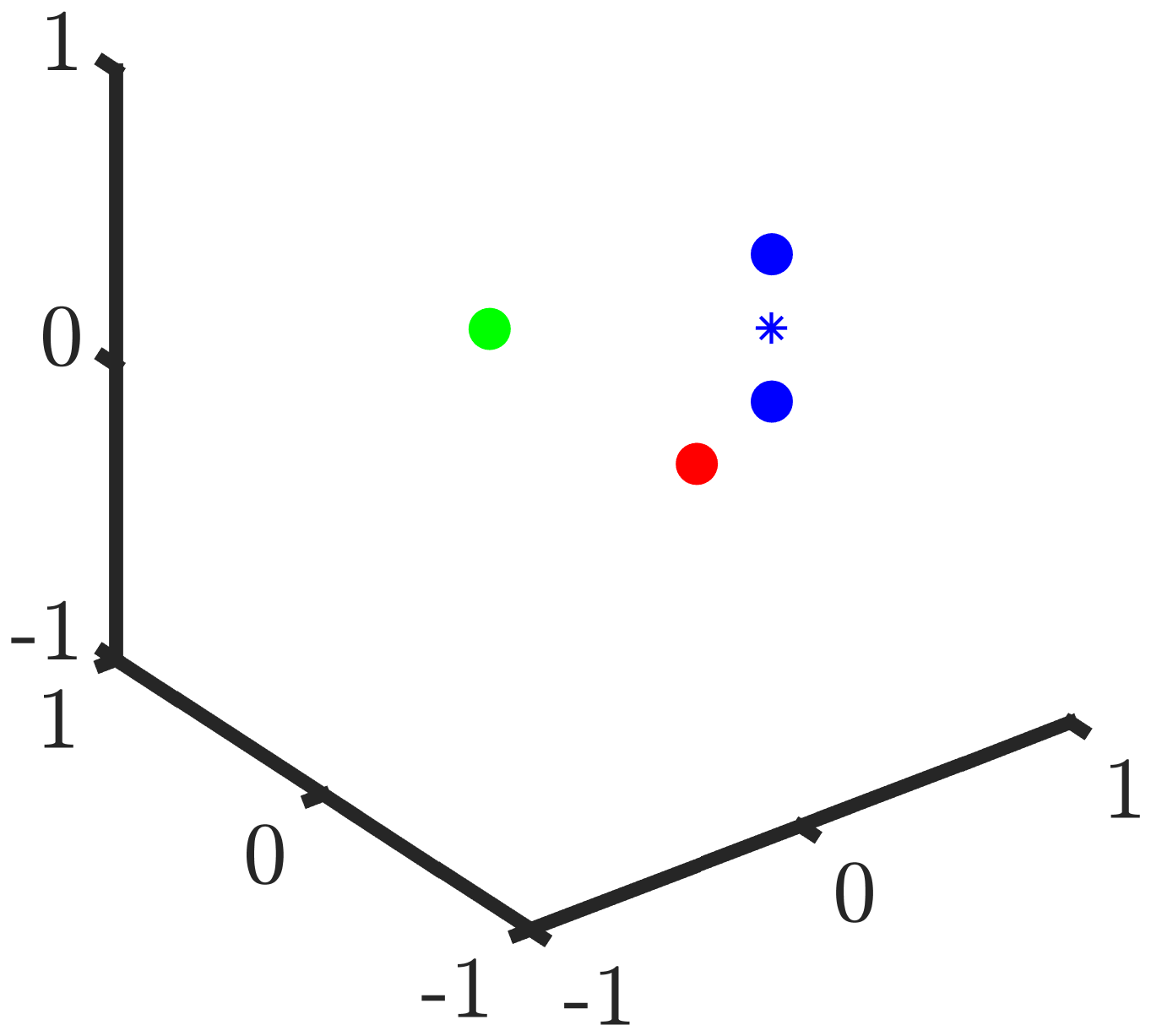}
        \subcaption{$f_2\cY_2$}
        \label{fig:squash-stretch-iso-fig:fY2}
    \end{subfigure}\\
    \begin{subfigure}{.24\linewidth}
        \centering
        \includegraphics[width=1.\linewidth]{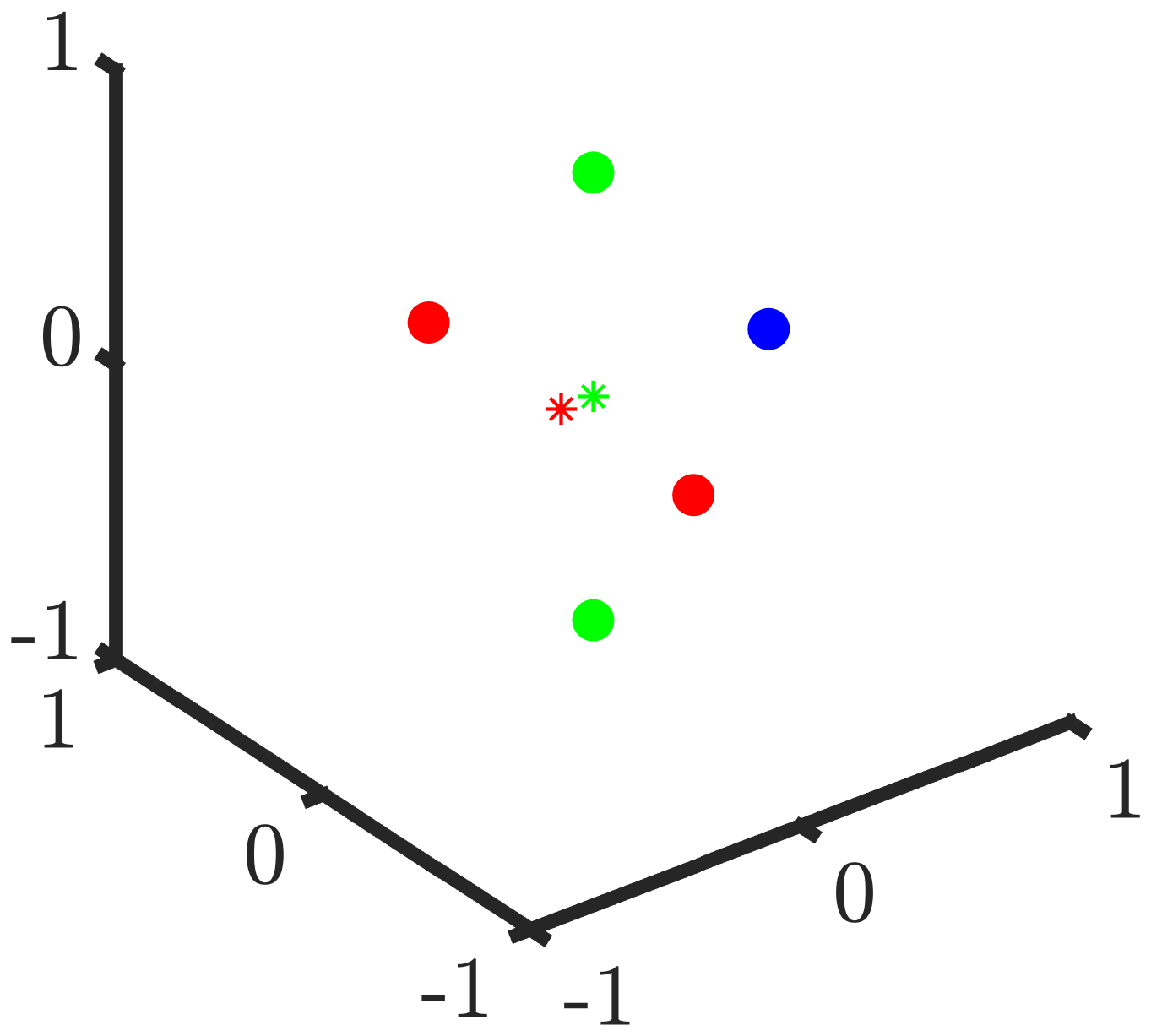}
        \subcaption{$\cY_3$}
        \label{fig:squash-stretch-iso-fig:Y3}
    \end{subfigure}
    \begin{subfigure}{.24\linewidth}
        \centering
        \includegraphics[width=1.\linewidth]{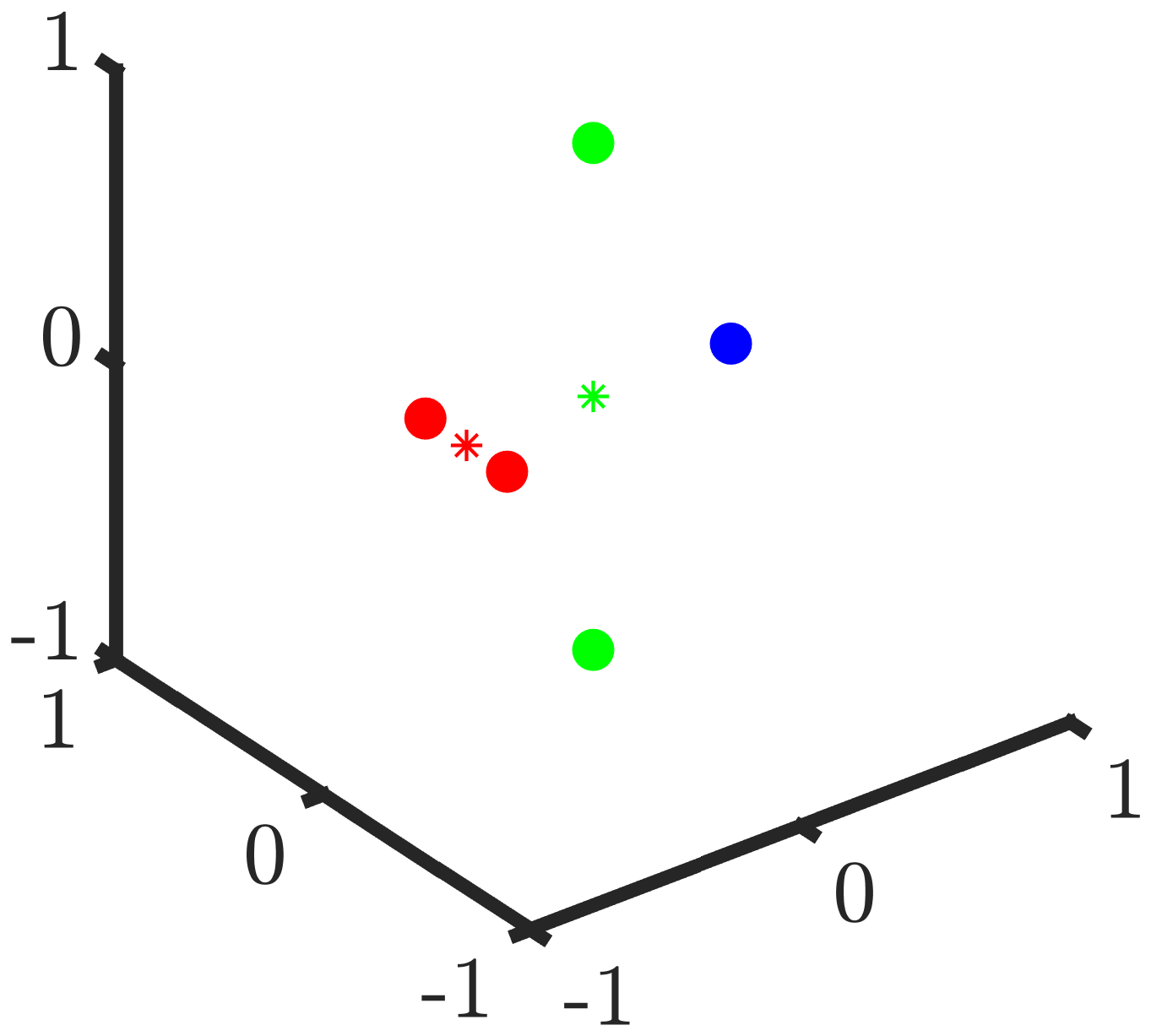}
        \subcaption{$f_3\cY_3$}
        \label{fig:squash-stretch-iso-fig:fY3}
    \end{subfigure}
    \begin{subfigure}{.24\linewidth}
        \centering
        \includegraphics[width=1.\linewidth]{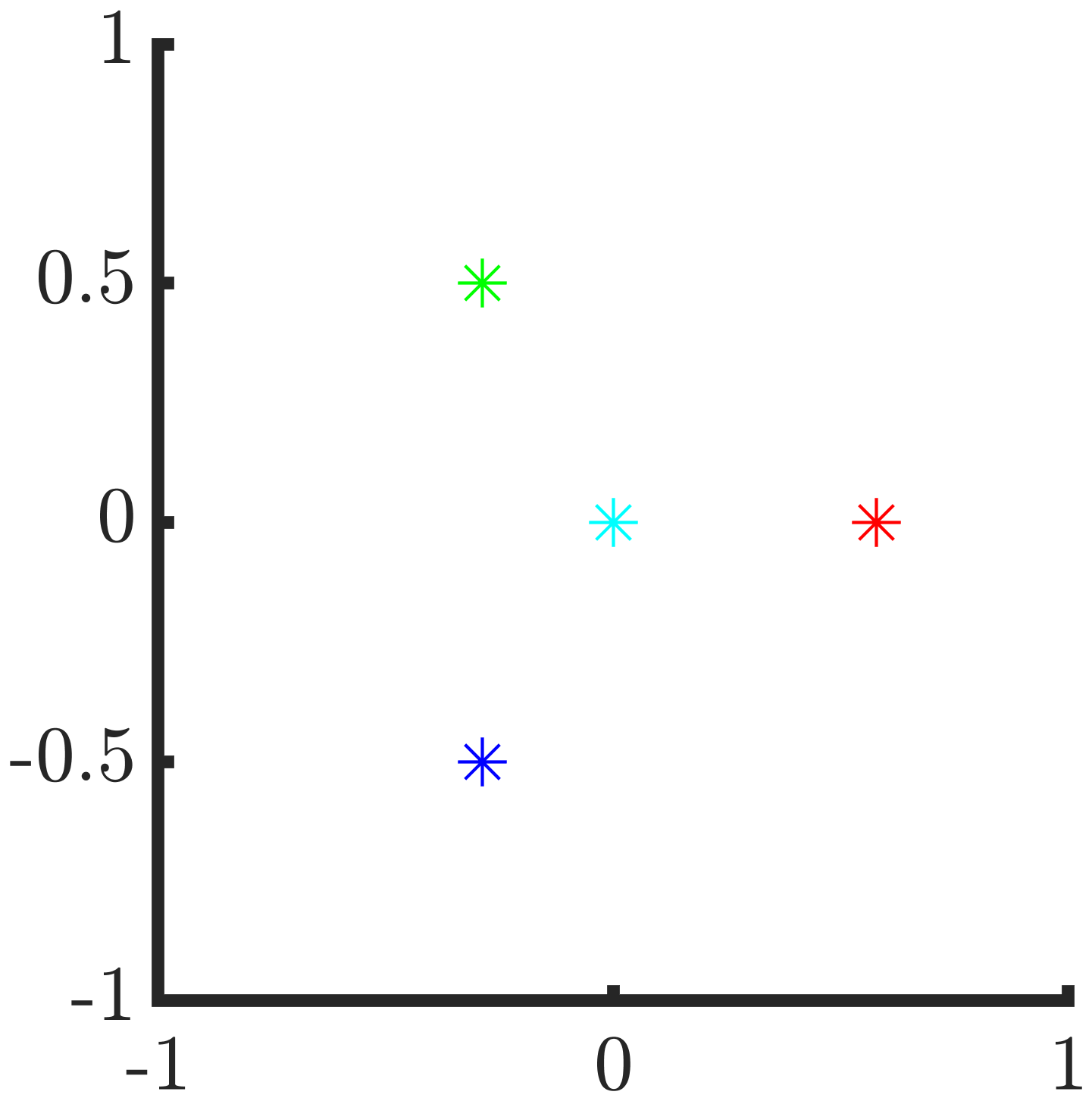}
        \subcaption{Centers of $\cY_4$}
        \label{fig:squash-stretch-iso-fig:Y4}
    \end{subfigure}
    \begin{subfigure}{.24\linewidth}
        \centering
        \includegraphics[width=1.\linewidth]{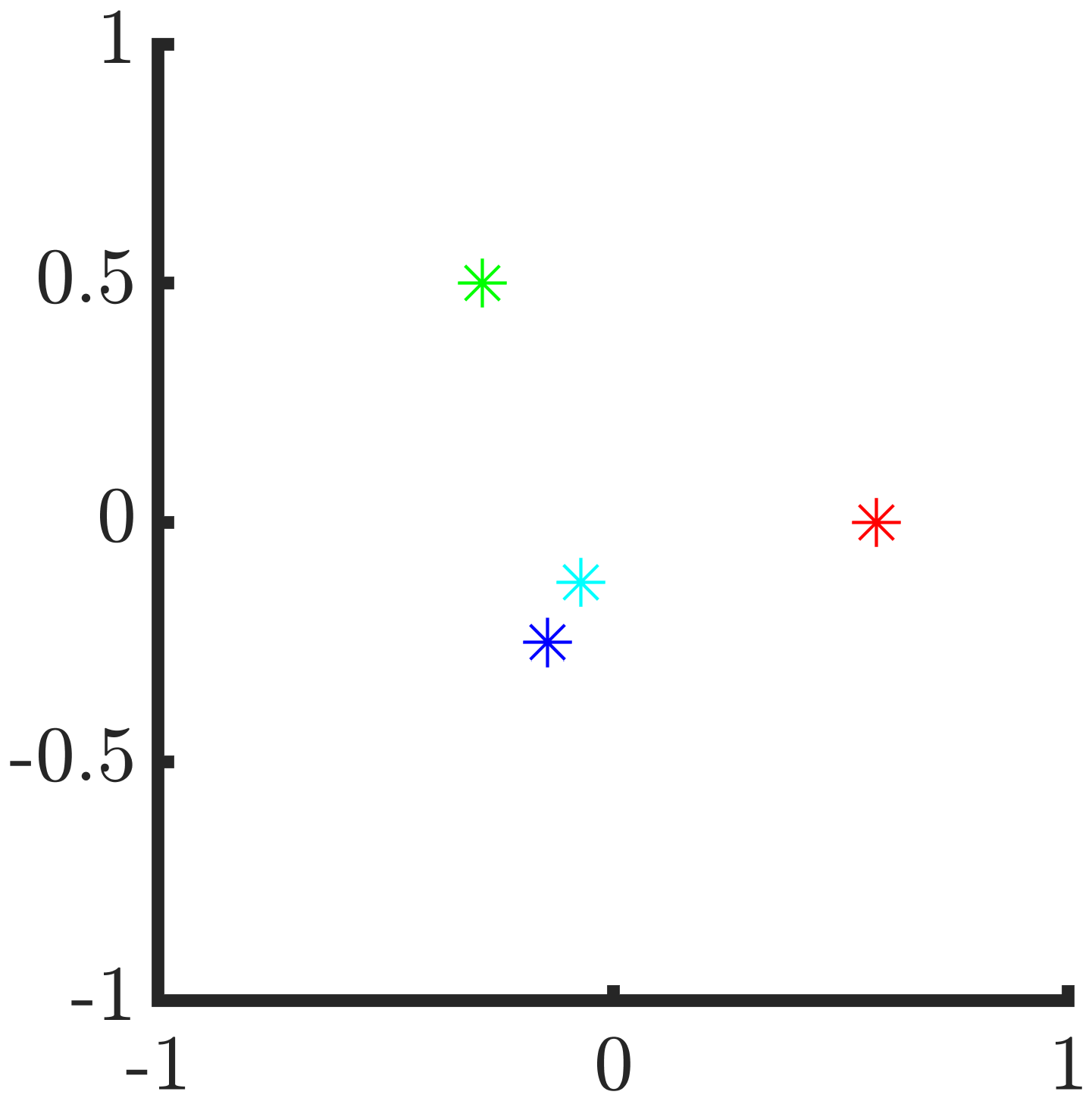}
        \subcaption{Centers of $f_4\cY_4$}
        \label{fig:squash-stretch-iso-fig:fY4}
    \end{subfigure}\\
    \caption{A plot of the examples of squash and stretch isometries detailed in Example \ref{examp:example-of-squash-stretch-isometries}. Six equidistant spacings are plotted. The four equidistant spacings, figures (a,c,e) in the left column, are squash and stretch isomorphic to the figures (b,d,f) in the respective position. Dots of the same color belong to the same class. For classes with more than two points, the centers are denoted by a star of matching color. The final figures (g,h) illustrate how squash and stretches effect the centers in higher dimensions. In figures (g,h) the equidistant spacings themselves are suppressed, and the centers of each class are plotted instead.}
    \label{fig:squash-stretch-iso-fig}
\end{figure}

\begin{example}[Example of Squash and Stretch Isometries]
    \label{examp:example-of-squash-stretch-isometries}

    Figures \ref{fig:squash-stretch-iso-fig:Y1} and \ref{fig:squash-stretch-iso-fig:fY1} demonstrate a squash and stretch isometry in two dimensions. Figure \ref{fig:squash-stretch-iso-fig:Y1} is a equidistant spacing formed by coloring opposite corners of a square of unit edge length the same color, and Figure \ref{fig:squash-stretch-iso-fig:fY1} is a rhombus with unit edge lengths. Notice how $f_1{}|_{\bigcup_{i = 1}^I c_i}$ is the identity and $f_1{}|_{A_i}$ acts by a dilation of factor $>1$ on the red class, and by factor $<1$ on the green class. Intuitively, dilating a class by a factor $<1$ moves its points closer to its center (and all other classes) and dilating a class by a factor $>1$ moves its points further away from its center (and all other classes). $f_1$ acts by applying one dilation $>1$, and another $<1$.

    Figures \ref{fig:squash-stretch-iso-fig:Y2} and \ref{fig:squash-stretch-iso-fig:fY2} demonstrate a squash and stretch isometry that acts on only one class, the blue class in both figures. $f_2$ acts via two actions. First, it moves the center blue class further from the green and red classes. Second, it dilates the blue class by factor $<1$, moving its points closer to its center.

    Figures \ref{fig:squash-stretch-iso-fig:Y3} and \ref{fig:squash-stretch-iso-fig:fY3} demonstrate a third kind of squash and stretch, $f_3$. $f_3$ acts on both the green and red classes. On the red class, it moves the red center further away from the other two centers and dilates the red class by a factor $<1$ to compensate. On the green class, it moves the center closer to the blue center, further from the red center, and dilates the green class by a factor $>1$.

    Figure \ref{fig:squash-stretch-iso-fig:Y4} and \ref{fig:squash-stretch-iso-fig:fY4} gives an idea for how squash and streches look in higher dimension. In this figure, the equidistant spacings are suppressed, only the centers are plotted. From this figure, it is each to see that a squash and stretch is a deformation of the centers that preserves orthocentricity.
\end{example}

\begin{remark}[Interpretation of Squash and Stretch Isometries on $\mps_{\neq}$]
    \label{rmk:interpretation-of-squash-and-stretch-isometries-on-mps-neq}
    Squash and stretch isometries have a natural interpretation when applied to $\mps_{\neq}$. Recall that Proposition \ref{prop:centers-are-orthocentric-systems} and Theorem \ref{thm:small-orthocentric-systems-correspond-to-centers-in-mps} together show that centers of $\mps_{\neq}$ form orthocentric systems. A squash and stretch is a deformation of the centers that preserves orthocentricity, and does not make the resulting simplex too large (in the sense of Theorem \ref{thm:small-orthocentric-systems-correspond-to-centers-in-mps}).
\end{remark}

As shown in example \ref{examp:example-of-squash-stretch-isometries}, it is possible to find squash and stretch isometries that are distinct from a rototranslation or recoloring. Further, as demonstrated by the squash and stretch isometry $f_3$ in Figure \ref{fig:squash-stretch-iso-fig}, squash and stretch isometries may be non-trivial, in the sense that they require moving several classes in a carefully balanced sequenced way to maintain the equidistant spacing.

\subsection{Canonical Form}
\label{sec:canonical-form}
The main difficulty in classifying maximal equidistant spacings is understanding when two equidistant spacings in $\mps_{\neq}$ are squash and stretch isomorphic. To solve this problem we develop a canonical form of maximal equidistant spacings in this section. It is an invariant of isometries, Proposition \ref{prop:invariance-of-equilateral-normal-form}.

The big idea behind the canonical form is that the centers of $\mps_{\neq}$ are \emph{not} geometrically homogeneous. There is one distinguished center that lies inside the simplex generated by the other centers, as in Def.\ref{def:inner-class}. Moreover, it turns out that squash and stretches that involve the inner center and any other center, called inner-outer squash and stretches, are straight-forward to characterize, see Proposition \ref{prop:outer-inner-squash-stretch}. Further, any squash and stretch may be written as the composition of inner-outer squash and stretches.

The following proposition is very useful, but requires some motivation. The idea is to explore the simplest possible squash and stretch, one that is the identity on all classes except the inner class and one outer class. These simple squash and stretches are the building blocks of many results to come.

The setup of Proposition \ref{prop:outer-inner-squash-stretch} is unusual, but it has an important message. Think of $\cY \in \ps(I)$ as being the main spacing of interest and $\cY_{1:I-2} \coloneqq \bigsqcup_{i = 1}^{I-2} Y_i \in \ps(I-2)$ and $\cY_{I-1:I} \coloneqq \bigsqcup_{i = I-1}^{I} Y_I \in \ps(2)$ as being the spacings formed by considering classes $Y_1,\dots,Y_{I-2}$ and $Y_{I-1},Y_{I}$ independently, where $Y_{I-1}$ is an outer class of $\cY$, and $Y_{I}$ is its inner class. Clearly, $\cY_{1:I-2}$ and $\cY_{I-1:I}$ glue (because $\cY_{1:I-2} \vee \cY_{I-1:I} = \cY$). What Proposition \ref{prop:outer-inner-squash-stretch} shows, is that there are squash and stretches $f$ of $\cY_{I-1:I}$ so that $\cY_{1:I-2}$ and $f\cY_{I-1:I}$ glue. We can regard such squashes and stretches of $\cY_{I-1:I}$ as squash and stretches of $\cY$ that are not the identity on two classes via $f\cY \coloneqq \cY_{1:I-2}\vee f\cY_{I-1:I}$. Proposition \ref{prop:outer-inner-squash-stretch} spells out under what conditions on $f$ the spacings $\cY_{1:I-2}$ and $f\cY_{I-1:I}$ glue. The $r_{I-1}$ and $r_I$ in the below proposition are the radii of $fY_{I-1}$ and $fY_I$.%

\begin{proposition}[Outer-Inner Squash and Stretch]
    \label{prop:outer-inner-squash-stretch}
    Let $\cY \in \ps(I-2,\R^n)$ be a spacing, and $V \subset \aff(|\cY|)^\perp$. %
    
    Let $r_{I-1} \in [0,\frac{\sqrt 2}2), r_I \in[\frac{\sqrt2}2,\sqrt{1 - r_{I-1}^2})$ be such that
    \begin{align}
        \label{eqn:prop:outer-inner-squash-stretch:extent-condition}
        \paren{\extent \cY}^2 = 1 - r^2_I - \frac{\paren{2r_I^2 - 1}^2}{4\paren{1 - r_I^2 - r_{I-1}^2}}, \quad 
        \dim V \geq 2 \quad \text{and}\quad \dim V \geq 3 \text{ if } r_{I-1} > 0.
    \end{align}
    Then there is a $c_I,c_{I-1} \in \R^n$ and $\cY' \in \ps(I,\R^n)$ so that the following holds
    \begin{enumerate}
        \item The center of $Y_{I-1}'$ and $Y_I'$ are $c_{I-1}$ and $c_I$ respectively.
        \item The radius of $Y_{I-1}'$ and $Y_I'$ are $r_{I-1}$ and $r_I$ respectively.
        \item $Y'_i = Y_i$ when $i \in 1,\dots,I_1$.
        \item $c_{I} \in \overline{c_{I-1} c^*_{I-1}}$ where $\quad c^*_{I-1} \in \proj_{\aff(c_1,\dots,c_{I-2})}c_{I-1}$.
    \end{enumerate}
\end{proposition}

\begin{figure}
    \centering
    \includegraphics[width=0.8\linewidth]{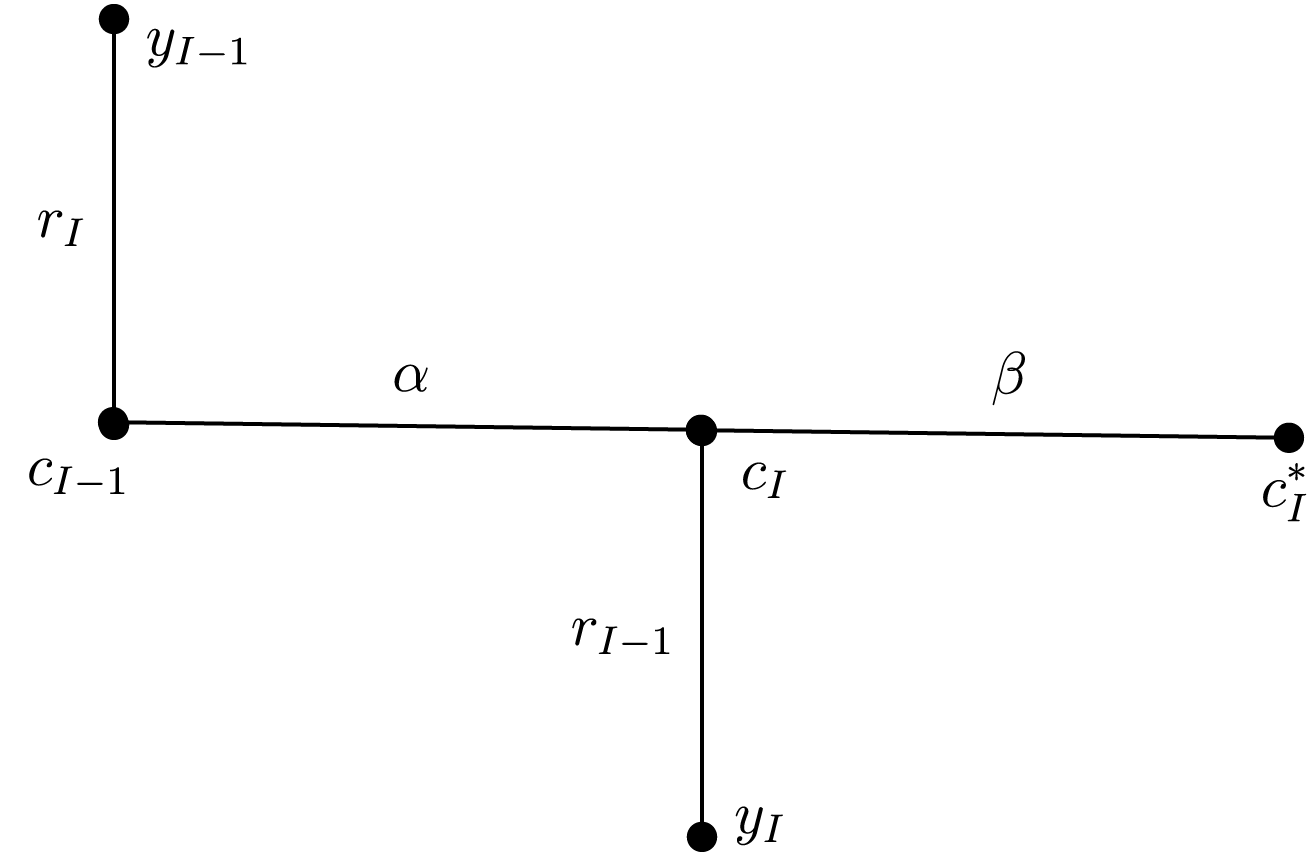}
    \caption{A sketch of the quantities used in the proof of Proposition \ref{prop:outer-inner-squash-stretch}}
    \label{fig:prop:outer-inner-squash-stretch}
\end{figure}
\begin{proof}
    The setup for the notation of the proof is given in Figure \ref{fig:prop:outer-inner-squash-stretch}.

    Let $e_1,e_2 \in V$ be orthogonal, and define $\alpha \coloneqq 1 - r_I^2 - r_{I-1}^2$ and $\beta \coloneqq \frac{2r_I^2 - 1}{2\sqrt{1 - r_I^2 - r_{I-1}^2}}$. Let $c^*$ be the glue site of $\cY$, and put
    \begin{align}
        c_{I-1} \coloneqq c^* + \beta e_1, \quad \text{and} \quad c_{I} \coloneqq c^* + (\alpha +\beta)e_1.
    \end{align}
    Define $Y_I\coloneqq S^{n-1}_{r_{I}}(c_I) \cap \paren{c_I + \Span{e_2}}$. If $r_{I-1} > 0$, let $e_3 \in V$ be orthonormal to $e_1,e_2$ and put $Y_{I-1}\coloneqq S^{n-1}_{r_{I-1}}(c_{I-1}) \cap \paren{c_{I-1} + \Span{e_3}}$. Otherwise put $Y_3 \coloneqq \set{c_{I-1}}$.

    Let $\alpha = \sqrt{1 - r_I^2 - r_{I-1}^2}$ and $\beta = \frac{2r_I^2 - 1}{2\sqrt{1 - r_I^2 - r_{I-1}^2}}$. By direct computation we have that 
    \begin{align}
        r_I^2 + \beta^2 + \paren{\extent \cY }^2 = 1 \quad \text{and} \quad r^2_{I-1} + \paren{\alpha + \beta}^2 = r_I^2 +\beta^2.
    \end{align}
    Then we must verify the following.
    \begin{enumerate}
        \item $\norm{y_I - y_i} = 1$ where $i < I-1$. 
        \item $\norm{y_{I-1} - y_i} = 1$ where $i < I-1$. 
        \item $\norm{y_I - y_{I-1}} = 1$.
    \end{enumerate}

    \begin{enumerate}
        \item Calculating, from orthogonality of the projection and collinearity of $c_I$ and $c_{I-1}$, we have that
        \begin{align}
            \norm{y_I - y_i}^2 &= \norm{y_I - c_I}^2 + \norm{c_I - c^*_{I-1}}^2 + \norm{c^*_{I-1} - y_i}^2\\
            &= r_I^2 + \beta^2 + \paren{\extent \cY }^2\\
            &= 1.
        \end{align}

        \item Again from orthogonality of the projection and collinearity of $c_I$ and $c_{I-1}$, we have that
        \begin{align}
            \norm{y_{I-1} - y_i}^2 &= \norm{y_{I-1} - c_{I-1}}^2 + \paren{\norm{c_{I-1} - c_I} + \norm{c_{I} - c_{I-1}^*}}^2 \nonumber\\
            &+ \norm{c^*_{I-1} - y_i}^2\\
            &= r_{I-1}^2 + \paren{\alpha + \beta}^2 + \paren{\extent \cY }^2\\
            &= r_I^2 + \beta^2+ \paren{\extent \cY }^2\\
            &= 1.
        \end{align}
        \item By orthogonality of all terms,
        \begin{align}
            \norm{y_{I-1} - y_I}^2 &= \norm{y_{I-1} - c_{I-1}}^2 + \norm{c_{I-1} - c_I}^2 + \norm{c_I - y_I}^2\\
            &= r_{I-1}^2 + \alpha^2 + r_{I}^2\\
            &= 1
        \end{align}
    \end{enumerate}
\end{proof}
Proposition \ref{prop:outer-inner-squash-stretch} has the following corollary.

\begin{corollary}[Necessary Conditions on Radius for Maximal Equidistant Spacings]
    \label{cor:nec-conditions-on-radius-for-maximal-equidistant-spacing}
    Let $\cY \in \mps$ be a equidistant spacing with inner class $I$. Then $r_I \geq \frac{\sqrt{2}}2$.
\end{corollary}
\begin{proof}
    Let $\cY|_{1:I-2} \coloneqq \sqcup_{i = 1}^{I-2} Y_i$. Proposition \ref{prop:outer-inner-squash-stretch} applies to $\cY|_{1 : I-2}$. Note that $\cY|_{1:I-2} \vee \paren{Y_{I-1}\sqcup Y_I} = \cY$. Because $\cY \in \mps$, $c_I \in \overline{c_{I-1} c_{I-1}^*}$ where $c_{I-1}^* \coloneqq \proj_{\aff{\cY|_{1:I-2}}} c_{I-1}$ by Proposition \ref{prop:centers-are-orthocentric-systems}. Proposition \ref{prop:outer-inner-squash-stretch} applies to $\cY|_{1:I-2}$. From the proof of Proposition \ref{prop:outer-inner-squash-stretch}, it is clear that the conditions on $r_I$ and $r_{I-1}$ are if and only if. Therefore, $r_I \geq \frac{\sqrt{2}}2$.
\end{proof}

\begin{corollary}[Determination of Center Coinciding of Maximal Spacings]
    \label{cor:determination-of-center-coinciding-of-maximal-spacings}
    Let $\cY \in \mps(I)$ for $I > 2$. Then the centers of $\cY$ coincide if and only if there exists a class $i$ so that $r_i = \frac{\sqrt{2}}2$. The centers of $\cY$ do not coincide if and only if there exists a class $i$ so that $r_i > \frac{\sqrt{2}}2$.
\end{corollary}

\begin{figure}
    \centering
    \includegraphics[width=0.8\linewidth]{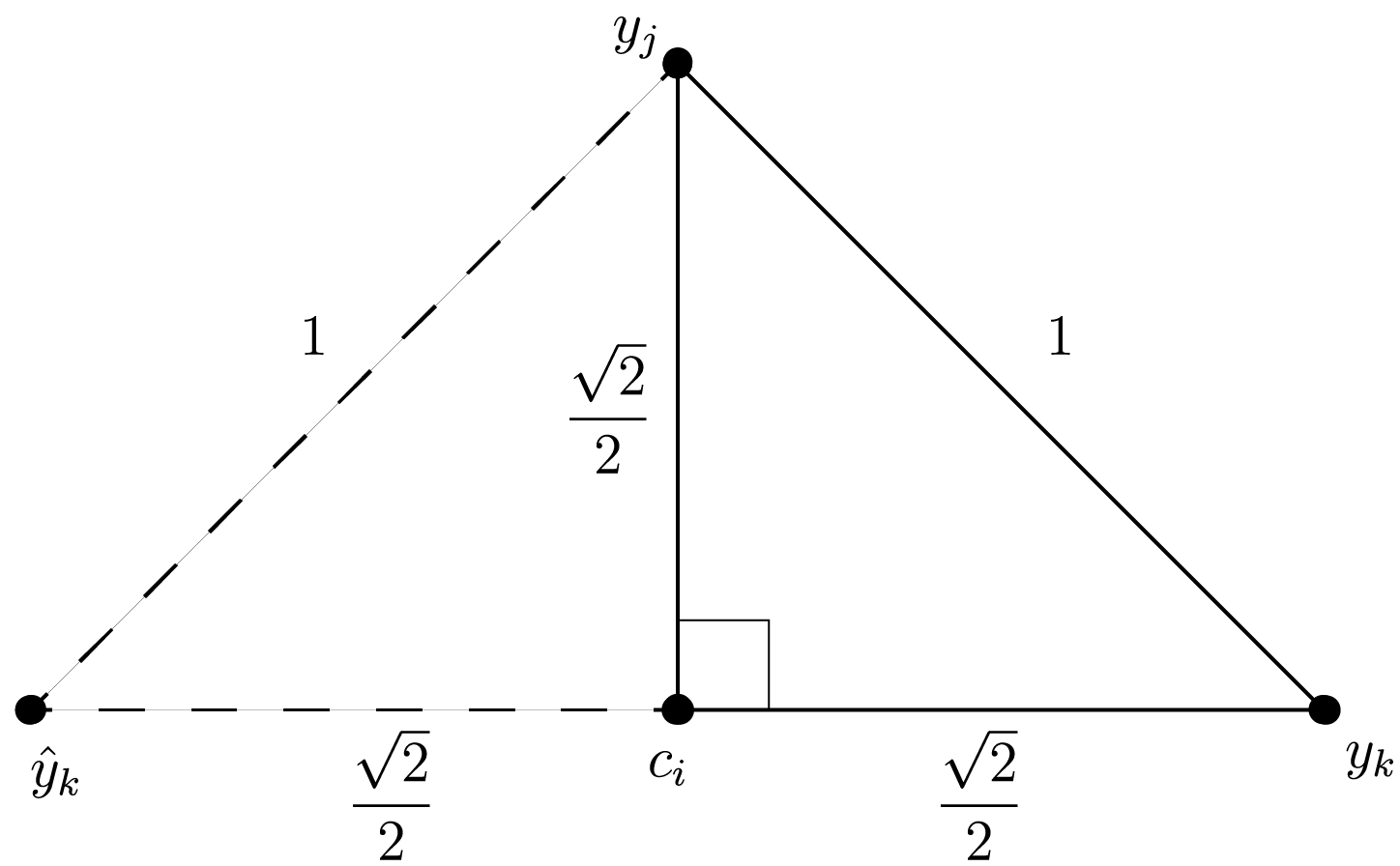}
    \caption{A sketch of the triangles $\bigtriangleup c_iy_jy_k$ and $\bigtriangleup c_iy_j\hat y_k$ used in the proof of Corollary \ref{cor:determination-of-center-coinciding-of-maximal-spacings}.}
    \label{fig:cor:determination-of-center-coinciding-of-maximal-spacings}
\end{figure}

\begin{proof}
    Let $i,j,k$ be distinct. Let $i$ be such that $r_i = \frac{\sqrt{2}}2$. Then $\norm{c_i - y_j} = \norm{c_i - y_k} = \frac{\sqrt{2}}2$. So, by $\norm{y_j - y_k} = 1$, the triangle $\bigtriangleup c_iy_jy_k$ is right. See Figure \ref{fig:cor:determination-of-center-coinciding-of-maximal-spacings}, for an illustration. If we define $\hat y_k \coloneqq 2c_i - y_k$ (the reflection of $y_k$ across the center $c_i$) then the triangle $\bigtriangleup c_iy_j\hat y_k$ is a congruent triangle, and so $\norm{\hat y_k - y_j} = 1$, and so $\hat y_k \in Y_k$ by maximality of $Y_k$. Therefore, $c_i \in \aff{Y_k}$, and so $c_k = c_i$ by Def. \ref{def:center-and-radius}. Therefore, all centers coincide.

    Let $\cY$ be such that all centers coincide at point $c$ and let $i,j,k$ be distinct. Then by orthogonality and Pythagorean's theorem, we have that the 
    \begin{align}
        \label{eqn:center-coinciding-matrix}
        \pmat{1&1&0\\1&0&1\\0&1&1} \pmat{\norm{y_i - c}^2\\\norm{y_j - c}^2\\\norm{y_k - c}^2} = \pmat{1\\1\\1}.
    \end{align}
    This system has only one solution, $\norm{y_i - c}^2 = \norm{y_j - c}^2 = \norm{y_k - c}^2 = \frac12$.

    Let $i$ be such that $r_i > \frac{\sqrt{2}}2$. Then the centers of $\cY$ must not coincide. If they did, this would yield a second solution to \ref{eqn:center-coinciding-matrix}.

    By applying Corollary \ref{cor:nec-conditions-on-radius-for-maximal-equidistant-spacing} to $\cY$, it must be that $r_i \neq \frac{\sqrt2}2$, otherwise the centers would coincide. Therefore, $r_i > \frac{\sqrt2}2$.
\end{proof}

\begin{corollary}[Characterization of Inner Class]
    \label{cor:characterization-of-inner-class}
    Let $\cY \in \mps_{\neq}$. Then the following are equivalent
    \begin{enumerate}
        \item Class $i$ is the inner class of $\cY$.
        \item Class $i$ satisfies $r_i > \frac{\sqrt{2}}2$.
        \item Class $i$ solves $i = \argmax_{j \in 1,\dots,I} r_j$.
    \end{enumerate}
\end{corollary}

\begin{proof}
    The proof of Corollary \ref{cor:determination-of-center-coinciding-of-maximal-spacings} shows that 1. and 2. are equivalent. The equivalence of points 2. and 3. follows from Pythagorean theorem.
\end{proof}

We now present two helper lemmas which will be useful in a theorem to come.

\begin{lemma}[Existence of Solution to Eqn. \ref{eqn:prop:outer-inner-squash-stretch:extent-condition}]
    \label{lem:existence-of-soln-to-extent-condition}
    For a given $\cY$ so that $\extent \cY \leq \frac{\sqrt 2}2$ let $r_{I-1} \in [0,\frac{\sqrt 2}2)$. Then there is a $r_I \in [\frac{\sqrt{2}}2,\sqrt{1 - r_{I-1}^2})$ so that Eqn. \ref{eqn:prop:outer-inner-squash-stretch:extent-condition} is satisfied.
\end{lemma}

\begin{proof}
    When $r_{I-1}^2 \in [\frac12,1 - r_I^2)$, the constraints on $r_I$ and $r_{I-1}$ are satisfied. 
    The idea is to regard $r_{I-1}$ as fixed and use the intermediate value theorem on the function 
    \begin{align}
        f(r) \coloneqq 1 - r^2 - \frac{\paren{2r^2 - 1}^2}{4\paren{1 - r^2 - r_{I-1}^2}} - \gamma
    \end{align}
    where $\gamma \coloneqq \paren{\extent \cY}^2 \leq \frac12$. Note that Eqn. \ref{eqn:prop:outer-inner-squash-stretch:extent-condition} is satisfied when $f(r) = 0$.
    When $r_I^2  = \frac 12$, then
    \begin{align}
        f(r) &= 1 - r^2 - \frac{\paren{2r^2 - 1}^2}{4\paren{1 - r^2 - r_{I-1}^2}} - \gamma \\
        &= 1 - \frac12 - \frac{\paren{2\frac12 - 1}^2}{4\paren{1 - \frac12 - r_{I-1}^2}} - \gamma\\
        &= \frac12 - \gamma \geq 0.
    \end{align}
    When sending $r^2 \to 1 - r_{I-1}^2$ yields 
    \begin{align}
        \lim_{r^2 \to 1 - r_{I-1}^2}f(r) &= \lim_{r^2 \to 1 - r_{I-1}^2} 1 - r^2 - \frac{\paren{2r^2 - 1}^2}{4\paren{1 - r^2 - r_{I-1}^2}}-\gamma\\
        &= 1 - \paren{1 - r_{I-1}^2} - \frac{\paren{2\paren{1 - r_{I-1}^2} - 1}^2}{\lim_{r^2 \to 1 - r_{I-1}^2} 4\paren{1 - r^2 - r_{I-1}^2}}-\gamma\\
        &= -\infty.
    \end{align}
    Hence, $f(r) = 0$ has a solution.
\end{proof}

\begin{corollary}[Calc. III]
    \label{cor:calc-3}
    Let $\cY$ have centers at $\set{c_i}_{i = 1}^k = E_{\ell,k}$ for some $k$, and $\ell^2 = 1 - 2r^2$ for some $r$. Then 
    \begin{align}
        \extent \cY^2 = \frac{2r^2 + k}{2(k+1)}.
    \end{align}
\end{corollary}
\begin{proof}
    Note the conditions imply that $r_i = r$ is constant.

    We may apply Proposition \ref{prop:circumcenter-of-equi-simplex}, we have that $\norm{c_i - c^*}^2 = \ell^2\frac{k}{2(k+1)}$ where $c^*$ denotes the glue site of $\cY$. By Proposition \ref{prop:tri-part-ortho}, 
    \begin{align}
        \norm{y - c^*}^2 &= \norm{y - c_i}^2 + \norm{c_i - c^*}^2\\
        &= r^2 + (1 - 2r^2)\frac{k}{2(k+1)}\\
        &=\frac{2r^2 + k}{2(k+1)}.
    \end{align}
\end{proof}

\begin{lemma}[Squash and Stretches of Positive-Radius Equidistant Spacings with an Inner Class]
    \label{lem:squash-stretches-of-maximal-positive-radius-equidistant-spacings}
    Let $\cY \in \ps(I)$ have an inner class and positive radii. For any $r \in \paren{0,\frac{\sqrt{2}}2}$, there is a squash and stretch isometry $f$ of $\cY$ so that the radius of the inner class of $f\cY$ is 
    $\sqrt{\frac{I+1-2r^2}{2I}}$ 
    and the radii of all other classes of $f\cY$ is $r$.
\end{lemma}

\begin{proof}
    The idea behind the proof is to give a sequence of inner-outer squash and stretches that each leave $Y_j$ with the correct radius and is the identity on other outer classes. The overall squash and stretch is then given by the composition of these inner-outer squash and stretches.

    W.l.o.g., let $I$ denote the center class of $\cY$ and $j \neq I$. By using the same trick as in Corollary \ref{cor:nec-conditions-on-radius-for-maximal-equidistant-spacing}, we may consider $\cY = \cY|_{1:j-1,j+1:I-1} \vee \cY|_j \vee \cY_I$ where $\cY|_{1:j-1,j+1:I-1}\coloneqq \sqcup_{i \in \set{1:j-1,j+1:I-1}} Y_i$, $\cY|_j \coloneqq \set{Y_j}$ and $\cY|_I \coloneqq \set{Y_I}$. By inner-ness of class $I$, we have that $c_I \in \overline{c_jc^*_j}$ where $c^*_j \coloneqq \proj_{\aff{\cY_{1:j-1,j+1:I-1}}} c_j$. 
    
    This calculation proceeds by explicitly constructing a squash and stretch $f_j$ of $Y_j$ and $Y_I$ so that $f_jY_j$ has the correct radius. In this calculation, $r_j$ and $r_I$ (resp. $c_j$ and $c_I$) denote the radii (resp. centers) of $\cY$ \emph{before} $f_j$ has been applied, and $r_j'$ and $r_I'$ (resp. $c_j'$ and $c_I'$) denote the radii (resp. centers) of $f_jY_j$, and $f_jY_I$. That is, after the squash and stretch has been applied. 
    
    We have that $\extent \cY|_{I-1,I} > \frac{\sqrt 2}2$, hence $\extent \cY|_{1:j-1,j+1:I-1} < \frac{\sqrt 2}2$. For a given $r'_i \in (0,\frac{\sqrt{2}}2)$, by Lemma \ref{lem:existence-of-soln-to-extent-condition}, there is a $r_I' \in [\frac{\sqrt 2}2 , \sqrt{1 - r_i'{}^2})$ so that Proposition \ref{prop:outer-inner-squash-stretch} applies where %
    \begin{align}
        c_j' \coloneqq c^* + \paren{\alpha +  \beta} e, \quad c_I' \coloneqq c^* + \beta e, \quad 
    \end{align}
    and $\alpha \coloneqq \sqrt{1 - r_I'{}^2 - r_i'{}^2}$, $\beta \coloneqq \frac{2r_I'{}^2 - 1}{2\sqrt{1 - r_I'{}^2 - r_i'{}^2}}$, $c^*$ is the glue site of $\cY|_{1:j-1,j+1:I-1}$, and $e = \frac{c_j - c_I}{\norm{c_j - c_I}}$. Notice that $\norm{c_j - c_I} \neq 0$ because the centers of $\cY$ do not coincide. 
    
    Let $f_{j}$ be the squash and stretch that leaves classes $Y_i$ fixed when $i \neq j, I$, and
    \begin{align}
        f_{j}c_j = c'_j, \quad f_{j}c_I = c'_I,
    \end{align}
    dilates class $j$ by factor $\frac{r}{r_j}$ and class $I$ by factor $\frac{r'_I}{r_I}$. $f_{j}$ is a squash and stretch by construction. Note that $f_jY_{j} \sqcup f_jY_I \in \ps$ follows from Proposition \ref{prop:outer-inner-squash-stretch} and the $j$'th class of $f_j\cY$ has radius $r$.

    Do this construction for each $j = 1,\dots,I_1$ and define
    \begin{align}
        f \coloneqq f_{1} f_{2}\dots f_{I-1}, \quad \cY' \coloneqq f \cY.
    \end{align}
    Then $f$ is a squash and stretch isometry by Proposition \ref{prop:prop-of-squash-stretch}. Further, class $i$ has radius $r$ by construction when $i \neq I$. 

    Calculating, we see that $\norm{c_i - c_j}^2 = 1 - 2r_j'{}^2$ when $i \neq j$, $i,j < I$, and so the centers of $f\cY$ form an equilateral simplex with edge length $\sqrt{1 - 2r_j'{}^2}$. So, Corollary \ref{cor:extent-of-equidistant-spacing} applies where $v = c_j$, $\ell = \sqrt{1 - 2r_j'{}^2}$, and $n = I-1$. Hence,
    \begin{align}   
        1 - r_I'{}^2 - r{'}_j^2 = \norm{c_j' - c_I'}^2 = \paren{1 - 2r'_j{}^2}\frac{I-1}{2(I - 1 + 1)}
    \end{align}
    Solving the above for $r_I'{}^2$, we have that
    \begin{align}
        r_I'{}^2 = r_j{'}^2 + \paren{1 - \frac{I-1}{2I}}\paren{1 - 2r'_j{}^2} = \frac{I+1-2r_j{'}^2}{2I}.
    \end{align}
    And so, $r_I' = \sqrt{\frac{I+1-2r'{}_j^2}{2I}}$. By construction of the squash and stretches, $r'_j = r$. 
\end{proof}

In the following definition, recall that $E_{\ell,k}$ denotes the equilateral simplex with $k$ vertices embedded in $\R^{k-1}$ with $k$ vertices and side length $\ell$.%

\begin{definition}[Equilateral Normal Form of Maximal Equidistant Spacings]
    \label{def:equilateral-normal-form}
    Let $\cY \in \mps$. We say that $\cY$ is in \emph{equilateral normal form} if there is an $r \in (0,\frac{\sqrt2}2)$ so that $r_i = r$ for all the non-inner positive-radius of $\cY$.%

\end{definition}
In Definition \ref{def:equilateral-normal-form}, if $\cY$ does not have an inner center, then all centers are considered non-inner. If $\cY$ has no positive-radius non-inner centers, then it is automatically in equilateral normal form. The following proposition shows why equilateral normal form is called equilateral. It also gives a geometric description for how spacings in equilateral normal form look.

\begin{proposition}[Characterization of Equilateral Normal Form]
\label{prop:characterization-of-equilateral-normal-form}
    Let $\cY \in \mps(I)$. Let $k'$ denote the number of classes with radius 0, and $k = I - k'$ the number of classes of with positive-radii. The following two are equivalent, up to translation/rotation.

    \begin{enumerate}
        \item $\cY$ is in equilateral normal form.
        \item One of the following two holds.
        \begin{enumerate}
            \item $\cY \in \mps_{=}$, the centers of $\cY$ lie at $\bszero^n$, and $r_i = \frac{\sqrt2}2$ for all $i$.
            \item $\cY \in \mps_{\neq}$, there is an $r \in (0,\frac{\sqrt{2}}2)$ so that the following holds
            \begin{enumerate}
                \item The centers with zero-radius lie at $E_{1,k'}\times \bszero^{n-k'-1} \times \alpha_0$ where 
                \begin{align}
                    \alpha_0 = \frac{1}{2(k'+1)}\sqrt{\frac{2k(k'+1)}{1 - 2r^2k'-2r^2 + k + k'}}.
                \end{align}
                \item All non-inner class of $\cY$ have radius $r$, and the centers lie at $\bszero^{k'-1}\times E_{\ell,k-1}\times- \alpha_{>0}$ where $\ell = \sqrt{1 - 2r^2}$ and
                \begin{align}
                    \alpha_{>0} &= \frac{1 - 2r^2}{2k}\sqrt{\frac{2k(k'+1)}{1 - 2r^2k'-2r^2 + k + k'}}.
                \end{align}
                \item The inner class of $\cY$ lies at $\bszero^n$ and has radius 
                \begin{align}
                    \sqrt{\frac12 + \frac{1 - 2r^2}{2(k+k'-2k'r^2 - 2r^2 + 1)}}.
                \end{align}
            \end{enumerate}
        \end{enumerate}
    \end{enumerate}
\end{proposition}

\begin{proof}
    It is clear that $2$ implies $1$, as the locations of the zero-radius and non-inner positive-radii centers lie on equilateral simplices.

    In proving that $1$ implies $2$, there are two cases which correspond to condition $2.(a)$ and $2.(b)$.
    
    The $2.(a)$ case is when any centers of $\cY$ align. In that case, all centers lie at the origin by Corollary \ref{cor:if-two-centers-coincide-in-a-maximal-spacing-they-all-do}. The formula $r_i = \frac{\sqrt{2}}2$ follows from Corollary \ref{cor:determination-of-center-coinciding-of-maximal-spacings} when $I > 2$. If $I = 2$, then $r_1 = r_2$ implies $r_1 = r_2 = \frac{\sqrt2}2$.

    In the $2.(b)$ the claim follows from computing the extent of everything in sight and applying the various identities developed in Section \ref{sec:gluing-equidistant-spacings-together}.
    
    Let $\cY|_0$ denote the classes of $\cY$ that have zero-radii, $\cY|_{>0}$ the \emph{non-inner} classes of $\cY$ that have positive-radii, and $Y_I$ denote the inner classes of $\cY$. For the moment, suppose that $k \geq 1$, and $k' \geq 2$. That $\ell = \sqrt{1 - 2r^2}$ follows from $\ell^2 = \norm{c_i - c_j}^2 = 1-2r^2$ where $c_i,c_j$ are non-inner positive-radii centers. The fact that $c_I$ and the glue sites of $\cY|_0$, $\cY|_{>0}$ are colinear follows from Lemma \ref{lem:recoloring-trick} and Corollary \ref{cor:leave-out-rank-and-max-spacing}. The formula for $\alpha_0, \alpha_{>0}$ and $r_I$ follows from verifying the following identities.
    \begin{align}
        \label{eqn:proof:prop:characterization-of-equilateral-normal-form:1}
        \extent^2 \cY|_0 &= \frac{k'}{2(k'+1)}\\ \label{eqn:proof:prop:characterization-of-equilateral-normal-form:2}
        \extent^2{\cY|_{>0}} &= \frac{2r^2 + k - 1}{2k}\\\label{eqn:proof:prop:characterization-of-equilateral-normal-form:3}
        1 - \extent^2\cY|_0 - \extent^2\cY|_{>0}&=\frac{1 - 2r^2k' - 2r^2 + k + k'}{2k(k'+1)}.
    \end{align}
    Equations \ref{eqn:proof:prop:characterization-of-equilateral-normal-form:1} and \ref{eqn:proof:prop:characterization-of-equilateral-normal-form:2} follow from Corollary \ref{cor:calc-3}. Equation \ref{eqn:proof:prop:characterization-of-equilateral-normal-form:3} follows from direct computation. Then a direct computation verified the following equations
    \begin{align}
        \label{eqn:proof:prop:characterization-of-equilateral-normal-form:4}
        \alpha_0 &= \frac{1}{2(k'+1)}\sqrt{\frac{2k(k'+1)}{1 - 2r^2k'-2r^2 + k + k'}}\\
        \label{eqn:proof:prop:characterization-of-equilateral-normal-form:5}
        \alpha_{>0} &= \frac{1 - 2r^2}{2k}\sqrt{\frac{2k(k'+1)}{1 - 2r^2k'-2r^2 + k + k'}}\\
        \label{eqn:proof:prop:characterization-of-equilateral-normal-form:6}
        r_I &= \sqrt{\frac12 + \frac{1 - 2r^2}{2(k+k'-2k'r^2 - 2r^2 + 1)}}
    \end{align}
    where Equations \ref{eqn:proof:prop:characterization-of-equilateral-normal-form:4} and \ref{eqn:proof:prop:characterization-of-equilateral-normal-form:5} follow from Proposition \ref{prop:calc-2} and \ref{eqn:proof:prop:characterization-of-equilateral-normal-form:6} follow from solving 
    \begin{align}
        1 = r_I^2 + \alpha_0^2 + \extent^2\cY|_0
    \end{align}
    for $r_I$.

    Now we consider the case when $k' = 0$ or $k = 1$. In this case, the statement that the centers lie on an equilateral simplex is vacuous. We use the convention that a single point is an equilateral 0-dimensional simplex $E_{\ell,0}$ for any $\ell$. Obviously, $\extent E_{\ell,0} = \sqrt{1 - \ell^2}$. In other words, $\extent {\cY|_{0}} = 0$ when $\cY|_0$ contains one class and $\extent{\cY_{>0}} = r$ when $\cY|_{>0}$ contains one class. With this convention, we see that Equations \ref{eqn:proof:prop:characterization-of-equilateral-normal-form:1} and \ref{eqn:proof:prop:characterization-of-equilateral-normal-form:2} hold nevertheless, and so the rest of the proof follows.
\end{proof}

As is shown later in Example \ref{examp:enumeration-of-equidistant-spacings-in-r-3}, there are ten elements of $\mps(I,\R^n)$ that are non-isometric when $n \leq 3$. Only one has a degree of freedom. Equilateral normal form has one degree of freedom (the choice of $r$), and so all elements of $\mps(I,\R^n)$ are already in equilateral normal form when $n \leq 3$.

When $n \geq 4$, $\mps(I,\R^n)$ is much larger and `most' of its elements are not in equilateral normal form.

Proposition \ref{prop:characterization-of-equilateral-normal-form} is useful as it links the abstract geometric definition of equilateral normal form to the concrete algebra of the radius, centers, and extents. We now present a theorem that all maximal equidistant spacings may be put into equilateral normal form via a sequence of isometries.

\begin{theorem}[Maximal Equidistant Spacings have Equilateral Normal Forms]
    \label{thm:maximal-equidistant-spacings-have-equilateral-normal-forms}
    Let $\cY \in \mps$. Then, there is an isometry $f$ of $\cY$ so that $f\cY$ is in equilateral normal form.  
\end{theorem}

\begin{proof}
    Consider the case when $\cY$ has (any two) centers coincide. Then all of them do, by Corollary \ref{cor:if-two-centers-coincide-in-a-maximal-spacing-they-all-do}. Let $T$ be the translation that puts the centers of $\cY$ at $\bszero$. If $I > 2$, then $T\cY$ is in equilateral normal form, as $r_i = r_j$ by Corollary \ref{cor:determination-of-center-coinciding-of-maximal-spacings}. If $I = 1$, $r_i = r_j$ holds vacuously. If $I = 2$, then there is a squash and stretch $f$ that puts $r_1 = r_2$. In any case, $fT\cY$ is in equilateral normal form.

    Suppose that no centers of $\cY$ coincide. The proof strategy is to give a sequence of isometries so that Proposition \ref{prop:characterization-of-equilateral-normal-form} point 2 holds.
    
    Let $r \in (0,\frac{\sqrt 2}2)$ and 
    \begin{align}
        \cY = \cY_{r=0} \vee \cY_{r>0}
    \end{align}
    where $\cY_{r=0}$ are classes of $\cY$ with radius zero, $\cY_{r>0}$ are the classes of $\cY$ with radius $> 0$. Let $k'$ be the number of classes of $\cY$ with radius 0, and $k$ the number of classes with radius $>0$. Note that $k+k'= I$.

    For ease of notation, let $1,\dots,k-1$ be the positive-radii classes of $\cY$, and $k$ the inner class. For each $i = 1,\dots,k-1$, we may apply Lemma \ref{lem:squash-stretches-of-maximal-positive-radius-equidistant-spacings} and obtain an inner-outer squash and stretch $f_i$ of $f_{i-1}\dots f_1\cY$ so that the radius of the $i$'th class of $f_if_{i-1}\dots f_1\cY$ is $r$. Let $f \coloneqq f_{k-1}\dots f_1\cY$.
    
    Let $c^*$ denote the $k$'th center of $f\cY$ and $T$ denote the translation operator $y \mapsto y - c^*$ and $\cY^{(1)}\coloneqq Tf\cY$. Let $A_{=0}\coloneqq \aff{\abs{\cY^{(1)}_{0}}}$ and $A_{>0}\coloneqq \aff{\abs{\cY^{(1)}_{>0}}}$ where $\cY^{(1)} = \cY^{(1)}_{0} \vee \cY^{(1)}_{>0}\vee \cY^{(1)}_{\mathrm{inner}}$ and $\cY^{(1)}_{0}$ are the radii zero classes, and $\cY^{(1)}_{>0}$ are the positive-radii classes (without the inner class), and $\cY^{(1)}_{\mathrm{inner}}$ is the inner class. Then by Lemma \ref{lem:recoloring-trick} and Proposition \ref{prop:tri-part-ortho}, we have that $A_{>0} \subset A_{0}^\perp$ and $A_{0} \subset A_{>0}^\perp$.  By orthogonality of $A_{0}$ and $A_{>0}$, there is a rotation $Q$ of $\R^n$ so that 
    \begin{align}
        QA_0 &\subset \R^{k'-1}\times \bszero^{n-k'-1} \times [0,+\infty)\\
        QA_{>0} &\subset \bszero^{k'-1}\times \bbR^{n-k'-1} \times (-\infty,0).
    \end{align}
    Let $\cY^{(2)}\coloneqq QTf\cY$. The claim is that $\cY^{(2)} \cong \cY$, and $\cY^{(2)}$ is in equilateral normal form.

    Clearly, $\cY^{(2)} \cong \cY$, as $\cY^{(2)} = QTf\cY$. Let $r^{(2)}_i$ and $c^{(2)}$ denote the radius and centers of $\cY^{(2)}$. By the construction of $f$, $r^{(2)}_1,\dots,r^{(2)}_{k-1} = r$. Hence, by Theorem \ref{thm:radius-recipe}, $\norm{c_i^{(2)} - c_j^{(2)}}^2 = 1 - 2r^2$, so the positive-radius, non-inner centers of $\cY^{(2)}$ lie on an equilateral simplex. Hence, $\cY^{(2)}$ is in equilateral normal form.

\end{proof}

Theorem \ref{thm:maximal-equidistant-spacings-have-equilateral-normal-forms} says that maximal equidistant spacings always have equilateral normal forms. It turns out, such forms are invariant under isometries, as shown by the following proposition. %

\begin{proposition}[Invariance of Equilateral Normal Form]
    \label{prop:invariance-of-equilateral-normal-form}
    Let $\cY,\cY'\in \ps$ be such that there is an isometry $f$ so that $f\cY = \cY'$. Then the equilateral normal forms of $\cY$ and $\cY'$ are the same.
\end{proposition}

\begin{proof}
    Clearly, all the isometries preserve the number of classes of a equidistant spacing. Also clearly, classes with radius $>0$ are mapped to classes with radius $>0$ via isometries. Therefore, two maximal equidistant spacings must have the same number of classes total and the same number of classes of radius 0 and radius $>0$. Hence the $k$ and $k'$ in the notation of Proposition \ref{prop:characterization-of-equilateral-normal-form} are the same. So the centers of their equilateral normal forms are the same.
\end{proof}

\subsection{Signatures of Equidistant Spacings}
\label{sec:signatures-of-equidistant-spacings}
Proposition \ref{prop:invariance-of-equilateral-normal-form} shows that normal forms are invariant under isometries. A natural question, is if this invariant characterizes isometry classes. That is, are two maximal equidistant spacings isometric if the centers of their equilateral normal forms coincide? This is true, subject to a natural dimension counting condition. To make the statement of the resulting theorem as clean as possible, we introduce the following definition. It is a combinatorial description of equidistant spacings which have a equilateral normal.

\begin{definition}[Signature of Equidistant Spacing]
    \label{def:signagure-of-equidistant-spacing}
    We call a pair of a natural number and list of natural numbers $(m,(d_1,\dots,d_k))$ where $m,d_1,\dots,d_k \in \bbN$ a \emph{signature} if $d_i \geq d_{i+1}$ for each $2 \leq i \leq k-1$. We denote the set of signatures by $\cS$. We call $m$ the \emph{zero-radius term} of the signature, $d_1,\dots,d_k$ the \emph{positive-radius terms} of the signature, and $d_1$ the \emph{inner term}.

    If $\cY \in \ps$ is in equilateral normal form, we say that the \emph{signature of $\cY$} denoted by $\sig \cY$ is given by 
    \begin{align}
        (\#(r = 0),\dim A_{1},\dim{A_2},\dots\dim A_{k})
    \end{align}
    where $\#(r = 0)$ is the number of zero-radii centers of $\cY$ and 
    \begin{enumerate}
        \item if $\cY$ has centers that coincide, $A_1,\dots,A_k$ are ordered so that $\dim{A_1} \geq \dots \geq \dim A_{k}$,
        \item if $\cY$ has centers that do not coincide, then $1$ is the inner class of $\cY$ and $A_2,\dots,A_k$ are ordered so that $\dim{A_2} \geq \dots \geq \dim A_{k}$.
    \end{enumerate}

    If $\cY \in \ps$ is isometric to a $\cY' \in \ps$ where $\cY'$ is in equilateral normal form, then we declare $\sig \cY \coloneqq \sig \cY'$.

    We define the map $\mathrm{sig} \colon \mps \to \cS$ by $\cY \mapsto \sig \cY$ as the map that associates equidistant spacings to signatures.

    We call $m + \qsum ik d_i$ the \emph{sum} of signature $s$.
\end{definition}

Note that Definition \ref{def:signagure-of-equidistant-spacing} only makes sense if signatures are invariant under isometries of equidistant spacings. This is the case for maximal equidistant spacings, which the following lemma shows.

\begin{lemma}[Signatures are Invariant Under Isometries]
    \label{lem:signatures-are-invariant-under-isometries}
    Let $\cY \in \mps$ and $f$ be an isometry. Then, $\sig \cY = \sig f\cY$.%
\end{lemma}

\begin{proof}
    First, we prove a helper result that if $\cY, \cY' \in\mps$ are isometric and in equilateral normal form, then $\sig \cY = \sig \cY'$. This, combined with Theorem \ref{thm:maximal-equidistant-spacings-have-equilateral-normal-forms}, implies the lemma.
    
    Let $\cY, \cY' \in\mps$ and $\cY \cong \cY'$. Then $\cY$ and $\cY'$ have the same centers of their equilateral normal forms by Proposition \ref{prop:invariance-of-equilateral-normal-form}. By the same arguments as in the proof of Proposition \ref{prop:invariance-of-equilateral-normal-form}, the length and radius 0 terms of their signatures are also the same. 

    Note that if $f$ is a rotation or squash and stretch of $Y_i$, then $\dim Y_i = \dim fY_i$. Further, if $f$ is a recoloring of $\cY$, then the total number of classes $\cY$ and $f\cY$ are the same. Thus, after sorting the dimensions of the non-zero-radii classes of $\cY$, it must agree with $\cY'$. Hence, their signatures agree in the non-zero component as well. So $\sig \cY = \sig \cY'$.

\end{proof}

\begin{proposition}[Determination of Coinciding Centers from Signature]
    \label{prop:determination-of-coinciding-centers-from-signature}
    Let $\cY \in \mps(I,\R^n)$ and $s = \sig \cY$ where  $s = m;d_1,\dots,d_k$. Then $\cY \in \mps_=$ if and only if $n = \qsum ik d_i$.
\end{proposition}

\begin{proof}

    Let $\cY \in \mps(I,\R^n)$. By Corollary \ref{cor:dim-of-tilde-A}
    \begin{align}
        n = \dim \tilde A + \qsum iI \dim A_i = \dim \tilde A + \qsum ik d_i.\label{eqn:prop:determination-of-coinciding-centers-from-signature}
    \end{align}

    If $\cY$ has centers that coincide, then $\dim \tilde A = \dim \aff{c_1,\dots,c_I} = 0$, so $n = \qsum ik  d_i$.

    If $\cY \in \mps(I,\R^n)$ and $n = \qsum ik d_i$, then Eqn. \ref{eqn:prop:determination-of-coinciding-centers-from-signature} implies that $0 = \dim \tilde A$, and so $c_1 = \dots = c_I$.
\end{proof}

As its name suggests, Proposition \ref{prop:determination-of-coinciding-centers-from-signature} shows that it may be determined if a maximal equidistant spacing has centers that coincide from their signatures alone. The following lemma is a consequence of Lemma \ref{lem:signatures-are-invariant-under-isometries} and is useful for subsequent calculations.

\begin{lemma}[$\sig \cY = \sig \cY'$ Implies the Equilateral Normal Form of $\cY$ and $\cY'$ are Isometric by Rotation and Recoloring]
    \label{lem:signatures-coincide-means-isometric-by-rotation}
    Let $\cY,\cY' \in \mps(\R^n,I)$, $\sig \cY = \sig \cY'$, both be in equilateral normal form. Then the following holds.
    \begin{enumerate}
        \item If either (i) $I \geq 3$ or (ii) $\cY \in \mps_{\neq}$, then $\cY$ and $\cY'$ are isometric by rotation and recoloring.
        \item If $I \leq 2$ and $\cY \in \mps_=$ and $\sig \cY = \sig \cY'$ then $\cY$ and $\cY'$ are isometric by rotation, recoloring and squash and stretch.
    \end{enumerate}
\end{lemma}
\begin{proof}
    There are two cases. We first consider when $\cY \in \mps_=$. Note that if the centers of $\cY$ coincide, then by $\sig \cY = \sig \cY'$ and Proposition \ref{prop:determination-of-coinciding-centers-from-signature} the centers of $\cY'$ coincide as well. The reverse is true by the same argument.
    
    First, suppose that $I \geq 3$. By Corollary \ref{cor:determination-of-center-coinciding-of-maximal-spacings}, $r_i = r'_i = \frac{\sqrt{2}}2$ for each $i$. By Lemma \ref{lem:ortho-theorem}, the $V_i$ (resp. $V_i'$) are pairwise orthogonal and direct sum. Further, because $\sig \cY = \sig \cY'$ we may via a recoloring $\sigma$, order the class of $\cY$ and $\cY'$ so that $\dim V_i = \dim V'_i$. So Lemma \ref{lem:existence-of-rotation-that-aligns-subspaces} applies, and there is a rotation so that $\cY = U\cY'$.

    Now we consider the case when $\cY \in \mps_{\neq}$. Define $V_1,\dots,V_I,V_{I+1}$ as follows. For $i = 1,\dots,I$, define $V_i$ as the space of differences of $A_i$ and $V_{I+1}$ the space of differences of $\tilde A$. Define $V'_1,\dots,V'_{I+1}$ likewise. Note that the $V_i$'s (resp. the $V_i'$s) direct sum by Proposition \ref{prop:tri-part-ortho}, and so Lemma \ref{lem:existence-of-rotation-that-aligns-subspaces} applies, and there is a rotation so that $V_i = UV_i$.
    By Theorem \ref{thm:radius-recipe} there is a recoloring (bijection) $\sigma \colon \set{1,\dots,k} \to \set{1,\dots,k}$ so that 
    \begin{align}
        S^{n-1}_{r_i}(c_i)\cap A_i = S^{n-1}_{r_{\sigma(i)}'}(c_{\sigma(i)}')\cap A_i'.
    \end{align} 
    By maximality, this means that $Y_i = Y_{\sigma(i)}'$
    Note that $r_i = r'_{\sigma(i)}$ as both are in equilateral normal form. Thus, $\Sigma U\cY = \cY'$.

    Finally, suppose that $I = 2$ and the centers of $\cY$ coincide. Then, Equation \ref{eqn:center-condition} in Theorem \ref{thm:radius-recipe} only has one case, when $i = 1$ and $j = 2$, in which case the equation reads
    \begin{align}
        \norm{c_1 - c_2} = 0 = \sqrt{1 - r_i^2 - r_j^2}. 
    \end{align}
    Hence, there is a squash and stretch of $\cY$ and $\cY'$ that puts $r_1 = r_2 = \frac{\sqrt{2}}2 = r_1' = r_2'$ directly. Then, apply the same arguments as in the case when $I > 2$ and the centers coincide.
\end{proof}

\begin{remark}
    Note that the condition that $I > 2$ or the centers of $\cY$ do not coincide in Lemma \ref{lem:signatures-coincide-means-isometric-by-rotation} is necessary. Consider the exemplary $\cY,\cY' \in \mps(2,\R^2)$ illustrated by Figures \ref{fig:squash-stretch-iso-fig:Y1} and \ref{fig:squash-stretch-iso-fig:fY1}. It is obvious there is no rotation that sends $\cY_1$ to $f\cY_1$.
\end{remark}

With the definition of signatures, we may prove the following theorem says that isometry classes of maximal equidistant spacings are given by their signatures.

\begin{theorem}[Maximal Equidistant Spacings are Isometric if and only if they have the same Signature]
    \label{thm:maximal-equidistant-spacings-are-isometric-iff-same-signature}
    Let $\cY,\cY' \in \mps$. Then, $\cY \cong \cY'$ if and only if $\sig \cY = \sig \cY'$.
\end{theorem}
\begin{proof}
    The forward direction follows immediately from Lemma \ref{lem:signatures-are-invariant-under-isometries}, which implies that $\sig \cY = \sig \cY'$.

    For the reverse direction, let $\sig \cY = \sig \cY'$, and let $f$ (resp. $f'$) be isometries that takes $\cY$ (resp. $\cY'$) to its equilateral normal form. Then, by Lemma \ref{lem:signatures-coincide-means-isometric-by-rotation}, there is a rotation $U$, recoloring $\sigma$ and squash and stretches $\tilde f$, $\tilde f'$ (which are the identity unless $I = 2$) so that $\tilde f \sigma Uf \cY = \tilde f' f'\cY'$, and so $f'{}^{-1}\tilde f'{}^{-1} \tilde f \sigma Uf\cY = \cY'$. Hence, $\cY \cong \cY'$.
\end{proof}

\begin{remark}[Geometric Interpretation of Signatures]
    \label{rmk:geometric-interpretation-of-signatures}
    In light of Theorem \ref{thm:maximal-equidistant-spacings-are-isometric-iff-same-signature}, it is clear that the signature of a maximal equidistant spacing is a combinatorial descriptor of its isometry class. It also has a straight forward geometric meaning. If $\sig \cY = (m,(d_1,\dots,d_k))$ then $m$ is the number of zero-radii centers of $\cY$, The number $k$ counts the number of classes of $\cY$ with radius $>0$, $d_1$ is the dimension of the inner class and $d_i$ is the dimension of class $Y_i$, once the non-inner classes have been ordered in decreasing order of dimension. 
    
    The reason the signatures are written as $(m,(d_1,\dots,d_k))$ rather than as a sequence $d_0,d_1,\dots,d_k$ is to make clear that $m$ has a different interpretation than the other $d_i$. $m$ counts the \emph{number} of classes, whereas $d_i$ counts the \emph{dimension} of individual classes.

    The geometric interpretation is especially useful, because it applies even if $\cY$ is not in equilateral normal form.
\end{remark}

\begin{proposition}[Properties of Signatures]
    \label{prop:properties-of-signatures}
    Let $\cY \in \mps(I,\R^n)$, $n \geq 1$ and $\sig \cY = (m,(d_1,\dots,d_k))$. Then the following holds.
    \begin{enumerate}
        \item $k \geq 1$.
        \item $I = m + k$.
        \item If $\cY \in \mps_=$ then $n = \qsum ikd_k$ and $m + k \geq 2$. Further, if $I = 2$ then $m \leq 1$ and if $I > 2$ then $m = 0$.
        \item If $\cY \in \mps_{\neq}$ then $m + k \geq 3$ and $n = m + k - 2 + \qsum ikd_k$.
    \end{enumerate}
\end{proposition}

\begin{proof}
    Let $\cY$ be in equilateral normal form. 
    \begin{enumerate}
        \item If $I = 1$, $\cY$ is not maximal. If $I \geq 2$, then one of the classes must have radius $>0$, so $k > 0$.
        
        \item The total number of classes in $\cY$ are the radius $0$ classes, $m$, plus the positive-radius classes $k$.
        
        \item This is precisely Proposition \ref{prop:determination-of-coinciding-centers-from-signature} plus the information that $m = 0$ is necessary if $I > 2$. This point follows from Corollary \ref{cor:determination-of-center-coinciding-of-maximal-spacings}.
        \item First we show that $\cY \in \mps_{\neq}$ implies $m + k \geq 3$. This follows from Eqn. \ref{eqn:leave-out-rank:1} in Corollary \ref{cor:leave-out-rank-and-max-spacing}. By Corollary \ref{cor:dim-of-tilde-A},
        \begin{align}
            n &= \dim \tilde A + \qsum iI \dim A_i \\
            &= I-2 + \qsum ik \underbrace{\dim A_i}_{d_i} + \sum_{i = k+1}^I \underbrace{\dim A_i}_{ = 0}\\
            &= I - 2 + \qsum ikd_i,
        \end{align}
        where $k,k',E_{\ell,k}$, and $E_{1,k'}$ refer to the quantities in Definition \ref{def:equilateral-normal-form}.

    \end{enumerate}
\end{proof}

Proposition \ref{prop:properties-of-signatures} shows that signatures must satisfy some constraints to correspond to maximal equidistant spacings. The next proposition shows that these conditions exactly characterize signatures of maximal equidistant spacings.

\begin{proposition}[Equidistant Spacings from Signatures]
    \label{prop:equidistant-spacings-from-signatures}
    Let $s = (m,(d_1,\dots,d_k))$ be a signature, $n \geq 1$ and $k \geq 1$. Then the following holds.
    \begin{enumerate}
        \item If $m+k \geq 3$ there is a $\cY \in \mps_{\neq}(m + k,\R^{m+k-2+\qsum ik d_i})$ so that $\sig \cY = s$.
        \item If $m = 0$ and $k \geq 1$, then there is a $\cY \in \mps_=(k,\R^{\qsum ik d_i})$ so that $\sig \cY = s$.
        \item If $m = 1$ and $k \leq 1$, then there is a $\cY \in \mps_=(2,\R^{d_1})$ if $k = 1$ and $\cY \in \mps_=(1,\R^{0})$ if $k = 0$ so that $\sig \cY = s$.
    \end{enumerate}
\end{proposition}

\begin{proof}
    Let $I \coloneqq m + k$. The idea for both parts of the proof is to construct a maximal equidistant spacing in equilateral normal form concretely by appealing to Proposition \ref{prop:characterization-of-equilateral-normal-form}.

    First, we show 1. Define $\cY$ a putative maximal equidistant spacing as follows. Let $r \in (0,\frac{\sqrt2}2)$ and define 
    \begin{align}
        \alpha_0 &= \frac{1}{2(m+1)}\sqrt{\frac{2k(m+1)}{1 - 2r^2k'-2r^2 + k + m}}\\
        \alpha_{>0} &= \frac{1 - 2r^2}{2k}\sqrt{\frac{2k(m+1)}{1 - 2r^2k'-2r^2 + k + m}}.
    \end{align}

    Let $n \coloneqq m+k-2 + \qsum ikd_k$. For each $i = 1,\dots,m+k$, define $n_i \coloneqq d_i$ if $i \leq k$ and $n_i \coloneqq 0$ if $i > k$. Let $e_1,\dots,e_{n}$. %
    For each $i = 1,\dots,m+k$, define 
    \begin{align}
        \label{eqn:proof:prop:equidistant-spacings-from-signatures:v-i-def}
        V_i \coloneqq \Span{e_{m+k + 1 + \qsum j{i-1} d_j},\dots,e_{m+k + \qsum j{i} d_j}}.
    \end{align}
    Note that $\Span{e_1,\dots,e_{m+k}} \bigoplus_{i = 1}^{m+k} V_i = \R^{m+k-2 + \qsum ik d_k}$, where the direct sum comes from orthonormality and spanning comes from dimension counting. Define $c_1,\dots,c_{k-1}$ to be the $k-1$ vertices of $E_{\ell,k-1}$ embedded in $\R^{k-2} \times \set{\bszero^{n + 1 - k}} \times \set{\alpha_{> 0}}$, and let $c_{k+1},\dots,c_{k+m}$ be the $m$ vertices of $E_{1,m}$ embedded in $\set{\bszero^{k-2}} \times \R^{m-1} \times \set{\bszero^{n + 2 - k - m}}  \times \set{\alpha_{ 0}}$, and put $c_k = \bszero^n$. For each $i = 1,\dots,m+k$, put
    \begin{align}
        A_i = c_i + V_i, \quad Y_i \coloneqq S^{n-1}_{r}(c_i) \cap A_i.
    \end{align}
    where 
    \begin{align}
        r_i \coloneqq \begin{cases}
            r & \text{if } i < k\\
            \sqrt{\frac12 + \frac{1 - 2r^2}{2(k+k'-2k'r^2 - 2r^2 + 1)}} &\text{if } i = k\\
            0 &\text{if } i > k
        \end{cases}.
    \end{align}

    We now show that $\cY\coloneqq \bigsqcup_{i = 1}^{m+k}Y_i \in \mps(m + k,\R^{m+k-2+\qsum ik d_k})$. First we show that $\cY \in \ps$ by verifying Corollary \ref{cor:characterization-of-equidistant-spacings}. Clearly, points 2. and 3. are satisfied. All that's left is to verify Eqn. \ref{eqn:center-compatability-condition}. There are five cases to consider. For these five cases, we use the shorthand that $c^*_{>0} \coloneqq \bszero^{n-1}\times \alpha_{>0}$ and $c^*_0 \coloneqq \bszero^{n-1} \times -\alpha_0$.
    \begin{enumerate}
        \item[Case 1.] $i,j$ correspond to positive-radii non-inner classes. Then
        \begin{align}
             \norm{c_i - c_j}^2 = 1 - 2r^2.
        \end{align}
        Note that classes $i$ and $j$ have radius $r$.
        \item [Case 2.] $i,j$ correspond to zero-radii classes. Then clearly
        \begin{align}
            \norm{c_i-c_j}^2 = 1.
        \end{align}
        \item[Case 3.] Classes $i,j$ are such that $i$ is a positive-radius class, and $j$ is a zero-radius class. Then a straightforward but involved calculation shows that 
        \begin{align}
            \norm{c_i - c_j}^2 &= \norm{c_i - c^*_{>0}}^2 + \norm{c^*_{>0} - c^*_{0}}^2 + \norm{c_j - c^*_{0}}^2\\
            &=\frac{m}{2(m+1)} + (\alpha_0 + \alpha_{>0})^2 + \frac{2r^2 + k - 1}{2k}\\
            &= 1 - r^2.
        \end{align}
        \item[Case 4.] Class $i$ is a positive-radius class and $j = k$ is the inner class. Then another  straightforward but involved calculation shows that
        \begin{align}
            \norm{c_i - c_k}^2 &= \norm{c_i - c^*_{>0}}^2 + \norm{c^*_{>0} - \bszero^n}^2 + \norm{\bszero^{n} - c_k}^2\\
            &= \frac{2r^2 + k - 1}{2k} + \alpha_1^2 + r_k^2 \\
            &= 1 - r^2 - r_k^2.
        \end{align}
        \item[Case 5.] Class $i$ is a zero-radius class and $j = k$, the inner class. Then
        \begin{align}
            \norm{c_i - c_k}^2 &= \norm{c_i - c^*_{0}}^2 + \norm{c^*_{0} - \bszero^n}^2 + \norm{\bszero^{n} - c_k}^2\\
            &= \frac{m}{2(m+1)} + \alpha_0^2 + r_k^2 \\
            &= 1 - r_k^2.
        \end{align}
    \end{enumerate}

    We now verify $\cY$ is maximal. Suppose that $\cY \subseteq \cY'$ so that $\cY'$ is maximal. By 
    \begin{align}
        \label{eqn:dim-maximality}
        \aff(c_1,\dots,c_{m+k}) \bigoplus_{i = 1}^{m+k} V_i = \R^{m+k-2 + \qsum ik d_k},
    \end{align}
    we have that $I' = m+k$. Clearly, $\cY' \in \mps_{\neq}$, and so $\cY'$ satisfies Proposition \ref{prop:the-ortho-reflection-lemma-produces-orthocentric-systems}. Hence, $c'_1,\dots,c'_I$ form a simplex plus orthocenter. By Eqn. \ref{eqn:dim-maximality}, we have that $V_i \subset V'_i$ and so $\aff(c_1',\dots,c_{m+k}') \subset \aff(c_1,\dots,c_{m+k})$. Hence, $c_i = c_i'$ for each $i$. Therefore, $\cY' \subset \cY$ and $\cY$ is maximal.

    Finally, what's left is to show that $\sig \cY = s$. $\cY$ has $m$ radius zero classes (one for each vertex of $E_{1,m}$), $k$ positive-radius classes (one for each vertex of $E_{\ell,k-1}$ plus one inner class), and $d_1,\dots,d_k$ are as promised from the definition of $V_i$ in Eqn. \ref{eqn:proof:prop:equidistant-spacings-from-signatures:v-i-def}.

    Now we verify $2.$. Put $c_1 = \dots = c_{k} = \bszero$. Let $e_1,\dots,e_{\qsum ikd_k}$ be an orthonormal basis for $\R^{\qsum ikd_k}$. Put 
    \begin{align}
        V_i \coloneqq \Span{e_{1 + \qsum j{k-1}d_j},\dots,e_{\qsum j{k}d_j}}
    \end{align}
    and define
    \begin{align}
        Y_i \coloneqq S^{n-1}_{\frac{\sqrt2}2}(\bszero) \cap V_i.
    \end{align}
    $\cY \coloneqq \sqcup^{m+k}_i Y_i$. We verify that $\cY \in \ps.$ If $i \neq j$, then 
    \begin{align}
        \norm{y_i - y_j}^2 &= \norm{y_i - c_i}^2 + \norm{c_i - c_j}^2 + \norm{y_j - c_j}^2\\
        &= \frac12 + 0 + \frac12 = 1.
    \end{align}
    Hence, $\cY \in \ps$. By counting, clearly $\cY \in \ps(k,\R^{\qsum ikd_i})$. Maximality follows from the same arguments as in the first part of this proof. Finally, the centers of $\cY$ coincide as defined, and so $\cY \in \mps_=(k,\R^{\qsum ikd_i})$.

    Now we show 3. If $m = 1$, $k = 0$, then we may consider the spacing given by $Y_1 \coloneqq \set{\bszero}$ This spacing is maximal in $\bbR^0 = \set{\bszero}$. If $m = 1, k = 1$ then put $Y_1 \coloneqq \set{\bszero}$ and $Y_2 \coloneqq S^{d_1-1}_{1}(\bszero)$.%

\end{proof}

Proposition \ref{prop:equidistant-spacings-from-signatures} motivates the following definition, which describes signatures that correspond to signatures in $\ps(I,\R^n)$. 
\begin{definition}[Valid Signatures]
    \label{def:valid-signatures}
    Let $I,n \in \bbN$ be given. We call the set of $I,n$-\emph{valid signatures} denoted by $\cS(I,n)$, as the set 
    \begin{align}
        \cS(I,n) \coloneqq \set{s \in \cS \colon k,m \geq 0,\quad  m + k \leq I ,\quad  m + k + \qsum ikd_k-2 \leq n}.
    \end{align}
    When $I,n$ are clear, we refer to $\cS(I,n)$ as simply the set of valid signatures. We define two subsets of $\cS$ by $\hat\cS_=$ and $\hat\cS_{\neq}$ by
    \begin{align}
        \label{eqn:hat-sc-neq-def}
        \hat\cS_{\neq}(I,n) &\coloneqq \set{s \in \cS(I,n) \colon  m + k = I > 2,\quad m+k +\qsum ikd_k = n+2},\\
        \label{eqn:hat-sc-eq-def}
        \hat\cS_=(I,n) &\coloneqq \set{s \in \cS(I,n) \colon m + k = I,\quad \qsum ikd_k = n, \quad    \twopartpiecewise{m=0}{I \geq 1}{m=1}{I \leq 2}}.
    \end{align}
\end{definition}

With these two definitions, several previous results may be phrased more simply.
\begin{corollary}[Signatures of $\mps$]
    \label{cor:signatures-of-mps}
    The following holds.
    \begin{enumerate}
        \item Theorem \ref{thm:maximal-equidistant-spacings-are-isometric-iff-same-signature}. Isometry classes of maximal equidistant spacings are characterized by their signatures.
        \item Proposition \ref{prop:properties-of-signatures} point 2. If $\cY \in \mps_=(I,\R^n)$ then $\sig \cY \in \hat\cS_=(I,n)$.
        \item Proposition \ref{prop:properties-of-signatures} point 3. If $\cY \in \mps_{\neq}(I,\R^n)$ then $\sig \cY \in \hat\cS_{\neq}(I,n)$.
        \item Proposition \ref{prop:equidistant-spacings-from-signatures} point 1. The map 
        \begin{align}
            \sig|_{\mps_{\neq}(I,\R^n)} \colon \mps_{\neq}(I,\R^n) \to \hat\cS_{\neq}(I,n)
        \end{align} is a surjection.
        \item Proposition \ref{prop:equidistant-spacings-from-signatures} points 2., 3. and 4.
        The map 
        \begin{align}
            \sig|_{\mps_{=}(I,\R^n)} \colon \mps_{=}(I,\R^n) \to \hat\cS_{=}(I,n)
        \end{align} is a surjection.
    \end{enumerate}
\end{corollary}

\begin{proof}
    Proof of all claims follow immediately from unrolling the definitions.
\end{proof}
In light of Corollary \ref{cor:signatures-of-mps}, it makes sense to pass through the quotient and regard $\sig$ as a map between isometry classes and signatures. This is proven in the following proposition.
\begin{proposition}[$\sig$ as a Map on Isometry Classes of $\mps$]
    \label{prop:sig-as-a-map-on-iso-classes-of-mps}
    Let $\frac{\mps(I,\R^n)}{\cong}$ denote the equivalence class of $\mps(I,\R^n)$ generated by rotations, translations, recolorings and squash and stretches. Define $\frac{\mps_=(I,\R^n)}{\cong}$ and $\frac{\mps_{\neq}(I,\R^n)}{\cong}$ likewise.
    We may regard
    \begin{align}
        \sig \colon \frac{\mps(I,\R^n)}{\cong} \to \cS(I,\R^n)
    \end{align}
    and likewise 
    \begin{align}
        \sig|_{\frac{\mps_=(I,\R^n)}{\cong}}& \frac{\mps_=(I,\R^n)}{\cong} \to \hat\cS_=(I,n),\\
        \sig|_{\frac{\mps_{\neq}(I,\R^n)}{\cong}}& \frac{\mps_{\neq}(I,\R^n)}{\cong} \to \hat\cS_{\neq}(I,n).
    \end{align}
\end{proposition}
\begin{proof}
    The only thing to check is that $\sig \cY = \sig \cY'$ when $\cY \cong \cY'$ $\sig$. This is proven in Theorem \ref{thm:maximal-equidistant-spacings-are-isometric-iff-same-signature}.
\end{proof}
For ease of notation, we sometimes omit the restriction and denote $\sig|_{\frac{\mps_{=}(I,\R^n)}{\cong}}$ as $\sig$ and likewise for $\sig|_{\frac{\mps_{\neq}(I,\R^n)}{\cong}}$ when it is clear from context. 

We may now state and prove our main result.%

\begin{theorem}[$\sig$ is a Bijection on Maximal Equidistant Spacings]
    \label{thm:sig-is-a-bijection-on-maximal-equidistant-spacings} 
    All of the following maps are bijections.
    \begin{align}
        \sig\colon \frac{\mps_=(I,n)}{\cong} &\to \hat\cS_=(I,n),\\
        \sig\colon \frac{\mps_{\neq}(I,n)}{\cong} &\to \hat\cS_{\neq}(I,n),\\
        \sig\colon \bigcup_{I,n \in \bbN}\frac{\mps_=(I,n)}{\cong} &\to \bigcup_{I,n \in \bbN}\hat\cS_=(I,n),\\
        \sig\colon\bigcup_{I,n \in \bbN} \frac{\mps_{\neq}(I,n)}{\cong} &\to \bigcup_{I,n \in \bbN}\hat\cS_{\neq}(I,n).
    \end{align}
    In short, the signature map is a bijection from isometry classes to the space of compatible signatures.
\end{theorem}
\begin{proof}
    Let $I,n$ be fixed. The proof that $\sig\colon \frac{\mps_=(I,n)}{\cong} \to \hat\cS_=(I,n)$ and $\sig\colon \frac{\mps_{\neq}(I,n)}{\cong} \to \hat\cS_{\neq}(I,n)$ are bijections is identical. The map is well-defined by Proposition \ref{prop:sig-as-a-map-on-iso-classes-of-mps}. The proof that it is a surjection is Corollary \ref{cor:signatures-of-mps} points 4. and 5. The proof that it is injective on equivalence classes is Theorem \ref{thm:maximal-equidistant-spacings-are-isometric-iff-same-signature}. This proves that these two maps are bijections.

    The proof that $\sig\colon \bigcup_{I,n \in \bbN}\frac{\mps_=(I,n)}{\cong} \to \bigcup_{I,n \in \bbN}\hat\cS_=(I,n)$ and $\sig\colon\bigcup_{I,n \in \bbN} \frac{\mps_{\neq}(I,n)}{\cong} \to \bigcup_{I,n \in \bbN}\hat\cS_{\neq}(I,n)$ follows from proving that 
    \begin{align}
        \label{eqn:proof:thm:signatures-are-in-bijection-with-isometry-classes-of-maximal-equidistant-signatures:1}
        \frac{\mps_=(I,n)}{\cong} \cap \frac{\mps_=(I',n')}{\cong} = \emptyset, \quad \frac{\mps_{\neq}(I,n)}{\cong} \cap \frac{\mps_{\neq}(I',n')}{\cong}= \emptyset
    \end{align}
    and 
    \begin{align}
        \label{eqn:proof:thm:signatures-are-in-bijection-with-isometry-classes-of-maximal-equidistant-signatures:2}
        \hat\cS_{=}(I,n) \cap \hat\cS_{=}(I',n') = \emptyset, \quad \hat\cS_{\neq}(I,n) \cap \hat\cS_{\neq}(I',n') = \emptyset.
    \end{align}
    whenever $I \neq I'$ or $n \neq n'$.

    To Prove Eqn. \ref{eqn:proof:thm:signatures-are-in-bijection-with-isometry-classes-of-maximal-equidistant-signatures:1}, if $I \neq I'$, then the number of elements of $\cY \in \ps(I)$ and $\cY' \in \ps(I')$ are not the same. If $n \neq n'$ then $\abs \cY$ and $\abs \cY'$ live in different spaces. 

    To prove Eqn. \ref{eqn:proof:thm:signatures-are-in-bijection-with-isometry-classes-of-maximal-equidistant-signatures:2}, if $s \in \hat\cS_{=}(I,n) \cap \hat\cS_{=}(I',n')$ then by $I = m + k = I'$ and $n = m+k-2 + \qsum ik d_k = n'$ by Corollary \ref{cor:signatures-of-mps} points 2 and 3. The same proof works when $s \in \hat\cS_{\neq}(I,n) \cap \hat\cS_{\neq}(I,n)$.%

\end{proof}

\begin{remark}[$\hat \cS_= \cap \hat \cS_{\neq} \neq \emptyset$]
    \label{rmk:hat-s-neq-and-hat-s-eq-overlap}
    A natural followup question in light of Theorem \ref{thm:sig-is-a-bijection-on-maximal-equidistant-spacings} is if 
    \begin{align}
        \label{eqn:rmk:hat-s-neq-and-hat-s-eq-overlap:1}
        \paren{\bigcup_{I,n \in \bbN}\hat\cS_{\neq}(I,n)} \cap \paren{\bigcup_{I,n \in \bbN}\hat\cS_{=}(I,n)} \stackrel?{=} \emptyset.
    \end{align}
    If Eqn. \ref{eqn:rmk:hat-s-neq-and-hat-s-eq-overlap:1} did hold, it would imply that $\sig \colon \bigcup_{n,I} \frac{\mps(n,I)}{\cong} \to \bigcup_{n,I}\paren{\cS_=(I,n) \cup \cS_{\neq}(I,n)}$ is a bijection. It is not so. It \emph{is} true that $\hat\cS_{=}(I,n) \cap \hat\cS_{\neq}(I,n) = \emptyset$, but a direct calculations shows that
    \begin{align}
        (0,(2,1,1)) \in \hat\cS_{=}(3,4) \cap \hat \cS_{\neq}(3,5).
    \end{align}
\end{remark}

Theorem \ref{thm:sig-is-a-bijection-on-maximal-equidistant-spacings} says that equidistant spacings are summarized by understanding their signatures. Proposition \ref{prop:equidistant-spacings-from-signatures} gives simple rules for generating all signatures. By combining the two, we may generate all elements of $\mps(I,\R^n)$, up to isometry.

\begin{example}[Enumeration of Equidistant spacings in $\R^1$, $\R^2$, and $\bbR^3$]
    \label{examp:enumeration-of-equidistant-spacings-in-r-3}
    In light of Theorem \ref{thm:sig-is-a-bijection-on-maximal-equidistant-spacings}, we can enumerate the $\ps(I,\R^n)$ for a given $n$ by computing \begin{align}
        \sig{}^{-1}\paren{\bigcup_{I \in \bbN}\hat \cS_{=}(I,n) \cup \hat \cS_{\neq}(I,n)}.
    \end{align}
    One straight-forward way to do this is to choose a target sum, $N$, and use Algorithm \ref{alg:generation-of-all-hat-sc-eq} and Algorithm \ref{alg:generation-of-all-hat-sc-neq} to generate all $\hat \cS_=$ and $\hat \cS_{\neq}$ that sum to $N$, respectively. Then, for each $s \in \hat \cS_= \cup \hat \cS_{\neq}$, we may use Eqns. \ref{eqn:hat-sc-neq-def} and \ref{eqn:hat-sc-eq-def} to compute the $I$ and $n$ so that $s \in \hat \cS_=(I,n)$ and $s \in \hat \cS_{\neq}(I,n)$.

    As an example of this idea in action, in Table \ref{tab:cs-eq}, we compute all signatures that sum to four or less in $\hat \cS_=$. In Table \ref{tab:cs-neq}, we compute all signatures that sum to four or less in $\hat \cS_{\neq}$. This gives us the signatures of all $\mps(I,\R^n)$ when $n = 0$ (see Figure \ref{fig:all-mps-n-leq-3:1sc}) $n = 1$ (see Figure \ref{fig:all-mps-n-eq-1}),  $n = 2$ (see Figure \ref{fig:all-mps-n-eq-2}), and $n = 3$ (see Figure \ref{fig:all-mps-n-eq-3}).
    
    \begin{table}[]
    \centering
    \begin{tabular}{c|c|c|c|c}
         Sum& Signatures & I & $n$ s.t. $s \in \hat \cS_=(I,n)$&Figure \\\hline\hline
    4&(1,(3))&2&3&\ref{fig:all-mps-n-leq-3:1sc3}\\
    4&(0,(4))&1&4&\\
    4&(0,(3,1))&2&4&\\
    4&(0,(2,2))&2&4&\\
    4&(0,(2,1,1))&3&4&\\
    4&(0,(1,1,1,1))&4&4&\\
    3&(1,(2))&2&2&\ref{fig:all-mps-n-leq-3:1sc2}\\
    3&(0,(3))&1&3&\ref{fig:all-mps-n-leq-3:0sc3}\\
    3&(0,(2,1))&2&3&\ref{fig:all-mps-n-leq-3:0sc21}\\
    3&(0,(1,1,1))&3&3&\ref{fig:all-mps-n-leq-3:0sc111}\\
    2&(1,(1))&2&1&\ref{fig:all-mps-n-leq-3:1sc1}\\
    2&(0,(2))&1&2&\ref{fig:all-mps-n-leq-3:0sc2}\\
    2&(0,(1,1))&2&2&\ref{fig:all-mps-n-leq-3:0sc11}\\
    1&(1,())&1&0&\ref{fig:all-mps-n-leq-3:1sc}\\
    1&(0,(1))&1&1&\ref{fig:all-mps-n-leq-3:0sc1}
    \end{tabular}
    \caption{All signatures in $\bigcup_{I,n \in \bbN}\hat \cS_=(I,n)$ which sum to four or less, together with their number of classes and ambient dimension for which they are maximal. If the ambient dimension for which $s$ is maximal is $3$ or less, then there is an entry in the Figure column referencing a figure that plot of $\sig^{-1} s$.}
    \label{tab:cs-eq}
    \end{table}

    \begin{table}[]
    \centering
    \begin{tabular}{c|c|c|c|c}
         Sum& Signature, $s$ & I & $n$ s.t. $s \in \hat \cS_{\neq}(I,n)$ &Figure\\\hline\hline
        4&(3,(1))&4&3&\ref{fig:all-mps-n-leq-3:3sc1}\\
        4&(2,(2))&3&3&\ref{fig:all-mps-n-leq-3:2sc2}\\
        4&(2,(1,1))&4&4&\\
        4&(1,(2,1))&3&4&\\
        4&(1,(1,2))&3&4&\\
        4&(1,(1,1,1))&4&5&\\
        4&(0,(2,1,1))&3&5&\\
        4&(0,(1,2,1))&3&5&\\
        4&(0,(1,1,1,1))&4&6&\\
        3&(2,(1))&3&2&\ref{fig:all-mps-n-leq-3:2sc1}\\
        3&(1,(1,1))&3&3&\ref{fig:all-mps-n-leq-3:1sc11}\\
        3&(0,(1,1,1))&3&4&\\
    \end{tabular}
    \caption{All signatures $s$ in $\bigcup_{I,n \in \bbN}\hat \cS_{\neq}(I,n)$ which sum to four or less, together with their number of classes and ambient dimension for which they are maximal. If the ambient dimension for which $s$ is maximal is $3$ or less, then there is an entry in the Figure column referencing a figure that plot of $\sig^{-1} s$.}
    \label{tab:cs-neq}
    \end{table}
\end{example}

\begin{figure}
    \centering
    \includegraphics[width=.32\linewidth]{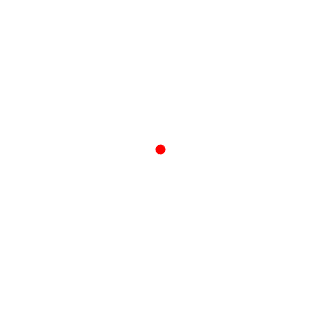}
    \caption{A plot of the only element of $\cY \in \mps$ embeddable in $\R^0$ with signature $(1,())$.}
    \label{fig:all-mps-n-leq-3:1sc}
\end{figure}

\begin{figure}
    \centering
    \begin{subfigure}{.32\linewidth} %
        \centering
        \includegraphics[width=1.\linewidth]{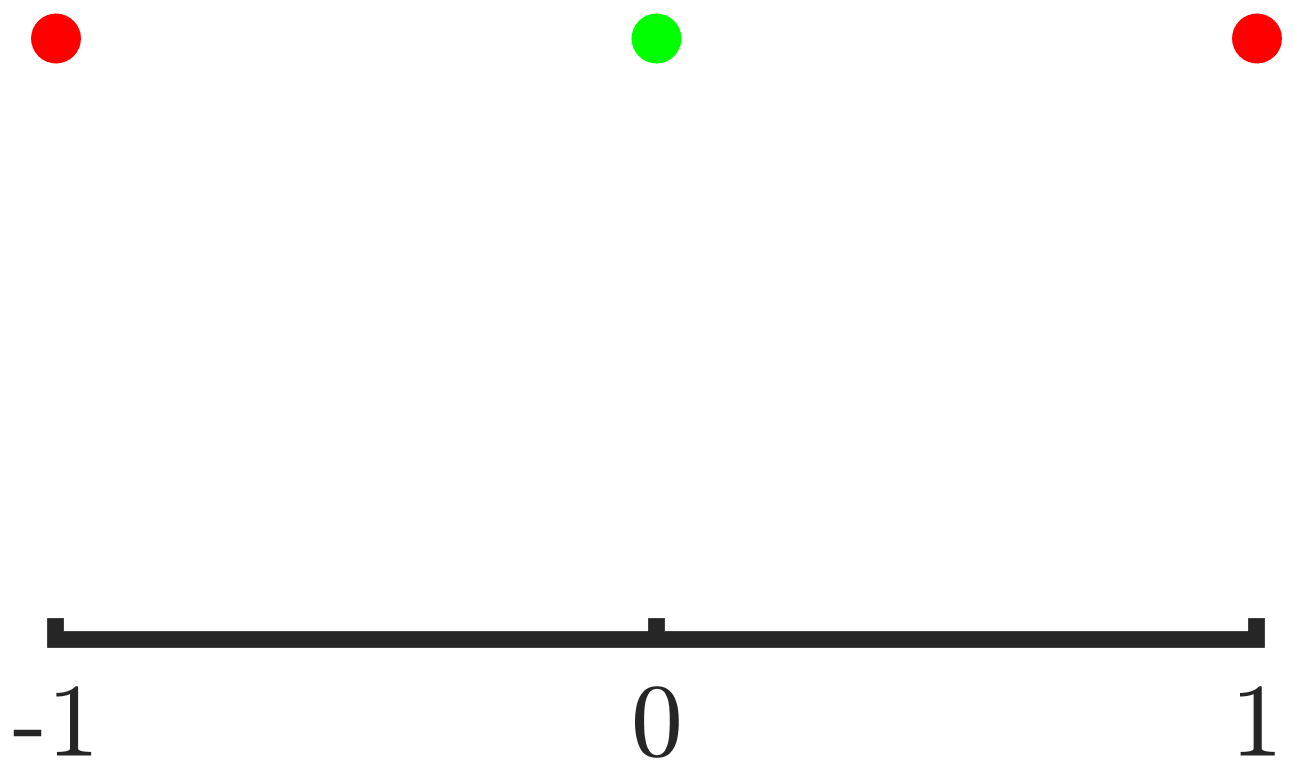}
        \subcaption{$(1,(1))$}
        \label{fig:all-mps-n-leq-3:1sc1}
    \end{subfigure}
    \begin{subfigure}{.32\linewidth} %
        \centering
        \includegraphics[width=1.\linewidth]{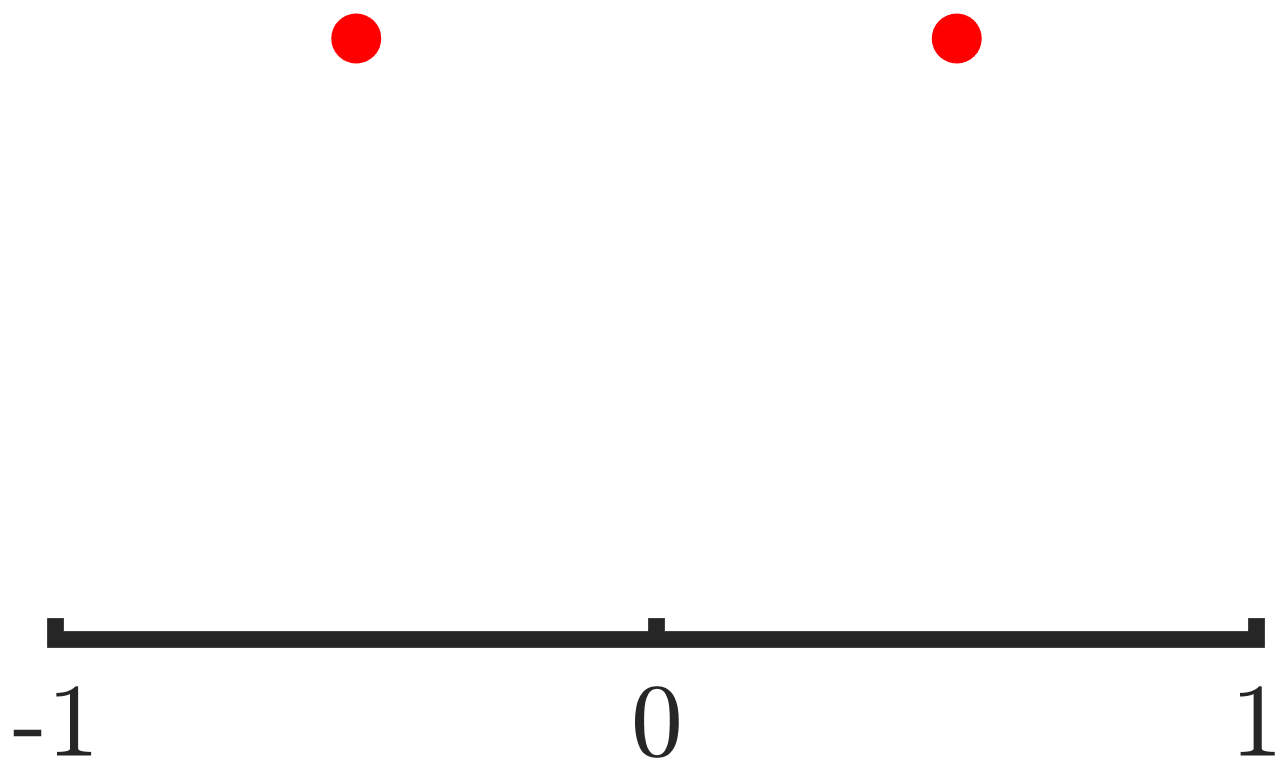}
        \subcaption{$(0,(1))$}
        \label{fig:all-mps-n-leq-3:0sc1}
    \end{subfigure}
    \caption{Plots of the two elements of $\cY \in \mps$ embeddable in $\R^1$.}
    \label{fig:all-mps-n-eq-1}
\end{figure}

\begin{figure}
    \centering
    \begin{subfigure}{.32\linewidth} %
        \centering
        \includegraphics[width=1.\linewidth]{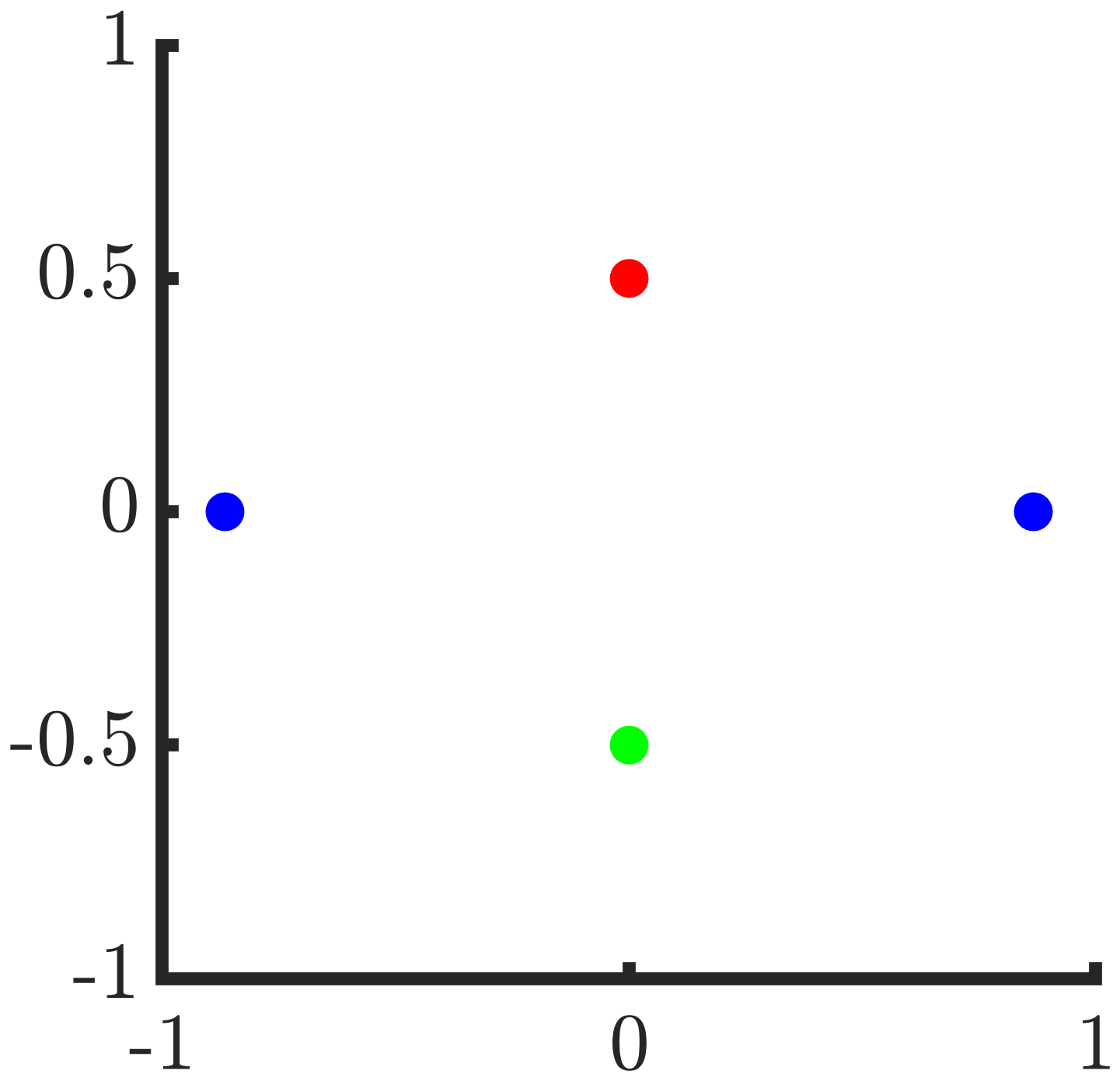}
        \subcaption{$(2,(1))$}
        \label{fig:all-mps-n-leq-3:2sc1}
    \end{subfigure}
    \begin{subfigure}{.32\linewidth} %
        \centering
        \includegraphics[width=1.\linewidth]{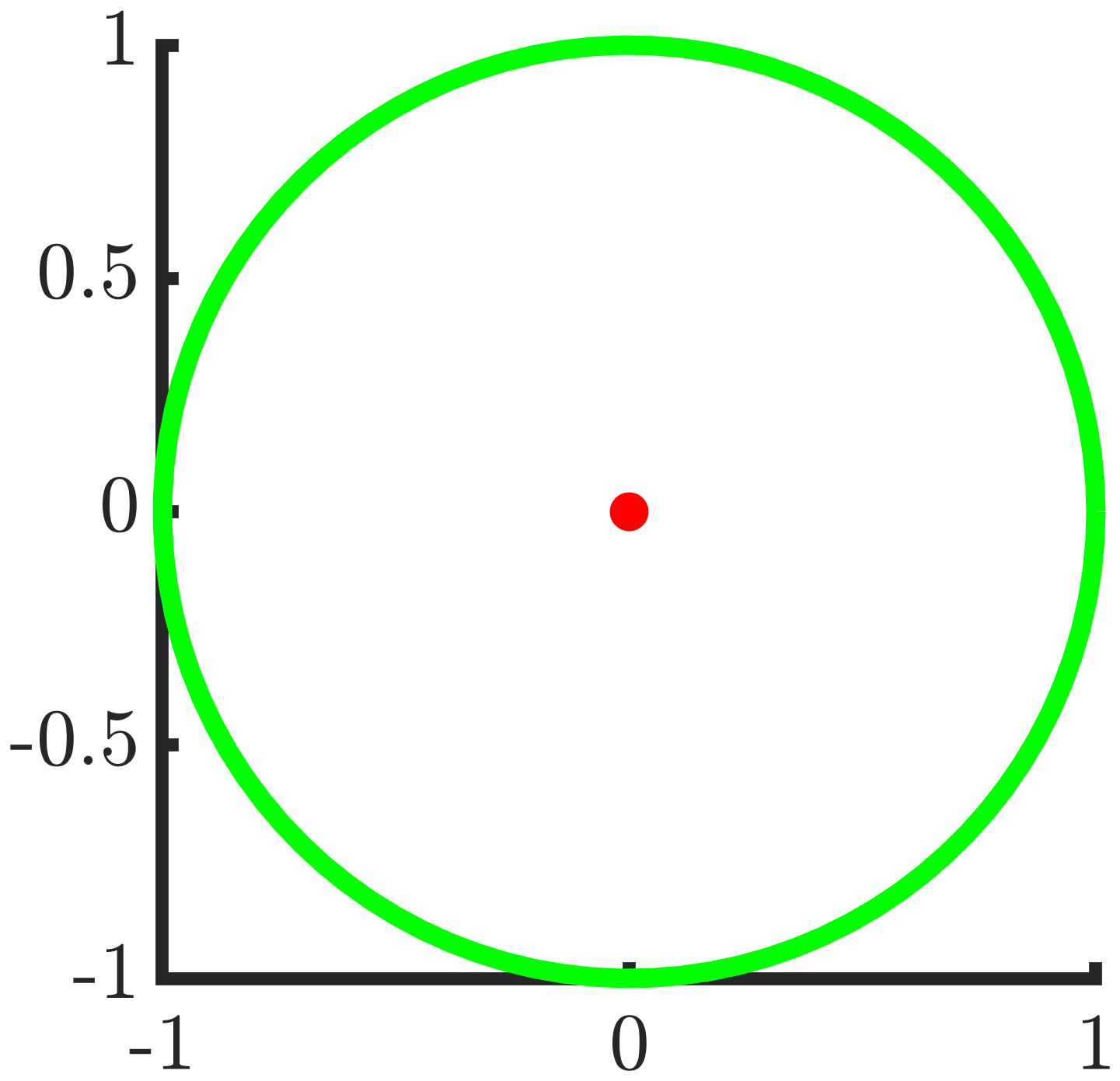}
        \subcaption{$(1,(2))$}
        \label{fig:all-mps-n-leq-3:1sc2}
    \end{subfigure}
    \begin{subfigure}{.32\linewidth} %
        \centering
        \includegraphics[width=1.\linewidth]{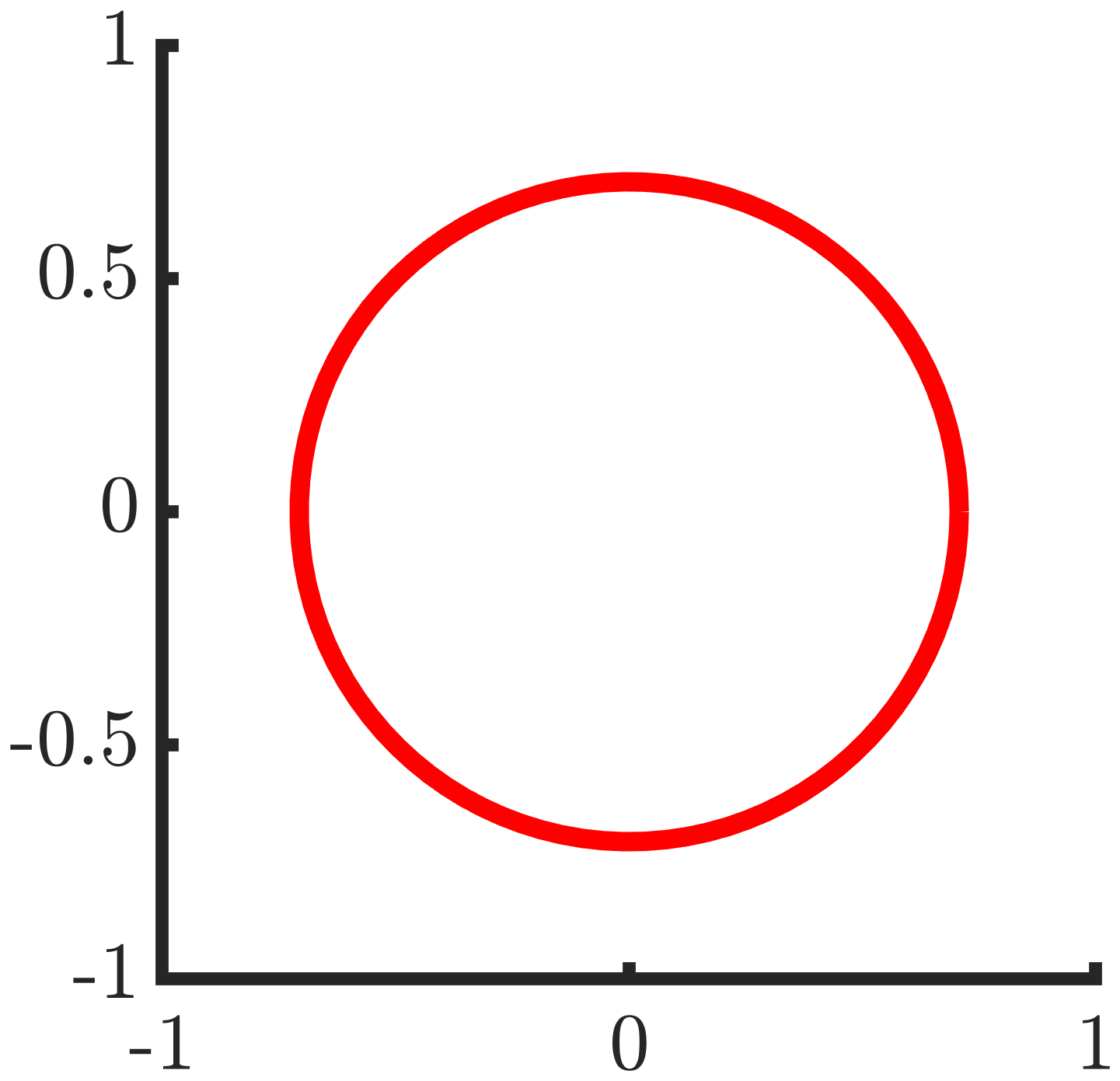}
        \subcaption{$(0,(2))$}
        \label{fig:all-mps-n-leq-3:0sc2}
    \end{subfigure}
    \begin{subfigure}{.32\linewidth} %
        \centering
        \includegraphics[width=1.\linewidth]{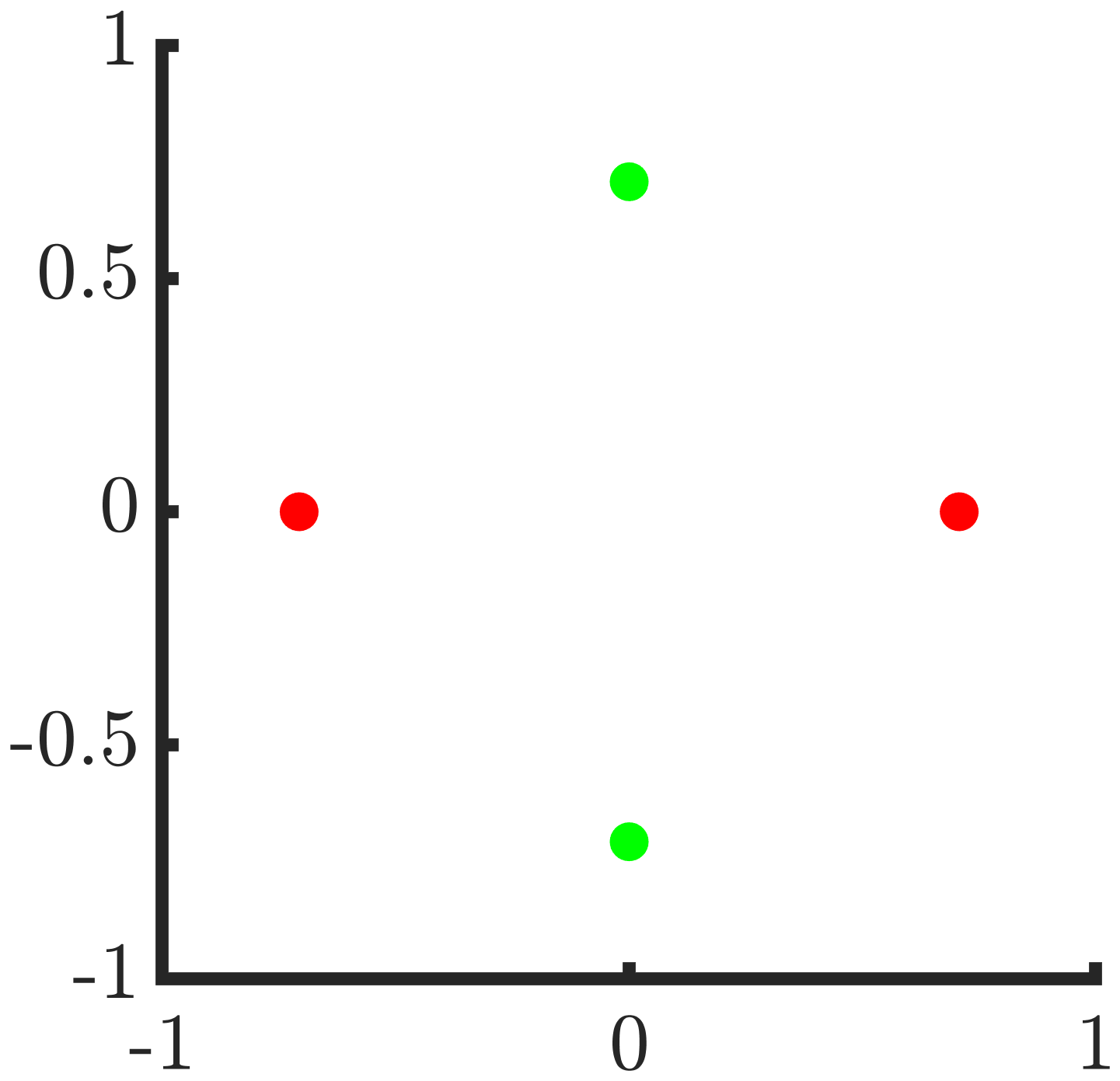}
        \subcaption{$(0,(1,1))$}
        \label{fig:all-mps-n-leq-3:0sc11}
    \end{subfigure}
    \caption{Plots of the four elements of $\cY \in \mps$ embeddable in $\R^2$.}
    \label{fig:all-mps-n-eq-2}
\end{figure}

\begin{figure}
    \centering
    \begin{subfigure}{.32\linewidth} %
        \centering
        \includegraphics[width=1.\linewidth]{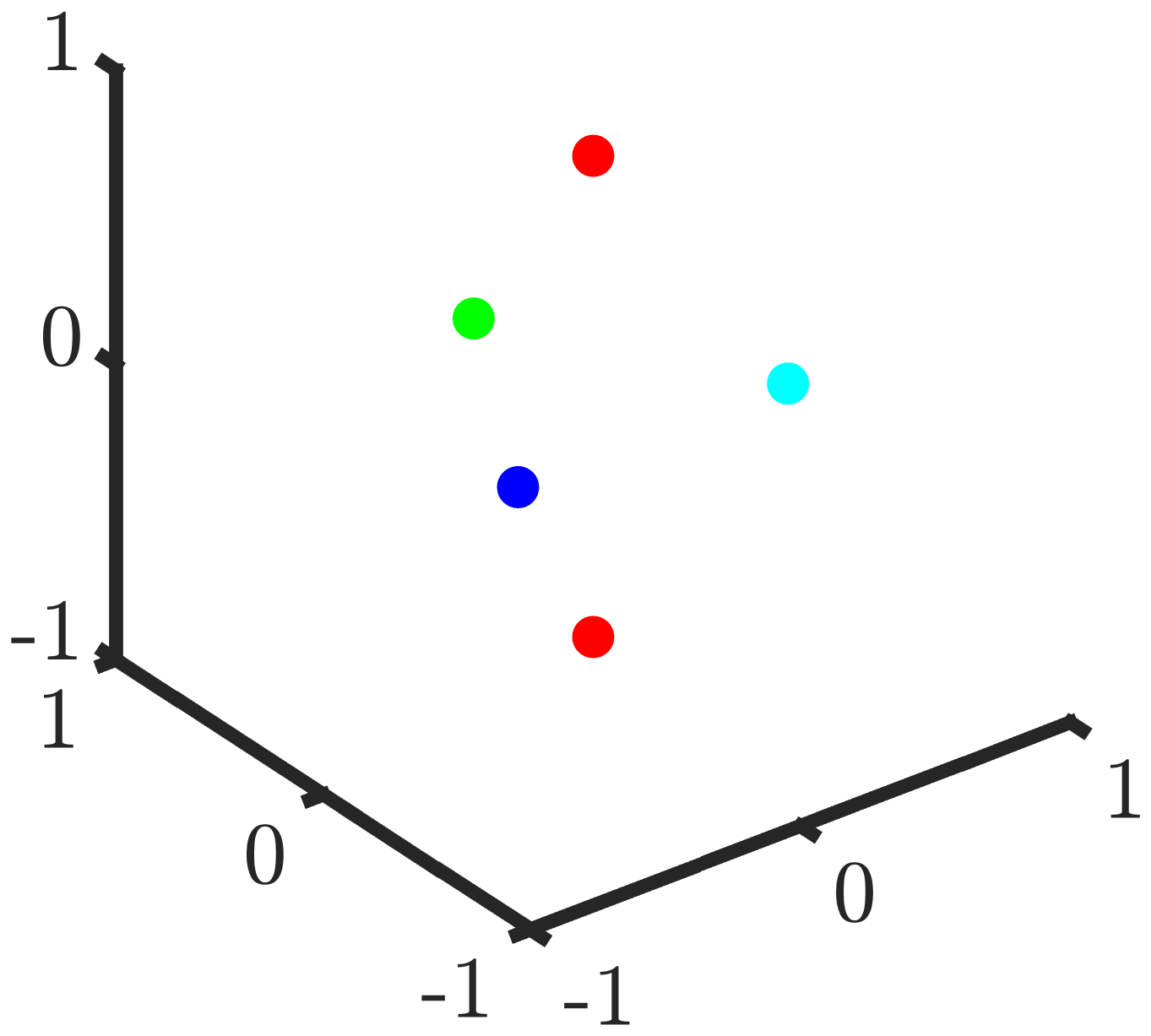}
        \subcaption{$(3,(1))$}
        \label{fig:all-mps-n-leq-3:3sc1}
    \end{subfigure}
    \begin{subfigure}{.32\linewidth} %
        \centering
        \includegraphics[width=1.\linewidth]{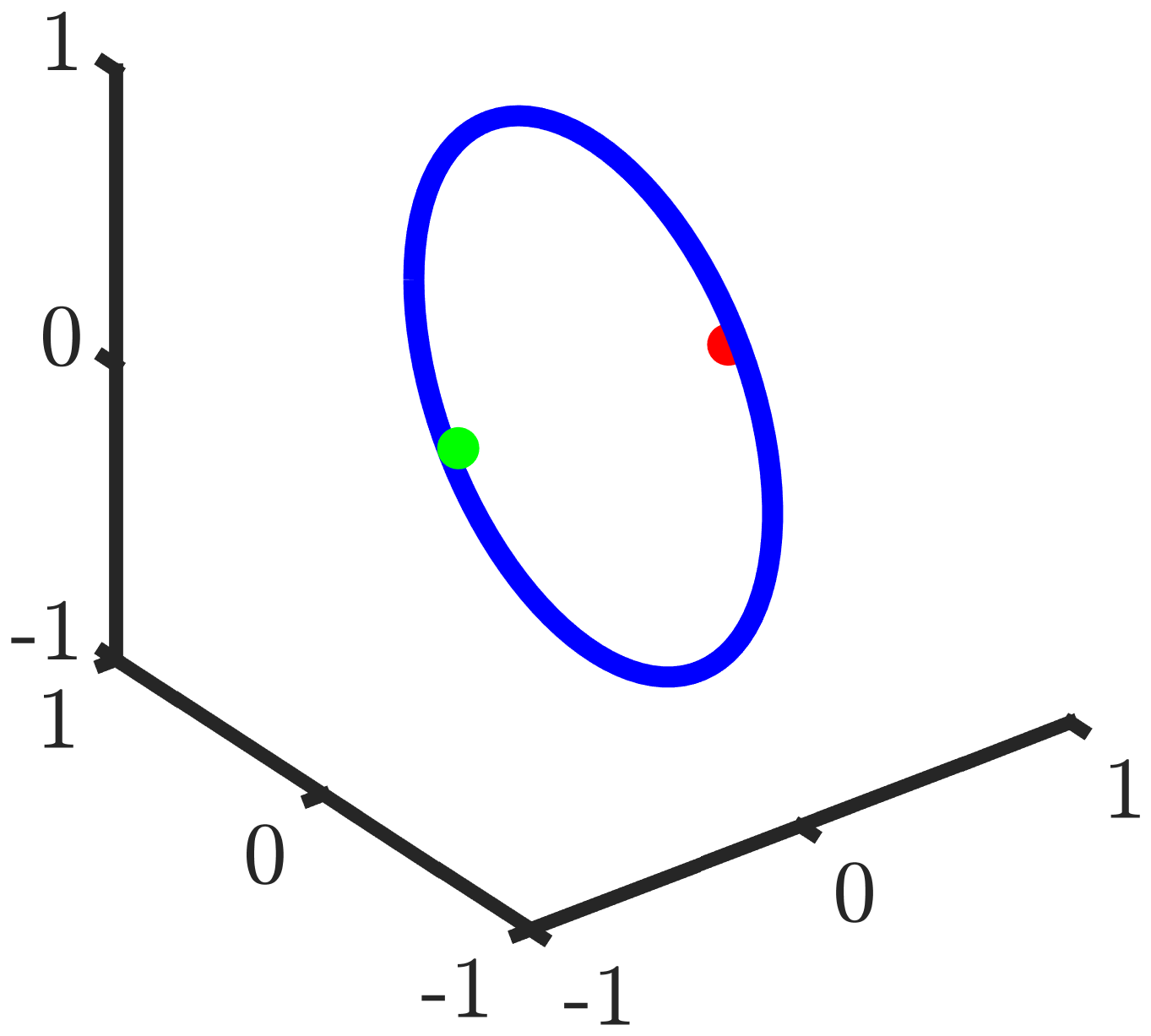}
        \subcaption{$(2,(2))$}
        \label{fig:all-mps-n-leq-3:2sc2}
    \end{subfigure}
    \begin{subfigure}{.32\linewidth} %
        \centering
        \includegraphics[width=1.\linewidth]{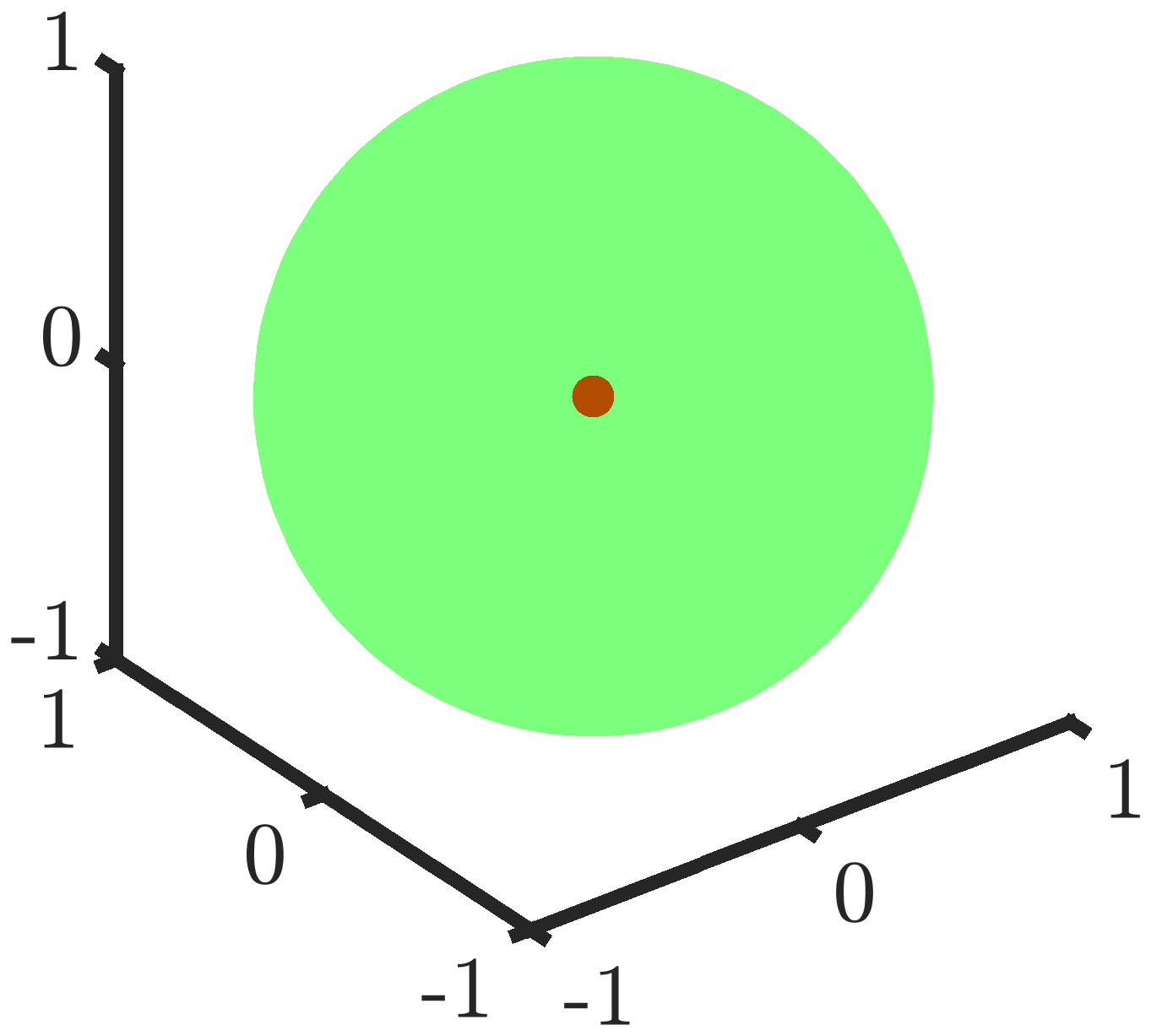}
        \subcaption{$(1,(3))$}
        \label{fig:all-mps-n-leq-3:1sc3}
    \end{subfigure}
    \begin{subfigure}{.32\linewidth} %
        \centering
        \includegraphics[width=1.\linewidth]{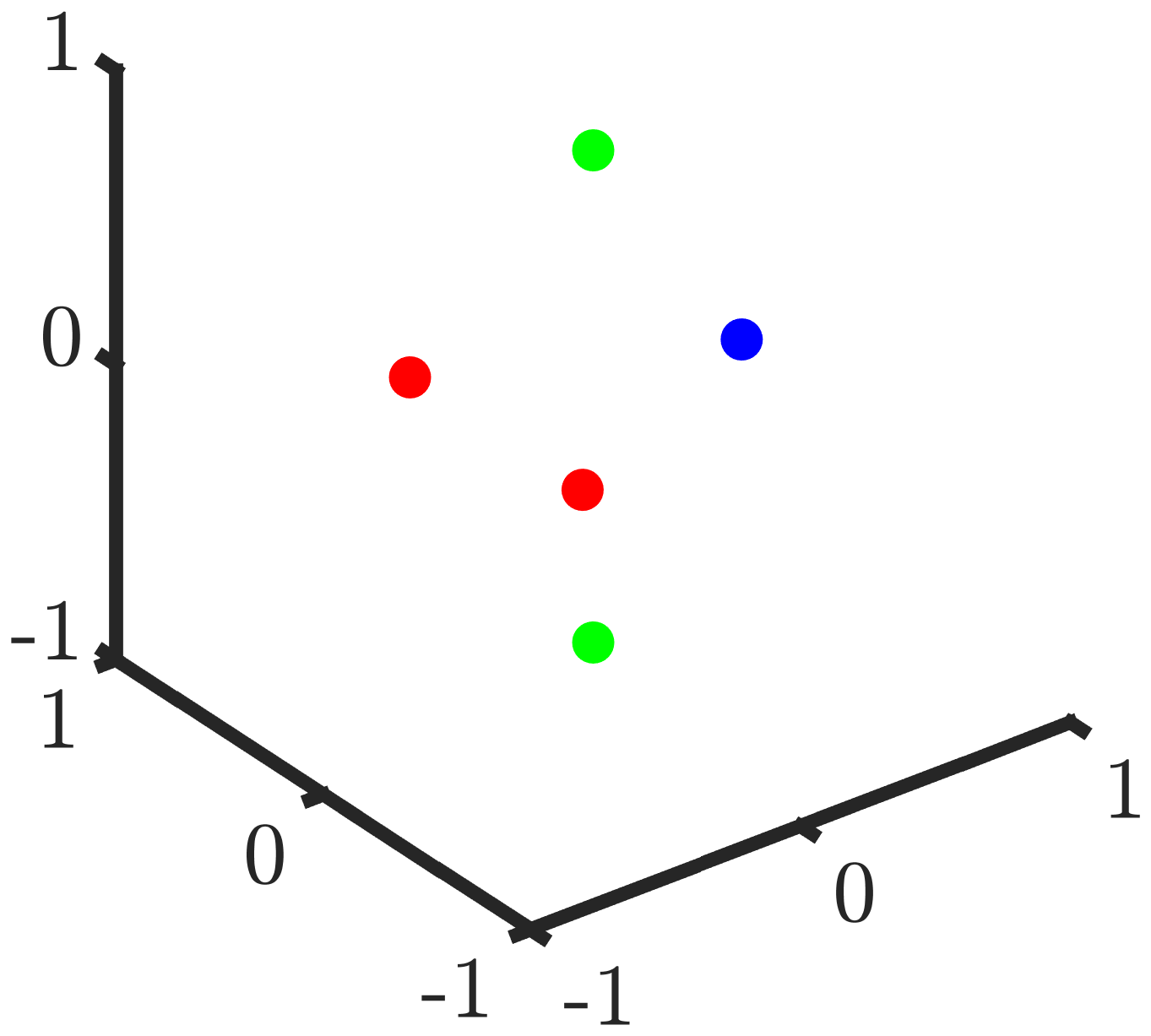}
        \subcaption{$(1,(1,1))$}
        \label{fig:all-mps-n-leq-3:1sc11}
    \end{subfigure}
    \begin{subfigure}{.32\linewidth} %
        \centering
        \includegraphics[width=1.\linewidth]{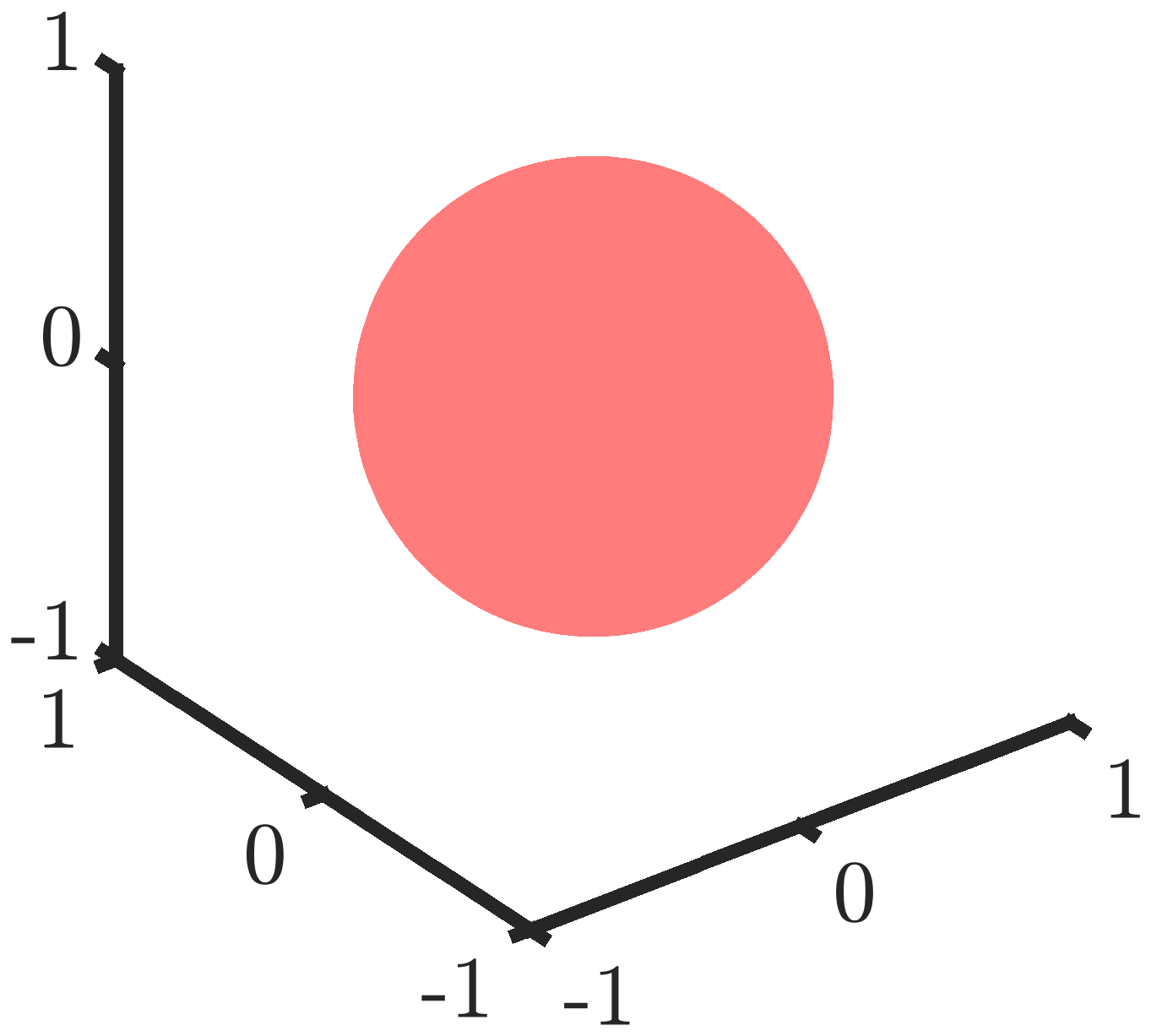}
        \subcaption{$(0,(3))$}
        \label{fig:all-mps-n-leq-3:0sc3}
    \end{subfigure}
    \begin{subfigure}{.32\linewidth} %
        \centering
        \includegraphics[width=1.\linewidth]{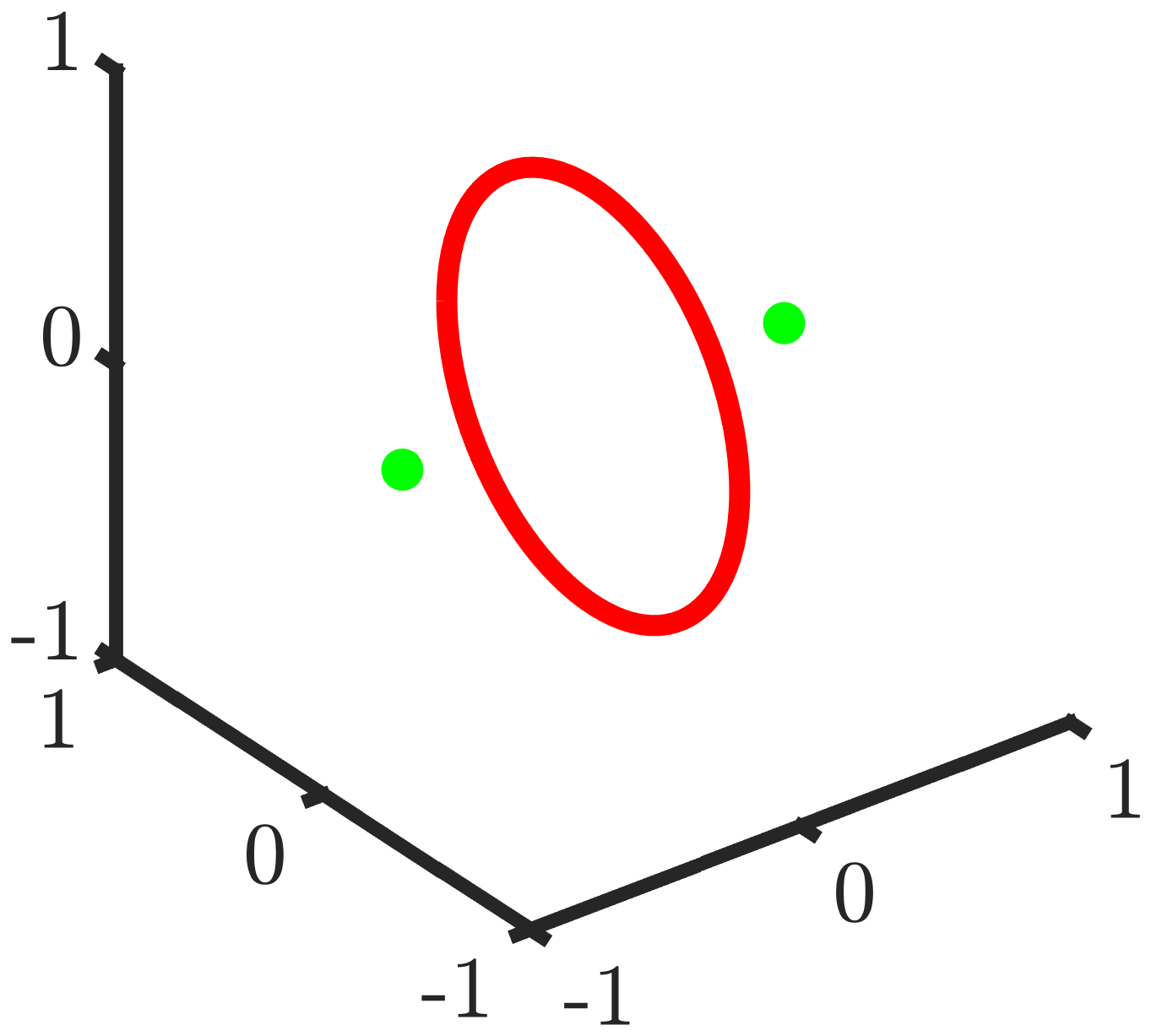}
        \subcaption{$(0,(2,1))$}
        \label{fig:all-mps-n-leq-3:0sc21}
    \end{subfigure}
    \begin{subfigure}{.32\linewidth} %
        \centering
        \includegraphics[width=1.\linewidth]{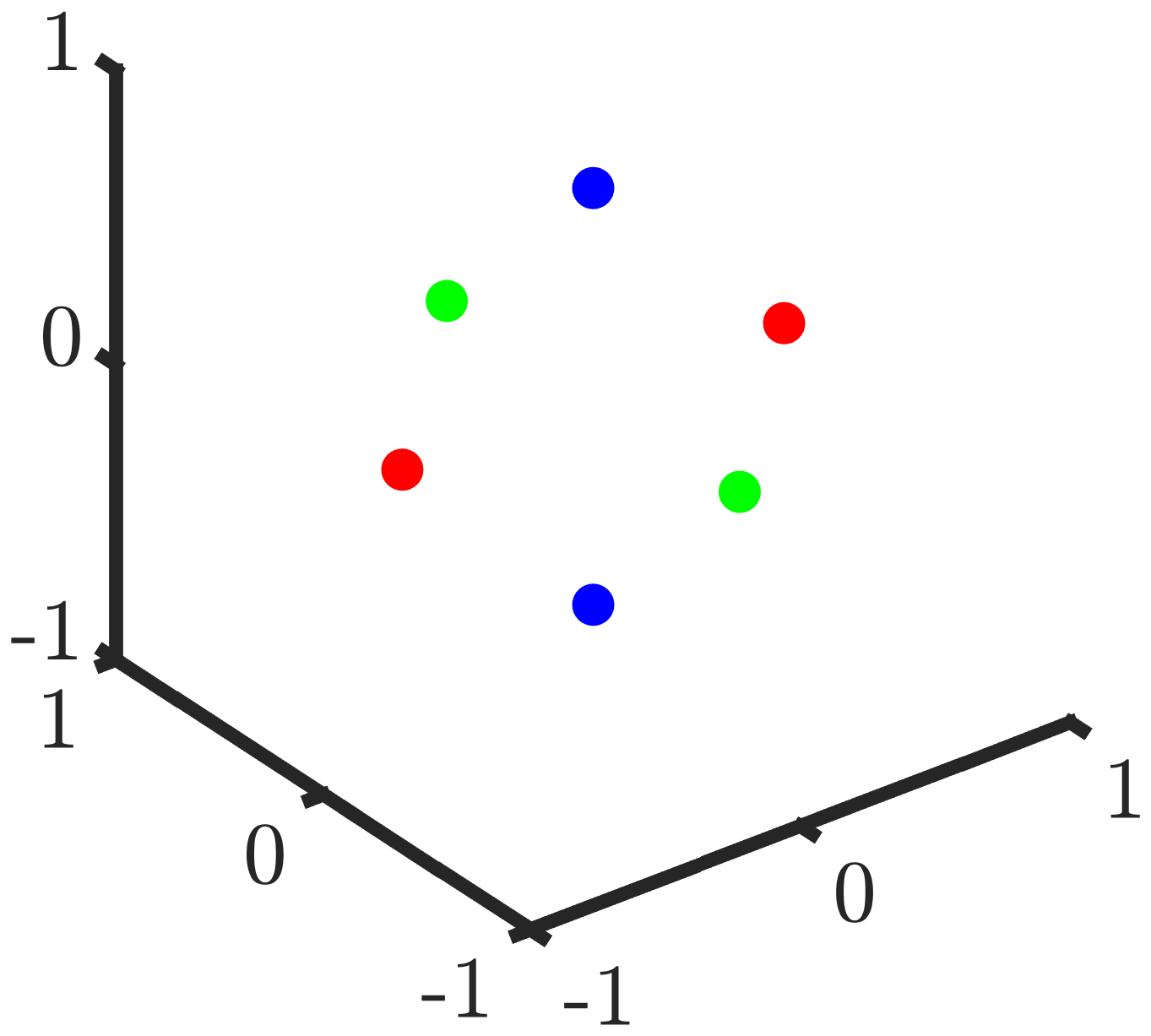}
        \subcaption{$(0,(1,1,1))$}
        \label{fig:all-mps-n-leq-3:0sc111}
    \end{subfigure}
    \caption{Plots of the seven elements of $\cY \in \mps$ embeddable in $\R^3$.}
    \label{fig:all-mps-n-eq-3}
\end{figure}

\begin{algorithm}
\caption{Generation of all $\hat \cS_=$ with sum $N$} \label{alg:generation-of-all-hat-sc-eq}
\begin{algorithmic}
\Require $N$, target sum \textbf{Result}: Set of all $\hat \cS_=$ with sum $N$
\State $\cS_= \gets \{\}$
\If{$N = 1$} \Comment{The case when $m = 1$}
\State $\cS_= \gets \{(1,())\}$
\Else%
\State $\cS_= \gets \{(1,(N-1))\}$
\EndIf
\For{$p \in \mathrm{partition}(N)$} \Comment{$\mathrm{partition}$ returns $p$, a tuple of non-zero integers that sum to $N$.}
\State $\cS_= \gets \cS_= \cup \set{(0,p)}$
\EndFor
\State \Return{$\cS_=$}
\end{algorithmic}
\end{algorithm}

\begin{algorithm}
\caption{Generation of $\hat \cS_{\neq}$ with sum $N$} \label{alg:generation-of-all-hat-sc-neq}
\begin{algorithmic}
\Require $N$, target sum \textbf{Result}: Set of all $\hat \cS_{\neq}$ with sum $N$
\State $\cS_{\neq} \gets \{\}$
\For{$m \in \set{0,N-1}$}
\For{$d_1 \in \set{1,N-m}$}
\For{$p \in \mathrm{partition}(N - m - d_1)$} \Comment{$\mathrm{partition}$ returns $p$, a tuple of non-zero integers that sum to $N-m$.}
\If{$m + \#(p) \geq 2$} \Comment{$\#(p)$ denotes the length of $p$}
\State $\cS_{\neq} \gets \cS_{\neq} \cup \set{s}$
\EndIf
\EndFor
\EndFor
\EndFor
\State \Return{$\cS_{\neq}$}
\end{algorithmic}
\end{algorithm}

We conclude the section on signatures by counting the number of elements in $\bigcup_{I = 1,\dots,\infty}\frac{\mps(I,n)}{\cong}$. In the following proposition, $p \colon \bbN \to \bbN$ denotes the partition function from number theory. That is, $p(n)$ is the number of distinct ways to write $n$ as the sum of natural numbers.

\begin{proposition}[Counting Maximal Equidistant Spacings]
    \label{prop:counting-maximal-equidistant-spacings}
    Let $n \in \bbN$ be given. Then the following holds.
    \begin{align}
        \label{eqn:prop:counting-maximal-equidistant-spacings:1}
        \#\paren{\bigcup_{I = 1,\dots,\infty}\frac{\mps_=(I,n)}{\cong}} &= 1 + p(n)\\
        \label{eqn:prop:counting-maximal-equidistant-spacings:2}
        \#\paren{\bigcup_{I = 1,\dots,\infty}\frac{\mps_{\neq}(I,n)}{\cong}} &= \sum_{i = 1}^{n-1}p(i)\\
        \label{eqn:prop:counting-maximal-equidistant-spacings:3}
        \#\paren{\bigcup_{I = 1,\dots,\infty}\frac{\mps(I,n)}{\cong}} &= \sum_{i = 0}^{n}p(i)
    \end{align}
\end{proposition}

\begin{proof}
    First we prove Eqn. \ref{eqn:prop:counting-maximal-equidistant-spacings:1}. By Theorem \ref{thm:sig-is-a-bijection-on-maximal-equidistant-spacings}, $\#\paren{\bigcup_{I = 1,\dots,\infty}\frac{\mps_=(I,n)}{\cong}} = \#\paren{\bigcup_{I = 1,\dots,\infty} \cS_=(I,n)}$. For a fixed $n$, Eqn \ref{eqn:hat-sc-eq-def} has two cases: when $m = 1$ and $I = 2$ and when $m = 0$ when $I \geq 2$. 
    
    If $m = 1$, then $k = 1$ and $d_1 = n$ (unless $n = 0$, in which case $k = 0$). Hence, there is one signature with value $(1,(n))$ or $(1,())$ when $n = 0$. 

    If $m = 0$, then the signatures are of the form $(0,(d_1,d_2,\dots,d_k))$. From the constraint $\qsum ik d_k = n$, there is one signature per partition of $n$.

    Summing these two cases, we have that $\#\paren{\bigcup_{I = 1,\dots,\infty} \cS_=(I,n)} = 1+p(n)$. This proves Eqn. \ref{eqn:prop:counting-maximal-equidistant-spacings:1}.

    Now we prove Eqn. \ref{eqn:prop:counting-maximal-equidistant-spacings:2}. Again, by Theorem \ref{thm:sig-is-a-bijection-on-maximal-equidistant-spacings}, $\#\paren{\bigcup_{I = 1,\dots,\infty}\frac{\mps_{\neq}(I,n)}{\cong}} = \#\paren{\bigcup_{I = 1,\dots,\infty} \cS_{\neq}(I,n)}$. We use Eqn. \ref{eqn:hat-sc-neq-def}. Rearranging terms and using $m+k = I$, we have that 
    \begin{align}
        \qsum ik d_k = n+2-I.
    \end{align}
    So, for each $I \geq 3$, there is one signature of the form $(I-k,(d_1,\dots,d_k))$ for each partition of $n+2-I$. Thus,
    \begin{align}
        \#\paren{\bigcup_{I = 1,\dots,\infty} \cS_{\neq}(I,n)} = \sum_{I = 3}^{n+1} p(n+2 - I) = \qsum i{n-1} p(i)
    \end{align}
    where the final equality comes from substituting $i \coloneqq n+2-I$. This proves Eqn. \ref{eqn:prop:counting-maximal-equidistant-spacings:2}.

    Eqn. \ref{eqn:prop:counting-maximal-equidistant-spacings:3} follows from Remark \ref{rmk:the-nature-of-equidistant-spacings-whose-centers-coincide} and Theorem \ref{thm:sig-is-a-bijection-on-maximal-equidistant-spacings}, so
    \begin{align}
        \#\paren{\bigcup_{I = 1,\dots,\infty}\frac{\mps(I,n)}{\cong}} = \#\paren{\bigcup_{I = 1,\dots,\infty}\frac{\mps_{\neq}(I,n)}{\cong}} + \#\paren{\bigcup_{I = 1,\dots,\infty}\frac{\mps_{=}(I,n)}{\cong}} = 1 + \qsum in p(i)
    \end{align}
\end{proof}

\subsection{Algorithm for Testing Equidistant Spacings}
\label{sec:algorithm-for-testing-equidistant-spacings}
In this section, we describe an algorithm for testing if a finite collection of points lie in a equidistant spacing, motivated by Corollary \ref{cor:characterization-of-equidistant-spacings}. It is a surprising fact that the runtime of the resulting algorithm is linear in the number of points to be tested, much faster than the quadratic runtime of the naive algorithm (i.e., directly compute $\norm{y_i - y_j}$ for each pair of differently colored points). The algorithm follows the following steps
\begin{enumerate}
    \item Use Gram-Schmidt orthonormalization to  compute $A_i$.
    \item Compute $c_i$ via orthogonal projection onto the $A_i$'s. 
    \item Verify that $\norm{c_i - y_i}$ is constant for each $y_i\in Y_i$ and $i$. 
    \item Check that the centers are correctly spaced. 
\end{enumerate}

The algorithm is given in \ref{alg:verification-of-equidistant-spacing}. The number $d_i \in \bbN$ is the dimension of the space $A_i$, and is defined implicitly from the output of the GS orthonormalization. 

\begin{algorithm}
\caption{Verification of Equidistant Spacing} \label{alg:verification-of-equidistant-spacing}
\begin{algorithmic}
\Require $I$, and family $\set{Y_i}_{i = 1}^I$, where $Y_i = \set{y^{(1)}_i,\dots,y^{(N_i)}_i}$
\For{$i \in I$} \Comment{Form $A_i$}
    \State $\set{e_i^{(1)},\dots,e_i^{(d_i)}} = \mathrm{GS}\paren{y^{(1)}_i - y^{(1)}_i ,\dots,y^{(N_i)}_i - y^{(1)}_i}$ \Comment{Gram-Schmidt Orthonomaliation}
\EndFor
\For{$i,j \in I\times I$} \Comment{Verify Orthogonality of $A_i$ and $V_i$}
    \For{$k_i, k_j \in \set{1,\dots,d_i} \times \set{1,\dots,d_j}$}
        \If{$\innerprod{e^{(k_i)}_i}{e^{(k_j)}_j}\neq 0$}
            \Return{{False}}
        \EndIf
    \EndFor
\EndFor
\For{$i \in {1,\dots,I}$} \Comment{Check that $\norm{y_i - c_i}$ is constant}
    \State $c_i \gets \proj_{A_i}y^{(1)}_{i + 1}$
    \State $r_i \gets \norm{c_i - y^{(1)}}$
    \For{$k \in \set{2,\dots,N_i}$}
        \If{$\norm{y^{(k)} - c_i} \neq r_i$}
            \State \Return{False}
        \EndIf
    \EndFor
\EndFor
\State $\set{\tilde e^{(1)}, \dots \tilde e^{(\tilde d)}} = GS\paren{c_1 - c_2,\dots,c_1 - c_I}$
\For{$i = 1,\dots,I$} \Comment{Verify $V_i$ is orthonormal to $\tilde V$.}
    \For{$k_i \in \set{1,\dots,d_i}, \tilde k \in \set{1,\dots,\tilde d}$}
        \If{$\innerprod{e^{(k_i)}_i}{\tilde e^{(\tilde k)}}\neq 0$}
            \State \Return{False}
        \EndIf
    \EndFor
\EndFor
\For{$i,j \in \set{1,\dots,I}^2$, $i \neq j$} \Comment{Check spacing of centers}
    \If{$\norm{c_i - c_j}^2 + r_i^2 + r_j^2 \neq 1$}
        \State \Return{False}
    \EndIf
\EndFor
\State \Return{True}
\end{algorithmic}
\end{algorithm}

\begin{theorem}[Correctness of Alg. \ref{alg:verification-of-equidistant-spacing}]
    \label{thm:correctness-alg-1}
    Let $\cY \coloneqq \set{Y_i}_{i = 1}^I$ and each $Y_i \subset \R^n$ be finite. Then if Algorithm \ref{alg:verification-of-equidistant-spacing} is run on $\cY$, then
    \begin{enumerate}
        \item Alg. 1 will run and output True or False in at most $O(n^2\#(\cY) + n^2I^2)$ operations.
        \item It will return True if $\cY \in \ps$ and False otherwise.
    \end{enumerate}
\end{theorem}

\begin{proof}
    First, we compute the runtime of each part of Alg. \ref{alg:verification-of-equidistant-spacing}. All times reported below \emph{per class} $i \in \set{1,\dots,I}$.
    \begin{enumerate}
        \item Formation of $A_i$. Gram-Schmidt orthonormalization takes $O\paren{N_i n d_i}$ operations.%
        \item Verification of orthogonality of $A_i$ takes $O(d_id_jI)$.
        \item Verifying that $\norm{y_i - c_i}$ is constant takes $O(N_in)$ operations.
        \item Verification of orthogonality of $V_i$ and $\tilde V$ takes $O(d_i\tilde d)$ where $\tilde d = \dim{\tilde V}$.
        \item The check spacing step takes $O(In)$ operations.
    \end{enumerate}
    Summing the above estimates yields
    \begin{align*}
        \qsum iI N_i n d_i + d_id_jI + + d_i\tilde d nN_i + nI 
        &\leq n^2\#(\cY) + n^2I^2 + n^2 + nI\#(\cY) + nI^2\\
        &\leq n^2\#(\cY) + n^2I^2 + n^2 + 2nI\#(\cY)\\
        &\leq 2\paren{n+1}^2\#(\cY) + n^2(I^2 + 1).
    \end{align*}
    This proves Point 1.

    Now we prove point 2. 
    It is apparent that if the algorithm is run and returns true, then this implies that the following holds.
    \begin{enumerate}
        \item The $V_i$'s are pairwise orthogonal.
        \item $r_i > 0$ if $A_i > 0$ (implicit in the construction of $V_i$).
        \item Each $A_i$ is orthogonal to $\tilde A$.
        \item $Y_i \subset S^{n-1}_{r_i}(c_i) \cap A_i$.
        \item Equation \ref{eqn:center-condition} holds.
    \end{enumerate}
    Hence, by Theorem \ref{thm:radius-recipe}, Algorithm \ref{alg:verification-of-equidistant-spacing} returns True if and only if $\cY \in \ps$.
\end{proof}

\section{Acknowledgments}
The author would like to offer their deep thanks to Matthew Biesecker and Greg Conner for their many comments on the manuscript. Especially for Matt's careful eye for edits and Greg's many narrative suggestions. The author received support from National Science Foundation Award \#2418909 LEAPS-MPS: Topological Machine Learning and Cryo-EM and also from CAPITAL Services Inc.

\bibliographystyle{plain}
\bibliography{imports/references.bib}

\appendix

\section{Notation}
\label{sec:notation}
In this section we summarize the notation used in this work. All notation is defined where it is first used. This section is a reference.

Strokes
\begin{enumerate}
    \item The set $\cY$ always denotes $\sqcup_{i = 1}^I Y_i$ for some $I$ where $Y_i \subset \R^n$.%
    \item The letters $i,j$ and $k$ denote a class label. It is always assumed that $i,j,k \in \set{1,\dots,I}$ for whatever $I$ is being considered.
    \item A \emph{class} refers to a $Y_i \in \cY$ for an $i = 1,\dots I,$. The index $i$ is called the \emph{class label} of the class $Y_i$.
    \item The set $A_i$ denotes $\aff{Y_i}$, see Def. \ref{def:smallest-affine-set}.
    \item The set $\ps(I,\R^n)$ denotes the set of equidistantly spaced points with $I$ classes in $\R^n$. See Def. \ref{def:equidistantly-spaced}. It is denoted by  $\ps$ when the $I$ and $\R^n$ are understood or immaterial.
    \item The letter $I$ denotes the number of classes.
    \item In figures, colors are used to denote classes. Points of the same color belong to the same class.
    \item The group $E(n)$ denotes the Euclidean group on $\R^n$. 
    \item The letter $V$ denotes a linear subspace. $V_i$ denotes the vector space of differences of $A_i$, see Def. \ref{def:smallest-affine-set}.
    \item The letters $c_i \in \R^n$ and $r_i \in \R$ denote the center and radius of $Y_i$, see Def. \ref{def:center-and-radius}.
    \item The set $S^m_r(c)$ denotes an $m$ sphere of radius $r$ with center $c$ embedded in $\R^{m+1}$.
    \item The notation $E_{\ell,n}$ is an embedded equilateral simplex with side length $\ell$ and $n$ vertices.
    \item The $\cS$ refers to signatures, see Def. \ref{def:signagure-of-equidistant-spacing}.
    \item $\bszero^n$ denotes the zero vector in $\R^n$. When $n$ is clear or unimportant it is suppressed.
\end{enumerate}
Operators
\begin{enumerate}
    \item The norm $\norm{\cdot}$ denotes the Euclidean norm over $\R^n$ for some $n$ and $\innerprod\cdot\cdot$ the Euclidean inner product over $\R^n$.
    \item The notation $\abs \cY \coloneqq \cup_{i = 1}^I Y_i$. That is, $\abs \cY$ refers to $\cY$ considered as a subset in $\R^n$.
    \item The operator $\aff{X}$ denotes the smallest affine subspace that contains $X \subset \R^n$.
    \item The $\proj_{X}$ denotes the orthogonal projection onto $X$ where $X$ is a affine subspace of $\R^n$.
    \item The notation $\bigtriangleup_{xyz\dots}$ refers to an embedded triangle or simplex with vertices $x,y,z,\dots$.
    \item The notation $\angle xyz$ refers to the angle formed by the lines $\overline{xy}$ and $\overline{yz}$.
    \item For a set $S$, $\#(S)$ denotes the cardinality of $S$.
\end{enumerate}
Diacritics, Modifiers
\begin{enumerate}
    \item For class-associated objects like centers, radii, etc. subscripts indicate which class they are associated with through a subscript. For example, $r_i$ is the radius associated with class $Y_i$, $A_i$ is the smallest affine space associated with class $Y_i$, etc.
    \item A prime attached to a variable indicates different elements when a choice w.l.o.g. was made. For example, if $\cY \in \ps(I,\R^n)$ is given, $\cY'$ denotes a different element of $\ps(I,\R^n)$. If $y_i \in Y_i$, then $y'_i$ denotes a different element of $Y_i$.
    \item Sometimes it is necessary to consider many different elements when a choice w.l.o.g. was made. In this case we avoid a preponderance of primes and instead use superscripts in parenthesis. For example, rather than using $\cY, \cY',\cY'',\dots$ we use $\cY^{(1)}, \cY^{(2)},\cY^{(3)},\dots$. 
    \item The diacritic $\widehat\cdot$ denotes a maximality under inclusion. Examples include how $\mps$ denotes the set of maximal equidistant spacings, $\widehat\cS$ denotes the maximal signatures, etc.
    \item The superscript $\cdot^*$ denotes the orthogonal projection of $\cdot$ onto some affine subspace. This is used in results that invoke Lemma \ref{lem:ortho-reflection-lemma}.
\end{enumerate}

\section{Background}

\begin{definition}[Affine Sets]
    \label{def:addine-set}
    Let $U$ be a vector space and $A \subset U$. We say that $A$ \emph{affine} if 
    \begin{align}
        V \coloneqq \set{u - v \colon u,v \in U}
    \end{align}
    is a linear subspace. We say that $V$ is the \emph{subspace of differences} of $A$.

    We say that two affine spaces are orthogonal if the subspace of their differences are orthogonal.

    Given a set of points $\set{x_i}_{i = 1}^n$, we define
    \begin{align}
        \aff(\set{x_i}_{i = 1}^n)
    \end{align}
    to be the smallest affine set that contains the points $\set{x_i}_{i = 1}$.
\end{definition}

We may characterize affine sets in a few different ways, all of which are useful.

\begin{proposition}[Characterization of $A$]
    \label{prop:characterization-of-a_i}
    Let $U$ be a subset of a vector space. The following are all equivalent.
    \begin{enumerate}
        \item $A$ is the minimal affine set that contains $U$.
        \item $A$ is the intersection of all affine sets that contain $U$.
        \item Let $m$ be the maximal number such that for the pairs $((y_i,y_i'))_{i = 1}^m$, the vectors $y_1 - y'_1,\dots,y_m - y'_m$ are linearly independent. Then for an arbitrary $y \in U$, $A$ is given by
        \begin{align}
            A = y + \Span{y_1 - y'_1,\dots,y_m - y'_m}.
        \end{align}
    \end{enumerate}
\end{proposition}

\begin{figure}
    \centering
    \begin{subfigure}{.32\linewidth}
        \centering
        \includegraphics[width=1.\linewidth]{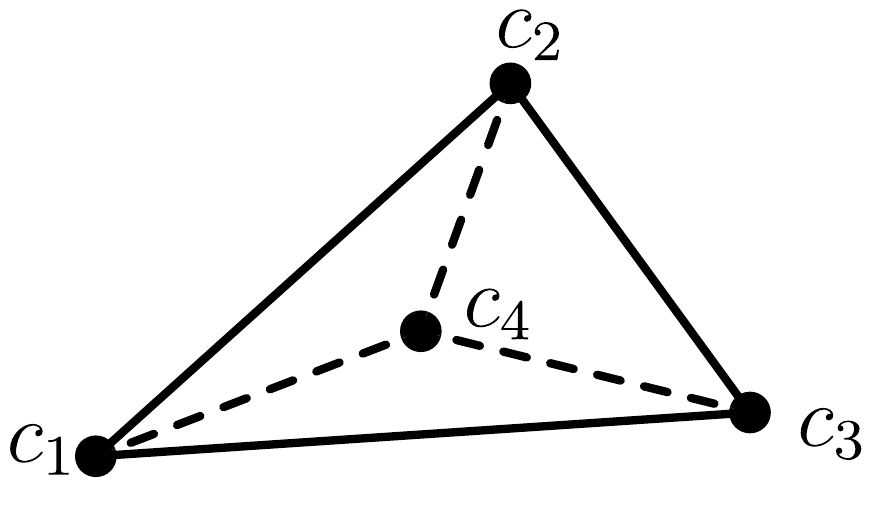}
        \subcaption{Simplex $\bigtriangleup$}
    \end{subfigure}
    \begin{subfigure}{.32\linewidth}
        \centering
        \includegraphics[width=1.\linewidth]{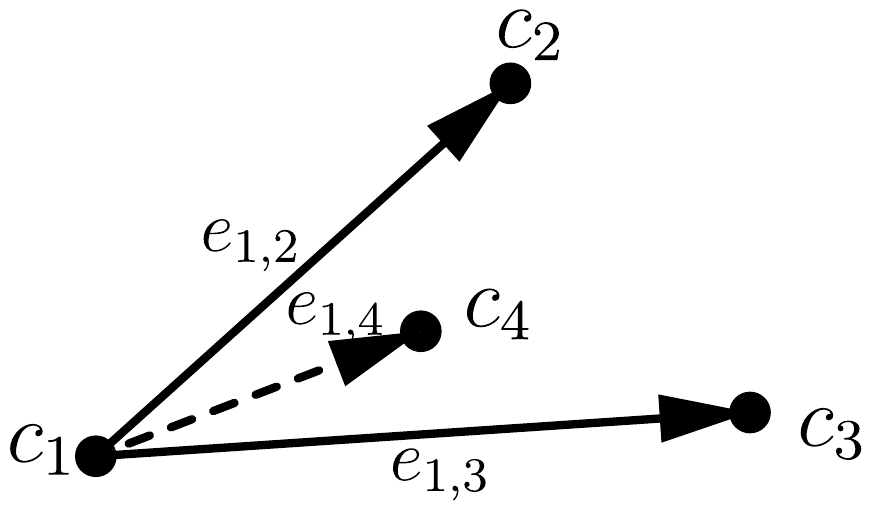}
        \subcaption{Edges $\set{e_{1,1},\dots,e_{1,4}}$}
    \end{subfigure}
    \begin{subfigure}{.32\linewidth}
        \centering
        \includegraphics[width=1.\linewidth]{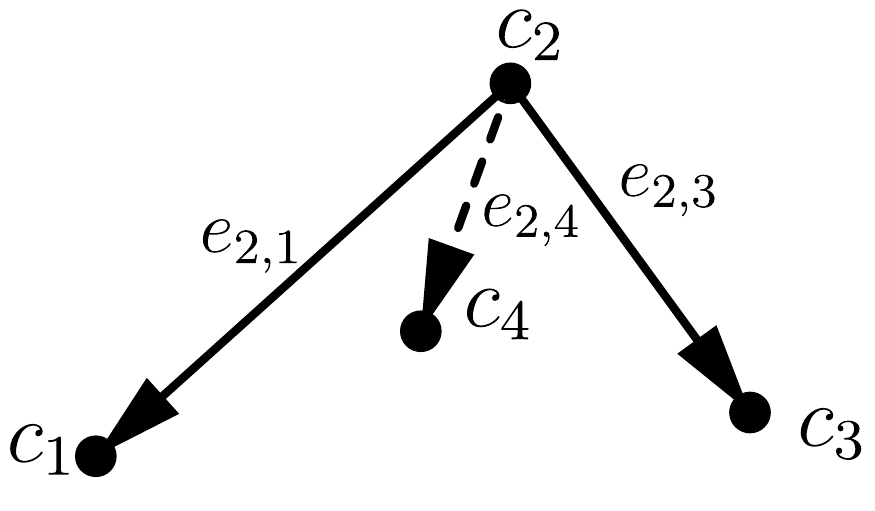}
        \subcaption{Edges $\set{e_{2,1},\dots,e_{2,4}}$}
    \end{subfigure}
    \caption{The idea behind Proposition \ref{prop:characterization-of-a_i}.}
    \label{fig:idea-behind-affine-space}
\end{figure}
\begin{proof}
    The proof is given in Figure \ref{fig:idea-behind-affine-space}. Clearly, each edge of $\bigtriangleup$ may be expressed as the vector difference between vectors that lie along the edges originating from one vertex of $\bigtriangleup$.
\end{proof}

\begin{definition}[Convex Hull]
    \label{def:convex-hull}
    Given a set $X \subset \R^n$, the \emph{convex hull} of $X$ is the smallest convex set that contains $X$.

    Equivalently, it is the intersection of all convex sets that contain $X$.
\end{definition}

\begin{definition}[Reflection Operator]
    \label{def:reflection-operator}
    Given a point $x \in \R^n$, and subset $X \subset \R^n$, the operator $\reflect_x X$ is the reflection of all point in $X$ about $x$.
\end{definition}

Next we must prove an abstract property about the geometry of simplices. This proof doesn't obviously related to equidistant spacings, but is necessary for later proofs.

\begin{proposition}[Affine Space Containing Points]
    \label{prop:affine-space-containing-points}
    Let $\bigtriangleup$ be $I$ points in $\R^n$. For any $i,j \in \set{1,\dots,I}$ define $e_{i,j} \coloneqq x_i - x_j$. Then the following holds.
    \begin{enumerate}
        \item For any $i,j$, 
        \begin{align}
            \label{eqn:adjacent-edges-span-the-same-space}
            \Span{e_{i,1}, \dots, e_{i,I}} = \Span{e_{j,1}, \dots, e_{j,I}}.
        \end{align}
        \item Define $\tilde V \coloneqq \Span{e_{i,1}, \dots, e_{i,I}}$, and let $V$ be a linear subspace of $\R^n$. Then the following statements are equivalent
        \begin{enumerate}
            \item $V$ and $\tilde V$ are orthogonal.
            \item There is an $i \in \set{1,\dots,I}$ so that $V$ is orthogonal to $e_{i,j}$ for each $j \in \set{1,\dots,I}$.
            \item $V$ is orthogonal to each $e_{i,j}$ for each $i,j \in \set{1,\dots,I}$.
        \end{enumerate}
    \end{enumerate}
\end{proposition}
\begin{proof}
    First, we show 1. Define
    \begin{align}
        \tilde V_i \coloneqq \Span{e_{1,1},\dots,e_{1,I}}\quad \tilde V_j \coloneqq \Span{e_{j,1},\dots,e_{j,I}}.
    \end{align}
    We aim to show that $\tilde V_i = \tilde V_j$. Notice that for any $k$, 
    \begin{align}
        e_{i,k} = \underbrace{e_{i,j}}_{\in \tilde V_j} + \underbrace{e_{j,k}}_{\in \tilde V_j} \in \tilde V_j
    \end{align}
    and so $\tilde V_i \subset \tilde V_j$. Interchanging the roles of $i$ and $j$ in the above calculation yields $\tilde V_j \subset \tilde V_i$, hence $\tilde V_i \subset \tilde V_j$ showing 1.

    First, we show that 2.a implies 2.b. Let $V$ and $\tilde V$ be orthogonal, and let $i\in\set{1,\dots,I}$ be given. Then $V$ and $\Span{e_{1,1},\dots,e_{1,I}} = \Span{e_{i,1},\dots,e_{i,I}}$, so $V$ is orthogonal to each $e_{i,j}$.

    2.b implies 2.c directly from Eqn. \ref{eqn:adjacent-edges-span-the-same-space}.

    2.c implies 2.a, as being orthogonal to each $e_{1,1},\dots,e_{1,I}$ implies orthogonality to $\tilde V$.
\end{proof}

\begin{proposition}[Orthocentrality is Invariant Under Face Projection]
    \label{prop:orthocenter-inv-face-proj}
    The orthogonal projection of the orthocenter of a simplex to a face is the orthocenter of the face. 
    
    Let $\bigtriangleup$ be a simplex in $\R^n$ with vertices $x_1,\dots,x_k$, and $x^*$ be the orthocenter of $\bigtriangleup$. Let $J \subset \set{1,\dots,k}$, $\bigtriangleup_J$ be the simplex with vertices $\set{x_j}_{j \in J}$, and $A_J \coloneqq \aff(\set{x_j}_{j \in J})$. Then $\proj_{A_J}x^*$ is the orthocenter of $\bigtriangleup_J$.
\end{proposition}

\begin{proof}
    From the orthocentrality of $x^*$, we have that
    \begin{align}
        \label{eqn:prop:orthocenter-inv-face-proj:1}
        \overline{x_i x^*} \in \aff(x_1,\dots,x_k)^\perp \subset A_J^\perp.
    \end{align}
    Let $ i\in J$, and let $v \coloneqq x^* - x_i$, then we may write $v = v_\parallel + v_\perp$ where $v_\parallel \in A_J$ and $v_\perp \in A_J$. By linearity and $i \in J$, we have that 
    \begin{align}
        v_\parallel = (x^* - x_i)_\parallel = x^*_\parallel - x_i,\quad v_\perp = (x^* - x_i)_\perp = x^*_\perp.
    \end{align}
    Let $\alpha \in \R$, then from Eqn. \ref{eqn:prop:orthocenter-inv-face-proj:1},
    \begin{align}
        \alpha v = \alpha (x^*_\parallel - x_i) + \alpha x^*_\perp \in A_J^\perp.
    \end{align}
    Hence, by subspace-ness of $A_J^\perp$, $\alpha (x^*_\parallel - x_i) \in A_J^\perp$. But 
    \begin{align}
        \overline{x_i\proj_{A_J}x^*} = \set{\alpha (x^*_\parallel - x_i) \colon \alpha \in \R},
    \end{align}
    and so $\overline{x_i\proj_{A_J}x^*} \in A_J^\perp$.
\end{proof}

\begin{proposition}[Circumcenter of Equilateral Simplex]
    \label{prop:circumcenter-of-equi-simplex}
    Let $E_{\ell,k}$ be an embedded equilateral simplex with $k$ points, and side-length $\ell$. Let $c$ denote the circumcenter of $E_{\ell,k}$. Then $\norm{v - c} = \ell\sqrt{\frac{k}{2(k + 1)}}$, where $v$ denotes a vertex of $E_{\ell,k}$.
\end{proposition}
\begin{proof}
    We may embed $E_{\ell,k}$ into $\R^k$ by embedding its vertices $v_1,\dots,v_k$ via
    \begin{align}
        v_1 \coloneqq \paren{\frac{\sqrt{2}}2 \ell,0,\dots,0},\dots,v_k \coloneqq \paren{0,\dots,0,\frac{\sqrt{2}}2 \ell}.
    \end{align}
    By symmetry, $c \coloneqq \paren{\alpha,\dots,\alpha}$ for some $\alpha \in \R$. In particular, it is the $\alpha$ that minimizes the expression $\norm{v_i - c}$. Computing, we have
    \begin{align}
        \norm{v_i - c}^2 = \paren{\frac{\sqrt 2}2 \ell - \alpha}^2 + k \alpha^2 &= \frac{\ell^2}2 - \sqrt 2 \ell \alpha + \alpha^2 + k\alpha\\
        &=\frac{\ell^2}{2} - \sqrt 2 \ell \alpha + (k + 1)\alpha^2
    \end{align}
    This quadratic expression is minimized when
    \begin{align}
        \alpha = \frac{\sqrt 2 \ell}{2 \paren{k+1}}.
    \end{align}
    Hence we have that 
    \begin{align}
        \norm{v_i - c}^2 &= \frac{\ell^2}2 - \frac{2\ell^2}{2(k + 1)} + \frac{2(k+1)\ell^2}{4(k+1)^2}\\
        &=\frac{\ell^2}2 - \frac{\ell^2}{k + 1} + \frac{\ell^2}{2(k+1)}\\
        &=\frac{\ell^2}2 + \frac{\ell^2}{2(k+1)}\\
        &= \ell^2 \paren{\frac12 + \frac{1}{k+1}} = \ell^2 \frac k {2(k+1)}
    \end{align}
\end{proof}

\begin{lemma}[Existence of Rotation that Aligns Subspaces]
    \label{lem:existence-of-rotation-that-aligns-subspaces}
    Let $(V_1,\dots,V_k)$ and $(V'_1,\dots,V'_k)$ be $k$-tuples of vector subspaces of $\R^n$ so that $V_1 \oplus \dots \oplus V_k$ and $V_1' \oplus \dots \oplus V_k'$.

    If $\dim V_i = \dim V'_i$, then there is a rotation $U \in \cL(\R^n)$ so that 
    \begin{align}
        UV_i = V'_i
    \end{align}
    for each $i = 1,\dots,k$.
\end{lemma}

\begin{proof}
    For each $i = 1,\dots,k$ $n_i = \dim V_i = \dim V'_i$. Define $e^{(i)}_1,\dots,e^{(i)}_{n_i}$ and $f^{(i)}_1,\dots,f^{(i)}_{n_i}$ as orthogonal unit vectors so that
    \begin{align}
        V_i = \Span{e^{(i)}_1,\dots,e^{(i)}_{n_i}}, \quad V_i' = \Span{f^{(i)}_1,\dots,f^{(i)}_{n_i}}.
    \end{align}
    Note that by the condition that $V_1,\dots,V_k$ and $V_1',\dots,V_k'$ direct sum, the sets
    \begin{align}
        \bigcup_{i = 1}^k \set{e^{(i)}_1,\dots,e^{(i)}_{n_i}}, \quad \bigcup_{i = 1}^k \set{f^{(i)}_1,\dots,f^{(i)}_{n_i}}
    \end{align}
    are linearly independent and orthonormal. Completing these sets to an orthonormal basis of $\R^n$ if necessary, we may consider the linear map $U \in \cL(\R^n)$ so that $Ue^{(i)}_j = f^{(i)}_j$ for each $i \in 1,\dots,k'$ and $j = 1,\dots,\dim Y_i$. $U$ sends one orthonormal basis to another, and so is a rotation. Finally,
    \begin{align}
        UV_i &= U\Span{e^{(i)}_1,\dots,e^{(i)}_{n_i}}\\
        &= \Span{Ue^{(i)}_1,\dots,Ue^{(i)}_{n_i}}\\
        &= \Span{f^{(i)}_1,\dots,f^{(i)}_{n_i}}\\
        &= V'_i.
    \end{align}
\end{proof}

\end{document}

%% file: imports/common_packages.tex
\usepackage{amssymb}
\usepackage{amsmath}
\usepackage{textcomp}
\usepackage{graphicx}
\usepackage[mathscr]{euscript}
\usepackage{lineno}
\usepackage{subcaption}
\usepackage[utf8]{inputenc}
\usepackage{float}
\usepackage[normalem]{ulem}
\usepackage{multicol}
\usepackage{lmodern}
\usepackage{lipsum}
\usepackage{marvosym}
\usepackage{mathtools}
\usepackage{soul}
\usepackage{bm}
\usepackage{amsthm}
\usepackage{xargs}                      %
\usepackage{xcolor}
\usepackage{bbm}
\usepackage{hyperref}
\usepackage[title]{appendix}

%% file: imports/caam_501_commands.tex
\usepackage{calc}
\usepackage{tikz}
\usepackage{ifthen}

\newcounter{homeworkproblem}
\newcounter{qualpointvalue}
\newcounter{warmupproblem}
\newcommand{\warmupproblem}{\ifnum\thehomeworkproblem>1 {\par} \else {}\fi\noindent\large{\textbf{Problem \thehomeworkproblem}: \ul{Warm-up Problem \thewarmupproblem}\stepcounter{homeworkproblem}\stepcounter{warmupproblem}\\}}

\newcounter{trainingproblem}
\newcommand{\trainingproblem}{\ifnum\thehomeworkproblem>1 {\par} \else {}\fi\noindent\large{\textbf{Problem \thehomeworkproblem}: \ul{Training Problem \thetrainingproblem}\stepcounter{homeworkproblem}\stepcounter{trainingproblem}\addtocounter{qualpointvalue}{-5}} 10 points\\}

\newcounter{qualifyingproblem}
\newcommand{\bigspace}{\bigskip \bigskip \bigskip \bigskip \bigskip \bigskip \bigskip \bigskip}
\newcommand{\multispace}[1]{
\if #1 0
{ }
\else
\foreach \n in {1,...,#1}{\bigspace}
\fi
}

\ifx \handoutMacro \undefined

\else

\fi

\usepackage{yfonts}

\usepackage{xparse}

\ExplSyntaxOn

\NewDocumentCommand{\instringTF}{mmmm}
 {
  \oleks_instring:nnnn { #1 } { #2 } { #3 } { #4 }
 }

\tl_new:N \l__oleks_instring_test_tl

\cs_new_protected:Nn \oleks_instring:nnnn
 {
  \tl_set:Nn \l__oleks_instring_test_tl { #1 }
  \regex_match:nnTF { \u{l__oleks_instring_test_tl} } { #2 } { #3 } { #4 }
 }

\ExplSyntaxOff

\ifx\hiddenTextFlag\undefined

\else
    \if\hiddenTextFlag1
         
    \else
        \if\hiddenTextFlag2
        
        \else
        
        \fi
    \fi
\fi

%% file: imports/common_commands.tex
\usepackage{mathtools}

\newcommand{\argmax}{\operatornamewithlimits{argmax}}

\newcommand{\reflect}{\operatornamewithlimits{refl}}

\newcommand{\R}{\mathbb{R}}

\newcommand{\set}[1]{\left \{ #1\right \}}
\newcommand{\norm}[1]{\left\lVert#1\right\rVert}
\newcommand{\abs}[1]{\left\lvert#1\right\rvert}

\newcommand{\Span}[1]{\operatornamewithlimits{span}\set{#1}}
\newcommand{\innerprod}[2]{\left \langle#1,#2\right\rangle}

\newcommand{\proj}{\operatornamewithlimits{proj}}

\newcommand{\aff}{\operatornamewithlimits{Aff}}
\newcommand{\sig}{\operatornamewithlimits{sig}}

\newcommand{\pmat}[1]{\begin{pmatrix} #1 \end{pmatrix}}
\newcommand{\paren}[1]{\left ( #1\right)}

\makeatletter
\def\mathcolor#1#{\@mathcolor{#1}}
\def\@mathcolor#1#2#3{%
  \protect\leavevmode
  \begingroup
    \color#1{#2}#3%
  \endgroup
}
\makeatother

\renewcommand{\eqref}[1]{{Eq.~\ref{#1}}}

\newcommand\restr[2]{{
  \left.\kern-\nulldelimiterspace 
  #1 
  \vphantom{
  \big|
  } 
  \right|_{#2} 
  }}
  
\makeatletter
\@ifclassloaded{beamer}{}{
    \newtheorem{theorem}{Theorem}[subsection]
    \newtheorem{lemma}[theorem]{Lemma}
    \newtheorem{corollary}[theorem]{Corollary}
    \newtheorem{proposition}[theorem]{Proposition}
    
    \theoremstyle{definition}
    \newtheorem{definition}[theorem]{Definition}
    \newtheorem{example}[theorem]{Example}

    \theoremstyle{remark}
    \newtheorem{remark}[theorem]{Remark}
}
\makeatother

\def\bszero{\boldsymbol{0}}

\makeatletter
\renewcommand{\boxed}[1]{\text{\fboxsep=.2em\fbox{\m@th$\displaystyle#1$}}}
\makeatother

\newcommand{\twopartpiecewise}[4]{\begin{cases} #1 & \text{if } #2 \\  #3 & \text{if } #4 \end{cases}}

\newcommand{\qsum}[2]{\sum^{#2}_{#1 = 1}}


%% file: imports/common_colors.tex
\definecolor{orange}{RGB}{255,127,0}
\definecolor{blackchocolate}{HTML}{191102}
\definecolor{rosewood}{HTML}{510d0a}
\definecolor{rufous}{HTML}{A31B14}
\definecolor{citron}{HTML}{a29f15}
\definecolor{orangeyellow}{HTML}{f3b61f}
\definecolor{teagreen}{HTML}{bbd8b3}
\definecolor{gray}{RGB}{128,128,128}
\definecolor{lightcitron}{RGB}{189,187,91}
\definecolor{lightrufous}{RGB}{190,95,90}
\definecolor{steelblue}{HTML}{4C86A8}
\definecolor{darkblue}{RGB}{0,0,192}
\definecolor{darkpurple}{RGB}{96,0,96}
\definecolor{darkred}{RGB}{192,0,0}

%% file: imports/matti_commands.tex
\def \bfo {\begin {eqnarray*} }
\def \efo {\end {eqnarray*} }
\def \ba {\begin {eqnarray*} }
\def \ea {\end {eqnarray*} }
\def \beq {\begin {eqnarray}}
\def \eeq {\end {eqnarray}}